\RequirePackage{fix-cm}
\documentclass[11pt]{amsart}
\usepackage{graphicx, amssymb, amsmath,amsthm}

\usepackage{enumerate}
\usepackage{graphicx}
\usepackage{geometry}
\usepackage{mathtools}
\usepackage{color}
\usepackage[mathscr]{euscript}
\usepackage{amssymb}
\usepackage{bbm}
\usepackage[colorlinks,
            linkcolor=blue,
            anchorcolor=blue,
            citecolor=blue]{hyperref}

   \def \r{\rho}

\newtheorem{theorem}{Theorem}[section]

\newtheorem{remark}{Remark}[section]

\newtheorem{lemma}[theorem]{Lemma}
\newtheorem{proposition}[theorem]{Proposition}

\newtheorem{corollary}[theorem]{Corollary}

\numberwithin{equation}{section}

\begin{document}



 \author{Rongchang Liu} \address[Rongchang Liu] {  Department of Mathematics\\
 Brigham Young University\\Provo, Utah 84602, USA}
 \email[R.~Liu]{rongchang@mathematics.byu.edu}

\author{Kening Lu} \address[Kening Lu] {  Department of Mathematics\\
 Brigham Young University\\
 Provo, Utah 84602, USA}
\email[k.~Lu]{klu@math.byu.edu}

\title[Statistical Properties of 2D stochastic Navier-Stokes Equations with time-periodic forcing and degenerate stochastic forcing]{Statistical Properties of 2D Navier-Stokes Equations with time-periodic forcing and degenerate stochastic forcing}

\pagestyle{plain}

\begin{abstract} {We consider the incompressible 2D Navier-Stokes equations with periodic boundary conditions driven by a deterministic time periodic forcing and a degenerate stochastic forcing.
We show that  the system possesses a unique ergodic periodic invariant measure which is exponentially mixing under a Wasserstein metric.
We also prove the weak law of large numbers for the continuous time inhomogeneous solution process. In addition, we obtain the weak law of large numbers and central limit theorem by restricting the inhomogeneous solution process to periodic times.
The results are independent of the strength of the noise and  hold true for any value of viscosity with a lower bound $\nu_1$ characterized by the Grashof number $G_1$ associated with the deterministic forcing. In the laminar case, there is a larger lower bound $\nu_2$ of the viscosity characterized by the Grashof number $G_2$ associated with both the deterministic and random forcing. We prove that in this laminar case, the system has  trivial dynamics  for any viscosity larger than $\nu_2$ by demonstrating the existence of a unique globally exponentially stable random periodic solution that supports the unique  periodic invariant measure.
} \end{abstract}

\maketitle

\baselineskip 14pt

\tableofcontents
\section{Introduction}
We study the statistic  properties of the incompressible 2D Navier-Stokes equations on the torus $\mathbb{T}^2=[-\pi, \pi]^2$ driven by a deterministic periodic forcing and  a degenerate stochastic forcing  using the vorticity formulation
\begin{align}\label{NS}
dw(t, x) + B(\mathcal{K} w, w) (t, x)d t = \nu\mathrm{\mathrm{\Delta}} w(t, x)dt  + f(t, x)dt + GdW(t), \quad t>s, \quad w(s) = w_0,
\end{align}
where $s \in \mathbb{R}$, $w(t, x) \in \mathbb{R}$ denotes the value of the scalar vorticity field,  $\mathrm{\mathrm{\Delta}}$ is the Laplacian operator with periodic boundary conditions, $B(\mathcal{K}w, w)= (\mathcal{K}w)\cdot\nabla w$ is the nonlinear term and $\mathcal{K}w$ is the velocity that is recovered from $w$ by the Biot-Savart integral operator $\mathcal{K}$.  The deterministic initial condition $w_0$ lies in  the state space of system \eqref{NS},  which is chosen as $H:= \left\{w \in \mathrm{L}^2\left(\mathbb{T}^{2}, \mathbb{R}\right): \int_{\mathbb{T}^2} w d x=0\right\}$, where the norm is denoted by $\left\|\cdot\right\|$ and the inner product is $\langle\cdot,\cdot\rangle$. We also define the interpolation spaces $H_{s}=\left\{w \in H^{s}\left(\mathbb{T}^{2}, \mathbb{R}\right): \int w d x=0\right\}$ and the corresponding norms $\left\|\cdot\right\|_s$ by $\|w\|_{s}=\left\|\left(-\mathrm{\Delta}\right)^{s/2} w\right\|$.
Here $W(t) $ is a standard $d$ dimensional two-sided  Brownian motion obtained as follows.  Let $W^{\pm}(t)$ be two independent standard $d$ dimensional Brownian motion, then  define
\begin{align*}
W(t) :=\left\{\begin{array}{rr}
W^{+}(t) \text{ , } t\geq 0,\\
W^{-}(-t) \text{ , } t < 0 .\\
\end{array}
\right.
\end{align*}
The sample space is denoted by  $(\mathrm{\mathrm{\Omega}} ,\mathcal{F},\mathbf{P})$, where $\mathrm{\Omega} = \left\{\omega\in C(\mathbb{R}, \mathbb{R}^d): \omega (0)=0\right\}$ endowed with the compact open topology, $\mathcal{F}$ is the Borel $\sigma$-algebra and $\mathbf{P}$  is the Wiener measure associated with the Brownian motion $W$. Denote by $\mathcal{F}_t$ the filtration of $\sigma$-algebras  generated by  $W(t)$.  The coefficient of the noise is a bounded linear operator $G: \mathbb{R}^d\rightarrow H_{\infty}:= \bigcap_{s>0} H_{s}$, such that $Ge_{i}= g_i$ , where $\{e_i\}$ is the standard basis of $\mathbb{R}^d$ and $g_i\in H_{\infty}$ for $i=1,2,\cdots,d$.  Then the noise can be expressed as $GW(t) = \sum_{i=1}^{d} g_{i}W_{i}(t)$. Also
for integer $k\geq 0$,  let $\mathcal{B}_{k}:=\sum_{i=1}^d\|g_i\|_{k}^2$ be the various norms of the energy input from the noise.

\vskip0.05in

We assume that the deterministic force $f(t, x)$ is periodic in $t$ with period $\mathcal{T}>0$ and $f\in C\left(\mathbb{R}, H_1\right)\cap \mathrm{L}_{\mathrm{loc}}^{2}\left(\mathbb{R}, H_2\right)$. The regularity condition imposed on $f$ is to ensure the spatial regularity of the solution that is needed to show the asymptotic smoothing effect of the dynamics. Under the above conditions, it is well known  that there exists a unique solution $w(t,\omega;s, w_0)$ to equation \eqref{NS}, which generates a Markov transition operator $\mathcal{P}_{s, t}: B_{b}(H)\rightarrow B_{b}(H)$ defined by
\[\mathcal{P}_{s, t}\varphi(w_0): = \mathbf{E}[\varphi(w(t;s, w_0))], \quad \varphi\in B_{b}(H),\]
where $B_{b}(H)$  is the space of bounded Borel measurable functions on $H$ with the  supremum norm. We denote the transition probabilities as $\mathcal{P}_{s, t}(w, A) : = \mathcal{P}_{s, t}\mathbb{I}_{A}(w) $ for $A\in\mathfrak{B}$, the Borel $\sigma$-algebra of $H$, where $\mathbb{I}_{A}$ is the characteristic function of $A$.
By duality, the transition operator $\mathcal{P}_{s, t}^*$ acts on the space $\mathcal{P}(H)$ of probability measures on $H$ by
\[\mathcal{P}_{s, t}^*\mu (A) = \int_{H} \mathcal{P}_{s, t}(w, A)\mu(dw), \text{ for } \mu\in \mathcal{P}(H), A\in \mathfrak{B}.\]

\vskip0.05in

 In this setting, a map $\mu_{\cdot}: t\rightarrow \mu_t$ from $\mathbb{R}$ to the space $\mathcal{P}(H)$ satisfying
\begin{align}\label{evolutionmeasures}
\int_{H} \mathcal{P}_{s, t} \varphi(w) \mu_{s}(d w)=\int_{H} \varphi(w) \mu_{t}(d w), \quad s \leq t,
\end{align}
for every $\varphi\in B_{b}(H)$ is called an {\bf invariant measure} for $\mathcal{P}_{s,t}$, which is also called an evolution system of meaures \cite{DD08} or entrance law \cite{Dyn65,Dyn71}.  It  is called  a {\bf $\mathcal{T}$-periodic invariant measure} if it is invariant and $\mu_{t+\mathcal{T}}=\mu_{t}$ for $t\in \mathbb{R}$.   $\mu_t$ is {\bf ergodic} if it is a disintegration of an ergodic invariant measure of the corresponding homogenized Markov semigroup. As it was shown in \cite{DD08}, the unique ergodicity of system \eqref{NS} is characterized by the uniqueness of the continuous $\mathcal{T}$-periodic  invariant measure $\mu_t$ for $\mathcal{P}_{s,t}$, where the continuity of $\mu_t$ is defined by endowing $\mathcal{P}(H)$ with the topology of weak convergence.

\vskip0.05in

In this paper, we show the unique ergodicity and mixing properties of the periodic invariant measure for system \eqref{NS}.  We also prove the weak law of large numbers for the time inhomogeneous solution process of system \eqref{NS} and  show the weak law of large numbers and  central limit theorem for the restriction of the time inhomogeneous solution process to periodic times. To state our main results, as in \cite{BM07, HM11}, we define recursively the following sets $\{A_k\}_{k=1}^{\infty}$ formed by the symmetrized nonlinear term $\widetilde{B}(u,w)=-B(\mathcal{K}u,w)-B(\mathcal{K}w,u)$. Set $A_1=\{g_l : 1\leq l \leq d\}$, $A_{k+1} = A_{k}\cup \{\widetilde{B}(h, g_{l}): h\in A_{k}, g_l\in A_1\}$, and $A_{\infty} = \overline{\mathrm{span}(\cup_{k\geq 1} A_{k})}$. These sets reflect how noise propagates to the state space through the nonlinear term, which plays an important role when proving the asymptotic regularizing effect of the transition operator.  For $\eta>0$ small, recall the metric $\rho$ on $H$ weighted by the Lyapunov function $e^{\eta\|w\|^2}$ as in \cite{HM08},
\begin{align}\label{rho}
\rho(w_1, w_2) = \inf_{\gamma}\int_0^1e^{\eta\|\gamma(t)\|^2}\|\dot{\gamma}(t)\|dt, \quad\forall w_1, w_2 \in H,
\end{align}
where the infimum is taken over all differentiable path $\gamma$ connecting $w_1$ and $w_2$.

This metric naturally induces a Wasserstein metric on $\mathcal{P}(H)$ by
\[\rho(\mu_1, \mu_2) = \inf_{\mu\in\mathcal{C}(\mu_1, \mu_2)}\int_{H\times H} \rho(u, v)\mu(dudv),\]
where $\mathcal{C}(\mu_1, \mu_2)$ is the set of couplings of $\mu_1, \mu_2\in \mathcal{P}(H)$.  The subset
\begin{align}\label{P1H}
\mathcal{P}_1(H) : = \{\mu\in\mathcal{P}(H): \rho(\mu,\delta_0)<\infty\}
\end{align}
 is complete under the metric $\rho$ \cite{Ch04,Vil08}, where $\delta_0$ is the Dirac measure at $0$. Denote by $C(\mathbb{R}, \mathcal{P}_1(H))$ the set of all continuous maps from $\mathbb{R}$ to $\mathcal{P}_1(H)$.
Let $c_0$  be  the constant from the Ladyzhenskaya's inequality
\[||w||_{\mathrm{L}^4(\mathbb{T}^2)}^2 \leq c_0\|w\|_1\|w\|.
\] Let $\displaystyle\|f\|_{\infty}:=\sup_{t\in\mathbb{R}}\|f(t)\|$, and define the Grashof number $ G_1 = \|f\|_{\infty}/\nu^2$ associated with the deterministic equation corresponding to system \eqref{NS}. Also let $G_2 =\sqrt{ G_1^2+\mathcal{B}_0/\nu^3}$ be the Grashof number for the whole system \eqref{NS}. These two different numbers  characterize different scales of the viscosity for which system \eqref{NS} has different dynamics.

\vskip0.05in

Our main results are summarized as the following:
\vskip0.05in

\noindent {\bf Main Theorem}. \label{mainthm}{\it

\begin{enumerate}
\item Assume $G_1<1/c_0$.
\begin{enumerate}[(i)]
\item If $A_{\infty} = H$, then the system \eqref{NS} has a unique $\mathcal{T}$-periodic invariant measure $\{\mu_t\}_{t\in\mathbb{R}} \in C(\mathbb{R}, \mathcal{P}_1(H))$, which is exponentially mixing in the sense that, under the Wasserstein metric $\rho$, any trajectory of $\mathcal{P}_{s, t}^*$ acting on $\mathcal{P}(H)$ is exponentially attracted by the periodic path $\{\mu_t\}_{t\in\mathbb{R}}$.
\item There are two cases in which $A_{\infty} = H$ can fail as in \cite{HM06}. In each case there is a unique exponentially mixing  $\mathcal{T}$-periodic invariant measure supported on $A_{\infty}$ when $A_{\infty}$ is an invariant subspace of the system \eqref{NS}.

\item In the above case (i) and  (ii),  the weak law of large numbers for the continuous time inhomogeneous solution process holds.  Also the weak law of large numbers and central limit theorem hold when restricting the inhomogeneous solution process to periodic times.

\end{enumerate}
\item If  $G_2<1/c_0$, then the unique $\mathcal{T}$-periodic invariant measure $\mu_t$ is supported on a unique random periodic solution of equation \eqref{NS}, that exponentially attracts any bounded set in $H$.  In words, the dynamics is trivial in this case. Here the random periodic solution is   periodic in a path wise manner defined by the stochastic flow (or random dynamical system) generated by the solution map of system \eqref{NS}, see \cite{FZZ11}.
\end{enumerate}

}

 A sufficient condition on the driven noise that guarantees  $A_{\infty} = H$ was characterized by Hairer and Mattingly in \cite{HM06}. In particular, the noise does not need to act on all determining modes and  can be extremely degenerate, for example it is allowed  to excite only in four directions. We stress here that part $(1)$ of the  Main Theorem does not depend on the strength of the noise and holds true for any viscosity satisfying $G_1<1/c_0$, which is purely deterministic.

 In the case when $f=0$, the unique ergodicity was proved in \cite{HM06}, and the weak law of large numbers and the central limit theorem for continuous time homogeneous solution process was shown in \cite{KW12}.  When $f$ is nonzero and independent of time,  the exponentially mixing property was shown in \cite{HM08} under the range condition $f\in\mathrm{range}(G)$. This range condition allows  one to apply the Girsanov's fromula to
 reduce the problem of the weak irreducibility for nonzero $f\in\mathrm{range}(G)$ to the case $f=0$ that has been proved in \cite{HM06}.  In the absence of the range condition,  we believe that the condition on the Grashof number is necessary to obtain our results in part $(1)$. Since even in the case that $f$ is independent of time, when the Grashof number $G_1\geq1/c_0$  is large, the dynamics on the attractor of the deterministic problem could become complicated \cite{CT79,CT81,CT83}, hence if there is no range condition, the noise is absent in certain directions on which the force $f$ acts,  so that the effects of the force $f$ may not be compensated by the noise instantly and system \eqref{NS} might fail to be uniquely ergodic.

\vskip0.05in

Comparing  the two parts of the Main Theorem, we see that the dynamics under the condition $G_1<1/c_0\leq G_2$ is uniquely ergodic and mixing but the structure of the attractor remains unknown. When $f\equiv 0$, this condition becomes $0<\nu<\nu_0$ for some $\nu_0$ determined by $c_0$ and $\mathcal{B}_0$ as above. And the comparison  becomes that of the results from \cite{HM06,HM08} and \cite{Mat99}, where the dynamics is uniquely ergodic and mixing for $0<\nu<\nu_0$ but the structure of the  attractor  is expected to be nontrivial for sufficiently small value of viscosity \cite{Fri95,HM06}. It is worthwhile to mention that in the theory of dynamical systems, there are indeed examples of uniquely ergodic system such as horocycle flow that have chaotic behavior. Therefore the unique ergodicity does not necessarily imply that the system has a trivial attractor and the dynamics of system \eqref{NS} are expected to be complicated  if $G_1<1/c_0\leq G_2$.

\vskip0.05in

In the case when the deterministic force is independent of $t$, i.e., $f(t, x)\equiv f(x)$, the solution of equation \eqref{NS} forms a time homogeneous Markov process, which generates a Markov semigroup $(\mathcal{P}_t)_{t\geq 0}$.
In this context, the statistical properties of the stochastic Navier-Stokes system have been extensively studied over the decades, see for example \cite{BKL01,BKL02,EMS01,EH01,FM95,Fer97,Hai02,KS00,KS01,Kuk02,Mat02b,MS05,MY02,Shi06} and references therein.  However, the results in the cited papers  require the random forcing to act on all determining modes. In particular, the dimension of the random forcing goes to infinity as viscosity approaches to zero. A major breakthrough was made by Hairer and Mattingly in their seminal work \cite{HM06}, where they introduced the  asymptotic strong Feller property to show unique ergodicity of the Navier-Stokes equation driven by an extremely degenerate noise. To obtain this asymptotic smoothing property, they developed a theory of infinite dimensional Malliavin calculus and established an infinite dimensional  Hörmander type theorem.  This asymptotic smoothing effect also allowed them to develop an infinite dimensional  Harris-like theorem in \cite{HM08}, to prove the exponential convergence to the unique invariant measure under the Wasserstein metric $\r$.  The result in  \cite{HM08} in turn led to a proof of the weak law of large numbers and the central limit theorem in \cite{KW12} for non-stationary Markov processes. It is notable  that the statistical results in \cite{HM06,HM08,KW12} are independent of the strength of the noise and the viscosity but require a range condition $f\in\mathrm{range}(G)$ when the time independent force $f(x)$ is nonzero. Our Main Theorem also implies the same statistical results for equation \eqref{NS} with a time independent deterministic forcing $f(x)$ without assuming the range condition, but requiring the condition  $G_1<1/c_0$ on the Grashof number.

\vskip0.05in

A natural question is then whether the similar results hold  for equation \eqref{NS} with  a nonzero deterministic force $f(t, x)$ that depends on $t$.
The setting in this case has a different flavor since the solution is a time inhomogeneous Markov process, which generates a Markov transition operator $(\mathcal{P}_{s, t})_{s\leq t}$ instead of a semigroup. And the concept of an invariant measure is replaced by a family of measures defined by equation \eqref{evolutionmeasures}. The first result in this direction was obtained by Da Prato and Debussche in \cite{DD08}. By generalizing the theorem of Doob and Khasminskii and applying a coupling argument, they proved the unique ergodicity and exponentially mixing for system \eqref{NS}, provided that the driven noise is nondegenerate. They also generalized the concept of asymptotic strong Feller property in \cite{HM06} to the time inhomogeneous setting. But the irreducibility condition proposed there is not easy to verify so that the unique ergodicity under extremely degenerate noise remains unproved. To the best of our knowledge, there is no results analogue to those in \cite{HM06,HM08,KW12} for system \eqref{NS} with an extremely degenerate noise. The present work is an attempt to give such results in the time inhomogeneous setting  for equation \eqref{NS}.

\vskip0.05in

To obtain our main results, we take an approach that is different from \cite{DD08}. Our approach is mainly inspired by the Harris-like theorem for infinite dimensional systems developed in \cite{HM08}, the theory of non-autonomous dynamical systems \cite{CS08,CV02a,CV02b} especially on the structure of attractors for evolution equations, and the limit theorems for additive functionals of non-stationary homogeneous Markov processes \cite{KW12}.

\vskip0.05in

To be specific, we first show that  under the Wasserstein metric $\rho$, the transition operator $\mathcal{P}_{s, t}^*$ has a strong contraction property uniform in the initial time $s$ when acting on $\mathcal{P}(H)$. Then, we characterize $\mathcal{P}_{s, t}^*$ as a cocycle over the hull $H(f)$ of the periodic force $f(t, x)$ equipped with the Bebutov shift flow.  Due to the Lyapunov structure, we only know that the cocycle is continuous in $h\in H(f)$  when acting on those probability measures under which the Lyapunov function is integrable. Therefore we choose to work with the closed subsets $\mathcal{P}_{R}$ of $\mathcal{P}_{1}(H)$ consisting of those measures for which the integral of the Lyapunov function is uniformly bounded by $R>0$.  Then by a pull back procedure we associate the cocycle with a semigroup $S^t$ acting on $C(H(f), \mathcal{P}_{R})$, the space of continuous mappings from $H(f)$ to $\mathcal{P}_{R}$, whose fixed point gives the desired $\mathcal{T}$-periodic invariant measure. However, for fixed $R>0$, the semigroup $S^t$ is not necessarily contracting on $C(H(f), \mathcal{P}_{R})$ for all $t>0$. As $t$ becomes smaller, we need larger $R$ to ensure that $S^t$ is contracting on $C(H(f), \mathcal{P}_{R})$.  Hence we choose to apply the fixed point theorem  on a family of nested closed subsets instead of the whole space $C(H(f), \mathcal{P}_1(H))$, and show the existence of a common  periodic invariant measure in all of the nested closed subsets.  The uniqueness and exponentially mixing then follows simultaneously from the contraction of $\mathcal{P}_{s, t}^*$ on $\mathcal{P}(H)$ and the invariance of the periodic invariant measure.

\vskip0.05in

The limit theorems are proved  after one shows the strong contraction of $\mathcal{P}_{s, t}^*$.  Here we first prove a weak law of large numbers for the continuous time inhomogeneous solution process with arbitrary initial condition. However it is unclear whether a central limit theorem holds or not for the continuous time inhomogeneous solution process. Fortunately the idea of the Poincare section behind the periodic structure enables us to show the weak law of large numbers and the central limit theorem by considering the restriction of  the solution process to periodic times. In particular, the proof for the central limit theorem follows from a modified martingale approximation approach from \cite{KW12}, where the idea dates back to the early works of  Gordin \cite{Gor69}, Kipnis and Varadhan \cite{KV86}.

\vskip0.05in

The proof for the contraction of $\mathcal{P}_{s, t}^*$ on $\mathcal{P}(H)$ involves several uniform in initial time  estimates that reflect the contraction of the dynamics at different scales.  Among those estimates, one is on the weak irreducibility that was introduced in \cite{HM08}, and here we adapt  it to our time inhomogeneous and periodic setting.
The weak irreducibility is a property about the existence of an accessible point $v$ in the state space such that for any neighborhood of $v$, there is a positive probability for the solutions starting from different initial points to enter into that neighborhood at some large time. For system \eqref{NS}, when $f(t, x) = 0$, the accessible point is the global stable equilibrium point of the corresponding deterministic system \cite{HM06}.  If  the deterministic force $f(t,x)\equiv f(x)$ is nonzero and time independent,  then one can reduce the problem to the case when $f(x) = 0$ as long as $f(x)$ satisfies the range condition \cite{HM08}. When the nonzero deterministic forcing $f(t, x)$ depends on $t$, the range condition is indeed much more restrictive, hence we do not require the range condition on the periodic force $f(t, x)$, but the trade off is the constraint $G_1<1/c_0$ on the Grashof number.  Under this constraint, there is a unique global stable  periodic solution \cite{CV02b}  of the deterministic equation corresponding  to \eqref{NS}. This allows us to show the weak irreducibility for the time inhomogeneous system \eqref{NS} since the stable  periodic solution is accessible in this case.
As a byproduct, we show that this type of irreducibility condition, together with the asymptotic strong Feller property proposed in \cite{DD08} that is verified in our work, gives an applicable Doob and Khasminskii type argument that shows system \eqref{NS} can only have at most one periodic invariant measure.
Another two estimates are the Lyapunov structure  that is used to compensate the non-uniformity of the dynamics of equation \eqref{NS} and the gradient inequality that reveals the asymptotic regularizing effect of the transition operator $\mathcal{P}_{s, t}$.
Since the estimates about the solution of equation \eqref{NS} are uniform in the initial time, hence by carrying over the proof from \cite{HM08} and \cite{HM11}, we obtain uniform in the initial time estimates for the Lyapunov structure and gradients inequality.


\vskip0.05in

In applications, it is more convenient to represent the mixing result by the action of transition operator on observables. In the time homogeneous case as in \cite{HM08}, the exponentially mixing that involves observables is  given by proving the quasi-equivalence of a particular norm $\|\cdot\|_{\eta}$ on the observables with an appropriate norm $\|\cdot\|_{\rho, \mu}$ on the space of Lipschitz functions on $H$ under the metric $\rho$, where the later norm depends on the invariant measure $\mu$. While in the present work, the periodic invariant measure $\{\mu_t\}_{t\in\mathbb{R}}$ depends on time. Hence we choose a family of norms $\{\|\cdot\|_{\rho, \mu_s}\}_{s\in\mathbb{R}}$ on the space of Lipschitz functions that depends on $\rho$ and the values of the periodic invariant measure at initial times. The mixing involving observables in our time inhomogeneous case is then obtained by showing the quasi-equivalence of $\|\cdot\|_{\eta}$ with $\|\cdot\|_{\rho, \mu_s}$ under the transition operator $\mathcal{P}_{s, t}$ for all $s\in \mathbb{R}$. Here the dependence on initial time of the norm  can be regarded as a property adapted to the time inhomogeneity, to yield a uniform contraction under the transition operator.

\vskip0.05in

The proof for the second part of our Main Theorem  is based on the approach from \cite{Mat99} of constructing the unique attracting stationary solution, where the author used the contraction property of the dynamics on the whole state space when the viscosity is large. We will use the contraction property  as well, which is the case if $G_2<1/c_0$. The main difference is that we need to incorporate the periodic structure here to construct the unique random periodic solution.  Note when $f(t, x) = 0$,  the second part of our theorem becomes the result proved by Mattingly in \cite{Mat99}.

\section{Preliminaries and main results}\label{setting}
Under the conditions given in the introduction, the existence and uniqueness of the solution to equation \eqref{NS}  is well known, see for example \cite{BM07} and references therein. To be specific,  for any initial time $s\in \mathbb{R}, T> s$ and any $w_0\in H$, equation \eqref{NS} has a unique strong solution $w(t, \omega; s, w_0)$, adapted to the filtration $\mathcal{F}_{t}$, which generates a stochastic flow $\mathrm{\Phi}(t, \omega; s, \cdot): H\rightarrow H$ such that $\mathrm{\Phi}(t, \omega; s, w_0) = w(t, \omega; s, w_0)$ for $s\leq t, w_0\in H$ and
\[w\in C\left([s,T]; H\right)\cap C\left((s,T];H_3\right),\,\, \mathbf{P}\text{-}\mathrm{a.s}..\]
Here by a stochastic flow $\mathrm{\Phi}(t, \omega; s, w_0)$, we mean that it is a modification of the solution of equation \eqref{NS} satisfying the following conditions:
\begin{enumerate}
\item It is adapted to $\mathcal{F}_t$ and for almost all $\omega$, $\mathrm{\Phi}(t, \omega; s, w_0)$ is continuous in $(t, s, w_0)$ and $\mathrm{\Phi}(s, \omega; s, w_0) = w_0.$
\item For almost all $\omega$,
\[\mathrm{\Phi}(t+\tau, \omega; s, w_0) = \mathrm{\Phi}(t+\tau, \omega; t ,\mathrm{\Phi}(t, \omega; s, w_0)),\]
for $s\leq t$, $\tau>0$ and $w_0\in H$.
\end{enumerate}
Throughout the work, we will also use $w_{s, t}(\omega, w_0)$ or $\Phi_{s, t}(\omega, w_0)$ to represent the solution $w(t, \omega; s, w_0)$.

\vskip0.05in

A random periodic solution of period $\mathcal{T}$ for system \eqref{NS} is a stochastic process  $w^*:\mathbb{R}\times\mathrm{\Omega}\rightarrow H$ that is progressively measurable and satisfies
\begin{align}\label{randomperiodic}
w^*(t+\mathcal{T}, \omega) = w^*(t, \theta_{\mathcal{T}}\omega)  \text{  and  } \mathrm{\Phi}(\tau, \omega; t, w^*(t, \omega)) = w^*(\tau, \omega),
\end{align}
for all $t\in\mathbb{R}$, $\tau\geq t$ and almost every $\omega$. Here $\theta: \mathbb{R}\times\mathrm{\Omega}\rightarrow\mathrm{\Omega}$ given by
\[\theta_{t} \omega(\cdot): =\omega(t+\cdot)-\omega(t)\]
is the Wiener shift.
\subsection{Periodic invariant measures}
By the well-posedness of equation \eqref{NS}, one can define the associated Feller transition operator $\mathcal{P}_{s,t}\left(s\leq t\right)$ by
\begin{align*}
\mathcal{P}_{s, t} \varphi(w_0)=\mathbf{E}[\varphi(w(t; s, w_0))], \quad \varphi \in B_{b}(H).
\end{align*}
It was shown in \cite{DD08} that this transition operator is a $\mathcal{T}$-periodic evolution operator, i.e.,  $\mathcal{P}_{s, t}=\mathcal{P}_{s, r} \mathcal{P}_{r, t}$  and $\mathcal{P}_{s, t} = \mathcal{P}_{s+\mathcal{T}, t+\mathcal{T}}$ for all $s \leq r \leq t$.  Following \cite{DR06,DD08}, it is natural to consider the invariant measure  for $\mathcal{P}_{s, t}$ as a map $\mathbb{R} \ni t \mapsto \mu_{t} \in \mathcal{P}(H)$, such that for every $\varphi \in B_{b}(H)$,
\begin{align*}
\int_{H} \mathcal{P}_{s, t} \varphi (w) \mu_{s}(d w)=\int_{H} \varphi(w) \mu_{t}(d w), \quad s \leq t,
\end{align*}
or equivalently
\begin{align*}
\mathcal{P}_{s,t}^*\mu_s = \mu_t, \quad s \leq t.
\end{align*}

Given $\mathcal{T}>0$, it is called $\mathcal{T}$-periodic if
\[\mu_{t+\mathcal{T}}=\mu_{t}, \quad t \in \mathbb{R}.\]

\subsection{Main results}
In this subsection, we formulate the main results of the present paper in details.  Recall that $\mathcal{P}(H)$ is the space of probability measures on $H$. For $\eta>0, 0<r\leq1$, define the metric $\rho_r$ on $H$ weighted by the Lyapunov function $V(w) = e^{\eta\|w\|^2}$ as in \cite{HM08},
\begin{align}\label{familymetrics}
\rho_r(w_1, w_2) = \inf_{\gamma}\int_0^1V^r(\gamma(t))\|\dot{\gamma}(t)\|dt, \quad\forall w_1, w_2 \in H,
\end{align}
where the infimum is taken over all differentiable path $\gamma$ in $H$ connecting $w_1$ and $w_2$. The family of metrics $\rho_r$ will be used in Section \ref{sectionergodicmixing}. When $r=1$ we have $\rho_r =\rho$ as given in \eqref{rho}. Denote by $\mathrm{Lip}_{\rho}(H)$ the space of Lipschitz functions on $H$ endowed with the metric $\rho$ and by $\mathrm{Lip}_{\rho}(\phi)$ the Lipschitz constant for $\phi\in \mathrm{Lip}_{\rho}(H)$.  The metric naturally induces a Wasserstein metric on $\mathcal{P}(H)$ by
\[\rho(\mu_1, \mu_2) = \inf_{\mu\in\mathcal{C}(\mu_1, \mu_2)}\int_{H\times H} \rho(u, v)\mu(dudv),\]
where $\mathcal{C}(\mu_1, \mu_2)$ is the set of couplings of $\mu_1, \mu_2\in \mathcal{P}(H)$.
The reason for working with the Wasserstein metric is that the transition probabilities in infinite dimensional systems are likely to be mutually singular, especially when the strong Feller property does not hold \cite{HM08}. Hence the convergence to the invariant measure often fails under the total variation metric and one would like to replace it by a weaker Wasserstein metric.
For $\eta>0$, let $C_{\eta}^{1}: = \{\phi\in C^1(H): \|\phi\|_{\eta}<\infty\}$, where
\begin{align}\label{coneeta}
\|\phi\|_{\eta}: =\sup_{w\in H}e^{-\eta\|w\|^2}\big(|\phi(w)|+ \|\nabla\phi(w)\|\big).
\end{align}
Recall from the introduction that the set $A_{\infty}$ is defined by setting $A_1=\{g_l : 1\leq l \leq d\}$, $A_{k+1} = A_{k}\cup \{\widetilde{B}(h, g_{l}): h\in A_{k}, g_l\in A_1\}$, and $A_{\infty} = \overline{\mathrm{span}(\cup_{k\geq 1} A_{k})}$, where $\widetilde{B}(u,w)=-B(\mathcal{K}u,w)-B(\mathcal{K}w,u)$ is the the symmetrized nonlinear term.

\vskip0.05in

The following Theorem \ref{ergodicmixing}-\ref{thmlimittheorem} contain the  results for  the case when $G_1<1/c_0$ and $A_{\infty} = H$, where system \eqref{NS} is uniquely ergodic and mixing on the whole space. The case when $A_{\infty} = H$ fails is formulated in Theorem \ref{subspace}, where the periodic invariant measure is unique and mixing when restricting system \eqref{NS} on the proper subspace $A_{\infty}$.  The laminar case  $G_2<1/c_0$  where system \eqref{NS} has trivial dynamics is given in Theorem \ref{thmperiodicattraction}.  Note that the metric $\rho$ weighted by the Lyapunov function $V(w) = e^{\eta\|w\|^2}$ depends on the parameter $\eta>0$.

\vskip0.05in

The first result is the following unique ergodicity and exponentially mixing of the periodic invariant measure under the Wasserstein metric.
\begin{theorem}\label{ergodicmixing}
(Ergodicity and mixing)
Assume $G_1<1/c_0$, $A_{\infty} = H$. Then  there is a constant $\eta_0>0$, such that for every $\eta\in (0, \eta_0]$, there exist constants $C, \gamma>0$ independent of initial time $s$, and a unique $\mathcal{T}$-periodic path $\{\mu_{t}\}_{t\in \mathbb{R}}\in C(\mathbb{R}, \mathcal{P}_1(H))$ such that $\mathcal{P}_{s,t}^{*}\mu_s = \mu_t$ and
\begin{align}\label{mixing-PH}
\rho(\mathcal{P}_{s, s + t}^{*}\mu, \mu_{s+t})\leq Ce^{-\gamma t}\rho(\mu, \mu_s),\quad \forall s\in\mathbb{R},  t\geq 0, \mu\in\mathcal{P}(H).
\end{align}
In an equivalent form that involves the transition operator acting on observables, we have for every $\phi\in C_{\eta}^{1}$,
\begin{align}\label{mixing-observables2}
\|\mathcal{P}_{s, s+t}\phi - \int_{H}\phi(w)\mu_{s+t}(dw)\|_{\eta}\leq C e^{-\gamma t}\|\phi\|_{\eta}.
\end{align}
\end{theorem}
Theorem \ref{ergodicmixing} will be proved in Section \ref{sectionergodicmixing}, by combining  a fixed point argument with the following contraction property proved in the same section.
\begin{theorem}\label{contractiontransition}
$($Contraction on $\mathcal{P}(H)$$)$ Assume $G_1<1/c_0$, $A_{\infty} = H$. Then  there exists $\eta_0>0$ such that for $\eta\in(0,\eta_0]$, there are positive constants $C$ and $\gamma$ such that
\begin{align}\label{eqcontraction}
\rho(\mathcal{P}_{s, s+t}^{*}\mu_1,\mathcal{P}_{s, s+t}^{*}\mu_2)\leq Ce^{-\gamma t}\rho(\mu_1,\mu_2),
\end{align}
for every $s\in \mathbb{R}$, $t\geq 0$, and any $\mu_1, \mu_2\in \mathcal{P}(H)$.
\end{theorem}

The second result is the weak law of large numbers  (WLLN) and the central limit theorem (CLT) for the solution process. The proof will be given in Section \ref{sectionlimittheorem}. Note that the restriction of  the solution process starting at $(s,w_0)\in\mathbb{R}\times H$ to periodic times $\{s+n\mathcal{T}\}_{n\geq 0}$ forms a time homogeneous Markov chain  $\{w_{s, s+n\mathcal{T}}(w_0)\}_{n\geq 0}$.
\begin{theorem}\label{thmlimittheorem}
Assume $G_1<1/c_0$, $A_{\infty} = H$. Then we have the following:

\vskip0.05in

1. (WLLN for the time inhomogeneous solution process) For any $s\in\mathbb{R}$, $w_0\in H$,
\[\lim_{T\rightarrow\infty}\frac{1}{T}\int_{0}^T\psi(w_{s, s+t}(w_0))dt = \int_{\mathbb{S}_{\mathcal{T}}}\int_{H}\psi(w)\mu_s(dw)\lambda(ds)\]
in probability.  Here  $\mathbb{S}_{\mathcal{T}} = \mathbb{R}/\mathcal{T}\mathbb{Z}$, and $\lambda$  is the normalized Lebesgue measure on $\mathbb{S}_{\mathcal{T}}$.

\vskip0.05in

2. (WLLN for the Markov chain) For any $s\in\mathbb{R}, w_0 \in H$,
\[\lim_{N\rightarrow\infty}\frac{1}{N}\sum_{k=0}^{N-1}\psi\left(w_{s, s+k\mathcal{T}}(w_0)\right) = \int_{H}\psi(w)\mu_s(dw)\]
in probability.

\vskip0.05in

3. (CLT for the Markov chain) Let $\widetilde{\psi}_s(w): = \psi(w) -  \int_{H}\psi(w)\mu_s(dw)$. Then for any $s\in\mathbb{R}, w_0 \in H$,
\begin{align*}
\lim_{N\rightarrow\infty}\mathbf{P}\left(\frac{1}{\sqrt{N}}\sum_{k=0}^{N-1}\widetilde{\psi}_s\left(w_{s, s+k\mathcal{T}}(w_0)\right)<\xi\right) = \Phi_{\sigma}(\xi),\quad \forall \xi\in\mathbb{R},
\end{align*}
where $\Phi_{\sigma}(\cdot)$ is the distribution function of a normal random variable with variance $\sigma^2$, and
\[\sigma^2 = \lim_{N\rightarrow\infty}\frac1N\mathbf{E}\left[\sum_{k=0}^{N-1}\widetilde{\psi}_s\left(w_{s, s+k\mathcal{T}}(w_0)\right)\right]^2,\]
where $\sigma: = \sigma(s)$ is $\mathcal{T}$-periodic in $s$.
\end{theorem}
The third result shows that if the viscosity is larger than that in the previous two results, then the dynamics of the system \eqref{NS} is actually trivial. The following theorem will be proved in Section \ref{sectiontrivialattractor}.

\begin{theorem}\label{thmperiodicattraction}
 If  $G_2<1/c_0$, then there exists a random $\mathcal{T}$-periodic solution $w^*(t, x)$ of equation \eqref{NS} in the sense of  the definition given by equation \eqref{randomperiodic}. Moreover, $w^*(t, x)$ exponentially attracts all other solutions both in forward and pullback times. The law of $w^*(t, x)$ gives the unique $\mathcal{T}$-periodic invariant measure.
\end{theorem}
We make remarks on the above results.

\vskip0.05in

First of all, the driven noise here can be extremely degenerate as in \cite{HM06}, where Hairer and Mattingly  give an algebraic characterization  of the condition $A_{\infty}=H$. To state this  characterization  we recall the following real Fourier basis of $H$ consisting of eigenvectors of the Laplacian as in \cite{HM06,MP06}.  Set  $\mathbb{Z}_{+}^{2}=\left\{\left(k_{1}, k_{2}\right) \in \mathbb{Z}^{2} \,\middle\vert\, k_{2}>0\right\} \cup\left\{\left(k_{1}, 0\right) \in \mathbb{Z}^{2} \,\middle\vert\, k_{1}>0\right\}$, $\mathbb{Z}_{-}^{2}= - \mathbb{Z}_{+}^{2}$ and $\mathbb{Z}_0^{2}=\mathbb{Z}_{+}^{2} \cup \mathbb{Z}_{-}^{2}$. And define for $k\in \mathbb{Z}_0^{2}$,
\begin{align*}
\gamma_{k}(x)=\left\{\begin{array}{ll}
\sin (k \cdot x) & \text { if } k \in \mathbb{Z}_{+}^{2}, \\
\cos (k \cdot x) & \text { if } k \in \mathbb{Z}_{-}^{2}.
\end{array}\right.
\end{align*}
Let $\mathcal{Z}_0: = \{k_l: 1\leq l\leq d_0\}\subset \mathbb{Z}_0^{2}$, where $d_0\in\mathbb{N}$. The following theorem was proved in \cite{HM06}:
 \begin{theorem}\cite{HM06}
Let $\mathcal{Z}_{0}$ satisfy  the following conditions:\\
A1. There exist at least two elements in $\mathcal{Z}_{0}$ with different Euclidean norms.\\
A2. Integer linear combinations of elements of $\mathcal{Z}_{0}$ generate $\mathbb{Z}^{2}$.\\
If $\{\gamma_k: k\in \mathcal{Z}_0\}\subset \mathrm{span}\{g_i: 1\leq i\leq d\}$, then $A_{\infty}=H$.
\end{theorem}
In particular, the random forcing is allowed to excite only on four directions. For example, one can take $\mathcal{Z}_{0}=\{(1,0),(-1,0),(1,1),(-1,-1)\}$.

\vskip0.05in

The case when $A_{\infty} \neq H$  is best characterized by $\mathcal{Z}_0$ as in \cite{HM06}. To do so, we set $Q_1=\{\gamma_l : l\in\mathcal{Z}_0\}$, $Q_{j+1} = Q_{j}\cup \{\widetilde{B}(h, \gamma_{l}): h\in Q_{j}, \gamma_l\in Q_1\}$ and $Q_{\infty} =\overline{ \mathrm{span}(\cup_{j\geq 1} Q_{j})}$. If the action of the noise is diagonal, then $A_{\infty}$ coincides with $Q_{\infty}$. However the noise here does not necessarily act on system \eqref{NS} in a diagonal way,  therefore we assume $A_{\infty} = Q_{\infty}$ in order to characterize $A_{\infty}$ by $\mathcal{Z}_0$ and ensure the invariance of $A_{\infty}$. Then there are two cases in which $A_{\infty} = H$ can fail \cite{HM06}:
\begin{enumerate}
\item \textbf{Case 1}: The elements of $\mathcal{Z}_{0}$ are all collinear or of the same Euclidean length. Then  $A_{\infty}$ is the finite-dimensional space spanned by $\left\{\gamma_{k} \,\middle\vert\, k \in \mathcal{Z}_{0}\right\}$.
\item \textbf{Case 2}:  Let $\mathcal{G}$ be the smallest subgroup of $\mathbb{Z}^{2}$ containing $\mathcal{Z}_{0}$. Then  $A_{\infty}$ is the space spanned by $\left\{\gamma_{k} \,\middle\vert\, k \in \mathcal{G} \backslash\{(0,0)\}\right\} .$ Let $k_{1}, k_{2}$ be two generators for $\mathcal{G}$ and define $v_{i}=2 \pi k_{i} /\left|k_{i}\right|^{2},$ then $A_{\infty}$ is the space of functions that are periodic with respect to the translations $v_{1}$ and $v_{2}$.
\end{enumerate}

In this situation, we have the following theorem, which is a counterpart of Theorem 2.3 from \cite{HM06} in our time inhomogeneous setting.
\begin{theorem}\label{subspace}
Assume $G_1<1/c_0$, $A_{\infty} = Q_{\infty}$ and $f(t)\in A_{\infty}$ for $t\in[0,\mathcal{T}]$. Then in each of the above two cases,  $A_{\infty}$ is a closed proper linear subspace of $H$ which is invariant under system \eqref{NS}. When restricting the dynamics of equation \eqref{NS} on $A_{\infty}$,  there is a unique $T$-periodic invariant measure supported on $A_{\infty}$, for which Theorem \ref{ergodicmixing}-\ref{thmperiodicattraction} hold.

\vskip0.05in

In particular, in \textbf{Case 1} the dynamics restricted to  $A_{\infty}$ is that of a periodic Ornstein-Uhlenbeck process \cite{DL07}.
\end{theorem}
\begin{remark}
The condition $A_{\infty} = Q_{\infty}$ enables us to characterize $A_{\infty}$ by $\mathcal{Z}_{0}$. Moreover, by rewriting equation \eqref{NS} in Fourier space \cite{HM06}, and combining with the assumption $f(t)\in A_{\infty}$ for $t\in[0,\mathcal{T}]$, the condition $A_{\infty} = Q_{\infty}$ also ensures the invariance of $A_{\infty}$ under the dynamics of  system \eqref{NS}. By considering equation \eqref{NS} on the invariant subspace $A_{\infty}$,  Theorem \ref{subspace} follows from the same arguments for proving corresponding  results in the case $A_{\infty}=H$. If $A_{\infty} \neq Q_{\infty}$, then in general we do not know if $A_{\infty}$ is invariant and there is no results similar to Theorem \ref{subspace}.
\end{remark}

\vskip0.05in

Secondly, once we have established the contraction \eqref{contractiontransition} of the transition operator $\mathcal{P}_{s, t}^*$,  then the existence, uniqueness and mixing of the periodic invariant measure follows simultaneously by a fixed point argument in Theorem \ref{fixedpoint}. In the time inhomogeneous setting, this is different from the approach in \cite{DD08}. The authors there showed the existence by first applying the Krylov-Bogoliubov  procedure to the associated homogenized Markov semigroup, and then obtained the  periodic invariant measure  as a continuous version of the  disintegration of a invariant measure for the homogenized Markov semigroup. They also generalized the concept of asymptotic strong Feller property to the time inhomogeneous case and established a Doob and Khasminskii type argument to show the uniqueness of periodic invariant measure. Unfortunately the irreducibile type condition proposed there is not easy to verify for system \eqref{NS}. As a byproduct of the present work, we verified the  asymptotic strong Feller property in Corollary \ref{ASF} from subsection \ref{gradientinequality}, and by introducing a type of weak irreducibility in subsection \ref{weakLya}, we provide a Doob and Khasminskii type argument to show that system \eqref{NS} has at most one periodic invariant measure, see Corollary \ref{DoobKha}.

\vskip0.05in

Lastly, we prove both WLLN and CLT for the corresponding time homogeneous Markov chain, but only WLLN for the time inhomogeneous solution process. It seems not easy to prove or even formulate a CLT for the continuous time inhomogeneous solution process. We leave this as a future work.

\section{Contraction on $\mathcal{P}(H)$: ergodicity and mixing}\label{sectionergodicmixing}

In this section, we will prove Theorem \ref{ergodicmixing} and Theorem \ref{contractiontransition}. More precisely,  we first show Theorem \ref{ergodicmixing} in subsection \ref{section-ergodicmixing} by assuming the contraction result of Theorem \ref{contractiontransition}.  The proof is accomplished by viewing the transition operator as a non-autonomous dynamical system on $\mathcal{P}(H)$ and employing a fixed point argument. This gives the unique ergodicity and mixing of the Navier-Stokes system \eqref{NS}.

\vskip0.05in

The proof for the contraction \eqref{contractiontransition} in Theorem \ref{contractiontransition} will be given in subsection \ref{section-contraction}, after we establish the necessary estimates in our time inhomogeneous setting.  In particular, we will prove the weak irreducibility in \ref{weakLya}.

\subsection{Ergodicity and mixing by contraction}\label{section-ergodicmixing}
In this subsection, we will prove Theorem \ref{ergodicmixing} by assuming the contraction \eqref{contractiontransition} of Theorem \ref{contractiontransition}.  In the following Theorem \ref{fixedpoint}, we show the existence of a unique $\mathcal{T}$-periodic invariant measure that has the mixing property \eqref{mixing-PH}. The mixing \eqref{mixing-observables2} in terms of the action on observables is given in Theorem \ref{mixingob} below.

\vskip0.05in

In the time homogeneous case \cite{HM08}, if the Markov semigroup is a contraction on  $\mathcal{P}(H)$, the existence of a unique  invariant measure follows directly by applying the fixed point theorem on the complete subset $\mathcal{P}_{1}(H)$ defined as \eqref{P1H}.  However, this is not an obvious task when considering the time inhomogeneous Markov transition operator $\mathcal{P}_{s, t}^*$. Inspired by the work in \cite{CS08} and \cite{CV02a} on non-autonomous dynamical systems,  we consider the transition operator as a cocycle  over the hull $H(f)$ of the periodic force $f$ equipped with the Bebutov shift flow. Then we associate the cocycle with a  semigroup acting on a family of nested closed subsets of $C\left(H(f), \mathcal{P}_{1}(H)\right)$, whose fixed point is the periodic invariant measure for $\mathcal{P}_{s, t}$. As we mentioned in the introduction, due to the Lyapunov structure, the contraction of the semigroup is not uniform on the whole complete space $C\left(H(f), \mathcal{P}_{1}(H)\right)$.  Hence we choose to apply the fixed point theorem  on a family of nested closed subsets instead of the whole space $C(H(f), \mathcal{P}_1(H))$ to adapte to this non-uniformity.
 In what follows we will elaborate the ideas and give the proof for Theorem \ref{fixedpoint}.
\vskip0.05in

We first state a result which shows that the system \eqref{NS} has a uniform in the initial time Lyapunov structure , which was first introduced in \cite{HM08} in a time homogeneous setting. Since our estimates about the solution of equation \eqref{NS} is uniform in the initial time, the proof of the Lyapunov structure in \cite{HM08} carries over to the time inhomogeneous setting here. For the sake of completeness, we give the proof in Appendix \ref{subsectionLyapunovstructure}.
\begin{proposition}[The Lyapunov Structure]\label{pre_Lya}
Let $V(w) = \exp(\eta\|w\|)$. Then there exists $\eta_0>0$ such that for any $\eta\in(0,\eta_0]$, we have the following\\
1. There are $\kappa> 1$ and $C>0$ such that $\|w\|V(w)\leq CV^{\kappa}(w) $ for all $w\in H$. \\
2. There is a constant $C>0$ such that for any fixed $0< r_0 <1$
\begin{align}\label{pre_eqLya01}
\mathbf{E}V^r(\Phi_{s,s+t}(w))(1+\|\nabla\Phi_{s, s+t}(w))\xi\|)\leq CV^{r\alpha(t)}(w)
\end{align}
for every $w, \xi\in H$ with $\|\xi\|=1$, $r\in[r_0, 2\kappa]$ and every $t\in [0,\mathcal{T}], s\in \mathbb{R}$, where $\alpha(t) = e^{-\frac{\nu}{4}t}.$ \\
3. There is a constant $C$ (possibly larger) such that
\begin{align} \label{pre_eqLya02}
\mathbf{E}V^r(\Phi_{s,s+t}(w))\leq C V^{r\alpha(t)}(w),
\end{align}
for every $t\geq 0$ and $r\in [r_0, 2\kappa]$. Here $\alpha(t)= \alpha(\mathcal{T})^{k}\alpha(\beta)$ is an extension to $\mathbb{R}$ of the function $\alpha(t)$ in \eqref{pre_eqLya01},  where $k\in \mathbb{N}$ and $\beta\in [0,\mathcal{T})$ are unique numbers such that $t = k\mathcal{T}+\beta$.
\end{proposition}
We view the stochastic flow generated by equation \eqref{NS} as a family of systems over the hull $H(f): = \mathrm{cl}\{f(s+\cdot )|s\in\mathbb{R}\}$ of the time dependent force $f$, where the closure is taking with respect to the uniform convergence topology on $C_b(\mathbb{R}, H)$, the space of bounded continuous maps. For each fixed $u\in H(f)$, there is a stochastic flow $\Phi_{s, t, u}(w_0)$ obtained by solving the equation
\begin{align*}
dw(t, x) + B(\mathcal{K} w, w) (t, x)d t = \nu\mathrm{\mathrm{\Delta}} w(t, x)dt  + u(t, x)dt + GdW(t), \quad t>s, \quad w(s) = w_0,
\end{align*}
 which in turn induces a transition operator $\mathcal{P}_{s, t, u}\phi(w_0):=\mathbf{E}[\phi(\Phi_{s, t, u}(w_0))]$ indexed by $u$. Here we consider the case when $f$ is $\mathcal{T}$-periodic, so $H(f)$ can be identified with the circle $\mathbb{S}_{\mathcal{T}} = \mathbb{R}/\mathcal{T}\mathbb{Z} = [0,\mathcal{T})$ of length $\mathcal{T}$, where each $y\in \mathbb{S}_{\mathcal{T}}$ is identified with $f(y+\cdot)$. We denote by $\sigma$ the rotation over the circle $\mathbb{S}_{\mathcal{T}}$ induced from the Bebutov shift on the hull $H(f)$. It is defined  by $\sigma(t)y = t+y \bmod\mathcal{T}$ for $y\in \mathbb{S}_{\mathcal{T}}$ and $t\in \mathbb{R}$. Note that for $y=0$, the transition operator is $\mathcal{P}_{s, t}$ corresponding to $f$.
 \begin{theorem}\label{fixedpoint}
Assume $G_1<1/c_0$, $A_{\infty} = H$. Then there is a constant $\eta_0>0$, such that for every $\eta\in (0, \eta_0]$, there is a unique $\mathcal{T}$-periodic path $\{\mu_{t}\}_{t\in \mathbb{R}}\in C(\mathbb{R}, \mathcal{P}_1(H))$ that is invariant, i.e.,  $\mathcal{P}_{s,t}^{*}\mu_s = \mu_t$ and  there exist constant $C, \gamma>0$ such that
\begin{align}\label{fixedpointmixing}
\rho(\mathcal{P}_{s,t}^{*}\nu, \mu_t)\leq Ce^{-\gamma(t-s)}\rho(\nu, \mu_s),
\end{align}
for any $s\leq t$ and $\nu\in\mathcal{P}(H)$. Also $\int_H \exp\left(2\kappa\eta\|w\|^2\right)\mu_s(dw)\leq C$ for all $s\in\mathbb{R}$, where $\kappa\geq 1$  is a constant.
\end{theorem}
\begin{proof}
We first show that the transition operator satisfies the following  translation identity, which is a conclusion by the uniqueness of solution:
\begin{align}\label{thm3.1translationid}
\mathcal{P}_{\tau+h, t+h, y} = \mathcal{P}_{\tau, t, \sigma(h) y}, \quad \forall \tau\leq t, h\in\mathbb{R}, y\in\mathbb{S}_{\mathcal{T}}.
\end{align}
Noting $Y_t: = \Phi_{t+h, \tau+h, y}(\omega, w_0)$ solves the equation
\begin{align*}
Y_t &= w_0 + \int_{\tau+h}^{t+h}\left(\nu\Delta Y_{r-h} - B(\mathcal{K}Y_{r-h}, Y_{r-h})\right)dr + \int_{\tau+h}^{t+h}f(y+r)dr + G(W(t+h, \omega) - W(\tau+h, \omega))\\
&= w_0 + \int_{\tau}^{t}\left(\nu\Delta Y_{r} - B(\mathcal{K}Y_{r}, Y_{r})\right)dr + \int_{\tau}^{t}f(y+r+h)dr + G(W(t, \theta_h\omega) - W(\tau, \theta_h\omega)).
\end{align*}
And $w_t = \Phi_{t, \tau,  \sigma(h)y}(\omega, w_0) $ solves the equation
\begin{align*}
w_t = w_0 + \int_{\tau}^t \left(\nu\Delta w_r - B(\mathcal{K}w_r, w_r)\right)dr +\int_{\tau}^t f(y+r+h)dr + G(W(t, \omega) - W(\tau, \omega)).
\end{align*}
It then follows from the uniqueness of solution that
\begin{align}\label{translationidentity}
\Phi_{t+h, \tau+h, y}(\omega, w_0) = \Phi_{t, \tau,  \sigma(h)y}(\theta_{h}\omega,w_0).
\end{align}
The fact that $\mathbf{P}$ is invariant under the Wiener shift $\theta_t$ implies that
\[\mathcal{P}_{\tau+h, t+h, y} = \mathcal{P}_{\tau, t, \sigma(h) y},\]
which is the desired translation identity.

\vskip0.05in

Let $R>0$ and $g(w) = V^{2\kappa}(w) = \exp\left({2\kappa\eta}\|w\|^2\right)$, where the Lyapunov function $V(w) = e^{\eta\|w\|^2}$ and $\kappa, \eta$ are from Proposition \ref{pre_Lya}. To apply the fixed point theorem, we first consider a family of  closed subsets $\mathcal{P}_R$ of  $\mathcal{P}_1(H)$ indexed by $R>0$, which is defined by
\[\mathcal{P}_R: =\{\mu\in\mathcal{P}(H): \int_{H}g(w)\mu(dw)\leq R\}, \quad R>0.\]
For each fixed $R$, $\mathcal{P}_R$ is indeed a closed subset of the complete space $\mathcal{P}_1(H)$ defined as in \eqref{P1H}. For any $\mu\in\mathcal{P}_R$, one has
\[\rho(\mu, \delta_0) = \int_H \rho(w, 0)\mu(dw)\leq \int_H \|w\|e^{\eta\|w\|^2}\mu(dw)\leq C \int_Hg(w)\mu(dw)<\infty,\]
so that $\mu\in\mathcal{P}_1(H)$. Let $\mu_n$ be a Cauchy sequence in $\mathcal{P}_R$ under the metric $\rho$. Then there is a unique $\mu\in\mathcal{P}_1(H)$ such that $\rho(\mu_n, \mu)\rightarrow 0 $ as $n\rightarrow\infty$. So $\mu_n$ converges to $\mu$ weakly. For $N>0$,  let $g_N(w) = \min\{g(w), N\}$, then $g_N \in C_b(H)$, and
\[\int_H g_N(w)\mu(dw) = \lim_{n\rightarrow\infty}\int_H g_N(w)\mu_n(dw)\leq \lim_{n\rightarrow\infty}\int_H g(w)\mu_n(dw)\leq R.\]
Hence by the monotone convergence theorem, one has $\int_H g(w)\mu(dw) \leq R$.  Therefore $\mu\in \mathcal{P}_R$, which shows that $\mathcal{P}_R$ is closed.

\vskip0.05in

Next we show that for any $t_0>0$, there is some $R_{t_0}>0$ such that for any $t\geq t_0$, $\mathcal{P}_{s, s+t}^*$ maps $\mathcal{P}_{R_{t_0}}$ to itself.  It follows from Proposition \ref{pre_Lya} that  for $\mu\in \mathcal{P}_R$,
\begin{align*}
\int_H g_N(w)\mathcal{P}_{s, s+t}^*\mu(dw)&\leq \int_H\mathcal{P}_{s, s+t} g(w)\mu(dw)\\
 & = \int_H \mathbf{E} g(\mathrm{\Phi}_{s, s+t}(w))\mu(dw)\leq C \int_H g^{\alpha(t)}(w)\mu(dw)\\
&\leq C \left( \int_H g(w)\mu(dw)\right)^{\alpha(t)}\leq CR^{\alpha(t_0)},
\end{align*}
where we used Jensen's inequality in the penultimate step. Hence
\begin{align}\label{nonuniformabsorbing}
\int_H g(w)\mathcal{P}_{s, s+t}^*\mu(dw) \leq CR^{\alpha(t_0)}
\end{align}
 again by monotone convergence theorem. In words, for any $R>0$, and $t\geq t_0$, the transition operator $\mathcal{P}_{s, s+t}^*$ maps  $\mathcal{P}_{R}$ to $\mathcal{P}_{CR^{\alpha(t_0)}}$.  If we choose $R = R_{t_0} := C^{\frac{1}{1-\alpha(t_0)}}$, then $CR_{t_0}^{\alpha(t_0)} = R_{t_0}$, therefore $\mathcal{P}_{s, s+t}^*$ maps $\mathcal{P}_{R_{t_0}}$ to itself.

\vskip0.05in

Now we show that for fixed $s\in \mathbb{R}, t>0$ and $\mu$ satisfying $\int_H g(w)\mu(dw)<\infty$, $\mathcal{P}_{s, s+t, y}^*\mu$ is continuous in $y$. Let $f_1, f_2\in H(f)$ and $w_{s, s+t}^{f_i}(w)$ be the corresponding unique solution of equation \eqref{NS} with $f$ replaced by $f_i$, then by the bound \eqref{continuousonhull}, there is constant $\eta_0>0$ so that for $\eta\in (0, \eta_0]$, there are constants $C, \gamma>0$ such that
\begin{align}\label{continuous01}
\mathbf{E}\|w_{s, s+t}^{f_1}(w)-w_{s, s+t}^{f_2}(w)\|^2\leq C e^{\gamma t}e^{\eta\|w\|^2}\|f_1-f_2\|_{\infty}^2.
\end{align}
From the definition of $\rho$ as in \eqref{rho}, one has
\begin{align}\label{continuous02}
\r(w_1, w_2)\leq \|w_1-w_2\|\left(e^{\eta\|w_1\|^2}+e^{\eta\|w_2\|^2} \right), \quad \forall w_1,w_2\in H.
\end{align}
It is known \cite{Ch04,Vil08} that for any $\mu_1, \mu_2\in \mathcal{P}(H)$,
\begin{align}\label{continuous03}
\rho(\mu_1, \mu_2) = \inf\mathbf{E}\rho(X_1, X_2),
\end{align}
where the infimum is taken over all couplings $(X_1, X_2)$ for $(\mu_1, \mu_2)$.
Combining \eqref{continuous01}-\eqref{continuous03} with bounds \eqref{eq: est1}, it follows that
\begin{align*}
&\rho(\mathcal{P}_{s, s+t, f_1}^*\delta_{w},\mathcal{P}_{s, s+t, f_2}^*\delta_{w} )\leq \mathbf{E}\rho(w_{s, s+t}^{f_1}(w), w_{s, s+t}^{f_2}(w))\\
&\leq \left(\mathbf{E}\|w_{s, s+t}^{f_1}(w)-w_{s, s+t}^{f_2}(w)\|^2\right)^{\frac12}\left(2\mathbf{E}\left[\exp({2\eta w_{s, s+t}^{f_1}(w)})+\exp({2\eta w_{s, s+t}^{f_2}(w)})\right]\right)^{\frac12}\\
&\leq C e^{\gamma t}g(w)\|f_1-f_2\|_{\infty}.
\end{align*}
Therefore by the Markov property,
\begin{align}\label{continuousiny}
\rho(\mathcal{P}_{s, s+t, f_1}^*\mu,\mathcal{P}_{s, s+t, f_2}^*\mu)\leq Ce^{\gamma t}\int_Hg(w)\mu(dw)\|f_1-f_2\|_{\infty}.
\end{align}
Since $\int_H g(w)\mu(dw)<\infty$, we see that $\mathcal{P}_{s, s+t, y}^*\mu$ is continuous in $y$.
However, it is not clear if $\mathcal{P}_{s, s+t, y}^*\mu$ is continuous in $y$ for any fixed $\mu\in \mathcal{P}_{1}(H)$.
In addition, in view of \eqref{nonuniformabsorbing}, we know that for fixed $R>0$, the transition operator $\mathcal{P}_{s, s+t}^*$ is not necessarily contracting on $\mathcal{P}_{R}$ for all $t>0$.  Hence we choose to work with the family of closed subsets $\left\{\mathcal{P}_{R}\right\}_{R>0}$.

\vskip0.05in

Observe that by Theorem \ref{contractiontransition}, for any $s\in\mathbb{R}$ and $t\geq 0$, $\mathcal{P}_{s, s+t}^*$ maps $\mathcal{P}_1(H)$  to itself. So by the translation identity \eqref{thm3.1translationid}, $\mathcal{P}_{s, s+t, y}^*$ maps $\mathcal{P}_1(H)$ to itself since for any $y\in \mathbb{S}_{\mathcal{T}}$ there is some $s_y\in\mathbb{R}$ such that $y = \sigma(s_y)0$. Denote for convenience
\[\varphi:  \mathbb{R}_{+}\times\mathcal{P}_1(H)\times \mathbb{S}_{\mathcal{T}}\rightarrow \mathcal{P}_1(H), \text{ by }\varphi(t, \mu, y) = \mathcal{P}_{0, t, y}^{*}\mu.\]
By the contraction property in Theorem \ref{contractiontransition} and translation identity \eqref{thm3.1translationid},  we have that for any $R>0$ and $\mu\in \mathcal{P}_{R}$,  $\varphi$ is continuous in $\mu$, uniformly with respect to $y$. And by inequality \eqref{continuousiny},  it is continuous in $y$ uniformly for $\mu$. Hence $\varphi$ is joint continuous in $(\mu, y)\in\mathcal{P}_{R}\times\mathbb{S}_{\mathcal{T}}$. Moreover, $\varphi$ has the cocycle property over the base dynamical system $(\mathbb{S}_{\mathcal{T}}, \mathbb{R}, \sigma)$ since for all $\tau, t\geq 0 $ and $y\in \mathbb{S}_{\mathcal{T}}$, $\mu\in \mathcal{P}_1(H)$,
\[\mathcal{P}_{0, t+\tau, y}^{*}\mu= \mathcal{P}_{t, t+\tau, y}^{*}\mathcal{P}_{0, t, y}^{*}\mu = \mathcal{P}_{0, \tau,\sigma(t)y}^{*}\mathcal{P}_{0, t, y}^{*}\mu.\]
Inspired by  \cite{CS08},  we consider the set of all continuous sections $C(\mathbb{S}_{\mathcal{T}}, \mathcal{P}_1(H))$, on which we define the metric
\[p(\gamma_1, \gamma_2) := \max_{y\in{\mathbb{S}_{\mathcal{T}}}}\rho(\gamma_1(y), \gamma_2(y)),\]
for $\gamma_1,\gamma_2 \in C(\mathbb{S}_{\mathcal{T}}, \mathcal{P}_1(H))$, which is complete under this metric since $(\mathcal{P}_1(H),\rho)$ is complete.

\vskip0.05in

 Fix $t_0\in(0,1)$. For $t\geq t_0$ and $\gamma\in C(\mathbb{S}_{\mathcal{T}}, \mathcal{P}_{R_{t_0}})$, define $S^t : C(\mathbb{S}_{\mathcal{T}}, \mathcal{P}_{R_{t_0}}) \rightarrow C(\mathbb{S}_{\mathcal{T}}, \mathcal{P}_{R_{t_0}}) $ by
\[S^t(\gamma)(y) = \varphi(t, \gamma(\sigma(-t)y), \sigma(-t)y)).\]
Since $\varphi$  maps $\mathcal{P}_{R_{t_0}}$ to itself,  $S^t$ is well defined for $t\geq t_0$.
The cocycle property of $\varphi$ in turn implies that $S^{t_1}S^{t_2} = S^{t_1+t_2}$ for $t_1, t_2\geq t_0$. Indeed, if we denote $\gamma'=S^{t_2}(\gamma)$, then one has
\begin{align*}
S^{t_1}S^{t_2}\gamma(y) &= \varphi(t_1, \gamma'(\sigma(-t_1)y),\sigma(-t_1)y)\\
& = \varphi(t_1, \varphi(t_2, \gamma(\sigma(-t_1-t_2)y), \sigma(-t_1-t_2)y)),\sigma(t_2)\sigma(-t_1-t_2)y)\\
&= \varphi(t_1+t_2, \gamma(\sigma(-t_1-t_2)y), \sigma(-t_1-t_2)y) = S^{t_1+t_2}\gamma(y),
\end{align*}
where in the penultimate step we used the cocycle property.

\vskip0.05in

Now note that by Theorem \ref{contractiontransition} and translation identity \eqref{thm3.1translationid}, one has
\begin{align*}
p(S^{t}\gamma_1,S^{t}\gamma_2) &= \max_{y\in \mathbb{S}_{\mathcal{T}}}\rho(\varphi(t, \gamma_1(\sigma(-t)y),\sigma(-t)y),\varphi(t, \gamma_2(\sigma(-t)y),\sigma(-t)y))\\
&= \max_{y\in \mathbb{S}_{\mathcal{T}}}\rho(\mathcal{P}_{0,t,y}^{*}\gamma_1(\sigma(-t)y), \mathcal{P}_{0,t,y}^{*}\gamma_2(\sigma(-t)y))\\
&\leq Ce^{-\gamma t}\max_{y\in \mathbb{S}_{\mathcal{T}}}\rho(\gamma_1(\sigma(-t)y),\gamma_2(\sigma(-t)y)) = Ce^{-\gamma t}p(\gamma_1,\gamma_2).
\end{align*}
Therefore for large $T>t_0$, $p(S^{T}\gamma_1,S^{T}\gamma_2)\leq c  p(\gamma_1,\gamma_2)$ for some $c\in(0,1)$. Fix such a $T$, then $S^T$ is a contraction over the complete metric space $C(\mathbb{S}_{\mathcal{T}}, \mathcal{P}_{R_{t_0}})$, so there is a unique fixed point  $\Gamma_{t_0}\in C(\mathbb{S}_{\mathcal{T}}, \mathcal{P}_{R_{t_0}})$ of $S^T$.  Noting for any $t\geq t_0$, $S^t$ maps $C(\mathbb{S}_{\mathcal{T}}, \mathcal{P}_{R_{t_0}})$ to itself, hence
\[S^T(S^t\Gamma_{t_0}) = S^t(S^T\Gamma_{t_0}) = S^t(\Gamma_{t_0})\]
implies that $S^t(\Gamma_{t_0}) = \Gamma_{t_0}$ by uniqueness of the fixed point, which shows that $\Gamma_{t_0}$ is a fixed point of $S^t$ for $t\geq t_0$.  For $0<t_1\leq t_0$, one has $R_{t_1}\geq R_{t_0}$, so $\mathcal{P}_{R_{t_0}}\subset \mathcal{P}_{R_{t_1}}$. And for the same $T>0$, $S^T$ is a contraction on $C(\mathbb{S}_{\mathcal{T}}, \mathcal{P}_{R_{t_1}})$, which has a unique fixed point $\Gamma_{t_1}$. By uniqueness, $\Gamma_{t_1} = \Gamma_{t_0}$, hence $\Gamma_{t_0}$ is also a fixed point of $S^t$ for $t\geq t_1$. Since $t_1$ is arbitrary, we see that $\Gamma: = \Gamma_{t_0}$ is a fixed point of $S^t$ for $t\geq 0$, that is, $\varphi(t, \Gamma(\sigma(-t)y),\sigma(-t)y) = \Gamma(y)$ for all $y\in Y$. Replacing $y$ with $\sigma(t)y$ we have $\varphi(t, \Gamma(y),y) = \Gamma(\sigma(t)y)$ which by definition is
\[\mathcal{P}_{0, t, y}^{*}\Gamma(y)=\Gamma(\sigma(t)y).\]
Replacing $y$ by $\sigma(s)0$, and using the translation identity \eqref{translationidentity}, we have $\mathcal{P}_{s, s+t}^{*}\Gamma(\sigma(s)0) = \Gamma(\sigma(s+t)0)$ for $s\in\mathbb{R}$ and $t\geq 0$. Hence $\mu_s:=\Gamma(\sigma(s)0) $ is a periodic invariant measure for the transition operator $\mathcal{P}_{s,t}$.
If there is another periodic invariant measure $\widetilde{\mu}_s$, then by Theorem \ref{contractiontransition}, we have that for any $s\in\mathbb{R}, n\in\mathbb{N}$,
\begin{align*}
\r(\mu_s, \widetilde{\mu}_s) = \r(\mu_{s+n\mathcal{T}}, \widetilde{\mu}_{s+n\mathcal{T}}) =\r(\mathcal{P}_{s, s+n\mathcal{T}}^{*}\mu_s, \mathcal{P}_{s, s+n\mathcal{T}}^{*}\widetilde{\mu}_s)\leq Ce^{-\gamma n\mathcal{T}}\r(\mu_s, \widetilde{\mu}_s).
\end{align*}
Letting $n\rightarrow\infty$, we see that $\r(\mu_s, \widetilde{\mu}_s) = 0$ and hence $\mu_s = \widetilde{\mu}_s$ for all $s\in\mathbb{R}$. So the periodic invariant measure is unique.
The exponentially contraction \eqref{fixedpointmixing} under the Wasserstein metric $\rho$ then follows from Theorem \ref{contractiontransition} and the invariance of $\mu_s$.
\end{proof}


The following result is equivalent to  Theorem \ref{fixedpoint} but represented by the action of the transition operator on Lipschitz observables. From the unique periodic invariant measure, one can define a family of norms $\|\cdot\|_{\rho,s}$ by
\begin{align}\label{normr1}
\|\phi\|_{\rho, s} = \mathrm{Lip}_{\rho}(\phi)+\left|\mu_s(\phi)\right|, \quad s\in \mathbb{R},
\end{align}
where  $\mathrm{Lip}_{\rho}(\phi) = \sup_{u\neq v}\frac{|\phi(u)-\phi(v)|}{\rho(u, v)}$ is the Lipschitz constant and  $\mu(\phi) = \int_{H}\phi(u)\mu(du)$ for $\mu\in\mathcal{P}(H)$ and $\phi\in\mathrm{Lip}_{\rho}(H)$. Unlike in \cite{HM08}, the norm here depends on the initial time to adapt to the time inhomogeneity.
\begin{corollary}\label{Lipmixing}
Assume $G_1<1/c_0$, $A_{\infty} = H$,  there is a constant $\eta_0>0$, such that for every $\eta\in (0, \eta_0]$. Then there exist constants $C, \gamma>0$ such that
\[\|\mathcal{P}_{s, t}\phi-\mu_t(\phi)\|_{\rho, s}\leq Ce^{-\gamma(t-s)}\|\phi-\mu_s(\phi)\|_{\rho, s},\]
for every Fréchet differentiable $\phi: H\rightarrow \mathbb{R}$ and $s\leq t$.
\end{corollary}
\begin{proof}
By the Monge-Kantorovich daulity \cite{Ch04,Vil08} and the contraction on $\mathcal{P}(H)$ from Theorem \ref{contractiontransition}, we have
\begin{align*}
|\mathcal{P}_{s, t}\phi(u)-\mathcal{P}_{s, t}\phi(v)|&\leq \mathrm{Lip}_{\rho}(\phi)\sup_{\mathrm{Lip}_{\rho}(\varphi)\leq 1}\left|\int_{H}\varphi(z)\mathcal{P}_{s, t}^{*}\delta_{u}(dz)-\int_{H}\varphi(z)\mathcal{P}_{s, t}^{*}\delta_{v}(dz)\right|\\
& = \mathrm{Lip}_{\rho}(\phi) \rho(\mathcal{P}_{s, t}^{*}\delta_{u}, \mathcal{P}_{s, t}^{*}\delta_{v})\leq \mathrm{Lip}_{\rho}(\phi) Ce^{-\gamma(t-s)}\rho(u,v),
\end{align*}
for any $u, v\in H$. By the invariance of the periodic invariant measure from Theorem \ref{fixedpoint}, $\int_{H}(\mathcal{P}_{s, t}\phi-\mu_t(\phi))\mu_s(du) = 0$.  Therefore
\[\|\mathcal{P}_{s, t}\phi-\mu_t(\phi)\|_{\rho, s}=\mathrm{Lip}_{\rho}(\mathcal{P}_{s, t}\phi)\leq \mathrm{Lip}_{\rho}(\phi) Ce^{-\gamma(t-s)} =Ce^{-\gamma(t-s)}\|\phi-\mu_s(\phi)\|_{\rho, s}. \]
\end{proof}
In the following we will prove the convergence \eqref{mixing-observables2}, i.e., the following
\begin{theorem}\label{mixingob}
Assume $G_1<1/c_0$, $A_{\infty} = H$. Then  there is a constant $\eta_0>0$, such that for every $\eta\in (0, \eta_0]$, there exist constants $C, \gamma>0$ independent of initial time $s$, such that for every $\phi\in C_{\eta}^{1}$ as in \eqref{coneeta},
\begin{align}
\|\mathcal{P}_{s, s+t}\phi - \int_{H}\phi(w)\mu_{s+t}(dw)\|_{\eta}\leq C e^{-\gamma t}\|\phi\|_{\eta},
\end{align}
for any $s\in\mathbb{R}, t\geq 0$.
\end{theorem}
The proof of the theorem will be given at the end of this section by combining  Corollary \ref{Lipmixing} and the quasi-equivalence of $\|\cdot\|_{\rho, s}$ with $\|\cdot\|_{\eta}$ that will be proved below.

\vskip0.05in

In the time homogeneous case as in \cite{HM08}, the exponentially mixing in an equivalent form that involves observables similar to \eqref{mixing-observables2} is  given by proving the quasi-equivalence under the Markov semigroup of $\|\cdot\|_{\eta}$  with an appropriate norm $\|\cdot\|_{\rho}$ on $\mathrm{Lip}_{\rho}(H)$. The norm  $\|\cdot\|_{\rho}$  is actually a combination of $\mathrm{Lip}_{\rho}(\phi)$ for $\phi\in \mathrm{Lip}_{\rho}(H)$ and the integral of $\phi$ with respect to the unique invariant measure.  In the present time inhomogeneous  setting, the periodic invariant measure depends on time. Therefore, to show the mixing property \eqref{mixing-observables2}, we choose the norm $\|\cdot\|_{\rho, s}$ on  $\mathrm{Lip}_{\rho}(H)$ defined as above. This is natural since $\|\cdot\|_{\rho, s}$ is quasi-equivalent to $\|\cdot\|_{\eta}$ under the transition operator $\mathcal{P}_{s, t}^{}$ and $\mathcal{P}_{s, t}^{}$ has a similar contraction property as in the time homogeneous case proved in Theorem 4.3 \cite{HM08}. The dependence on initial time of the norm  $\|\cdot\|_{\rho, s}$ can be regarded as a property that adapts to the time inhomogeneity, to yield a uniform contraction under the action of the  transition operator, see Theorem \ref{adaptedcontraction} below.

\vskip0.05in

To begin with, we first define a family of auxiliary norms for $r\in [0, 1]$. The first  involves the Lipschitz constant in terms of the metric $\rho_r$ given in \eqref{familymetrics}. Define
\[\|\phi\|_{\rho_r, s}: = \mathrm{Lip}_{\rho_r}(\phi)+\left|\mu_s(\phi)\right|, \quad s\in \mathbb{R},\]
where $\mathrm{Lip}_{\rho_r}(\phi) = \sup_{u\neq v}\frac{|\phi(u)-\phi(v)|}{\rho_r(u, v)}$. When $r=1$, it is the norm given as \eqref{normr1}. The second one is a norm weighted by the  Lyapunov function $V(w) = e^{\eta\|w\|}$, which was introduced in \cite{HM08}.
\begin{align}\label{normVr}
\|\phi\|_{V^r}: =\sup_{w\in H} \frac{|\phi(w)|+ \|\nabla\phi(w)\|}{V^r(w)}.
\end{align}
Note when $r=1$, $\|\cdot\|_{V^r} = \|\cdot\|_{\eta}$. We first show that $\|\cdot\|_{\rho_r, s}$ can be bounded by $\|\cdot\|_{V^r}$ from both sides with different values of $r$.
\begin{proposition}\label{boundedbothsides}
There is a constant $C>0$ such that
\begin{align}\label{eqboundbothsides}
C^{-1}\|\phi\|_{V^{\kappa r}}\leq \|\phi\|_{\rho_r,s}\leq C\|\phi\|_{V^r},
\end{align}
for $r\in [r_0, 1]$, $s\in \mathbb{R}$ and $\phi\in C^1(H)$, where the constants $0<r_0<1$ and $\kappa>1$ are taken from the Lyapunov structure in Proportion \ref{pre_Lya}.
\end{proposition}
Before giving the proof, we need a lemma that connects the norm $\|\cdot\|_{\rho_r,s}$ with the derivative part of $\|\cdot\|_{V^r}$.
\begin{lemma}\label{lemmaderivative}
For every $\phi\in C^1(H)$, we have
\[\|\phi\|_{\rho_r,s} = \sup_{w\in H} \frac{\|\nabla\phi(w)\|}{V^r(w)} + \left|\mu_s(\phi)\right|, \quad s\in \mathbb{R}. \]
\end{lemma}
\begin{proof}
We first claim that  for $v\in H$,
\[\lim_{\varepsilon\rightarrow 0}\sup_{u: \|u-v\|<\varepsilon}\frac{|\phi(u)-\phi(v)|}{\rho_r(u, v)} = \frac{\|\nabla\phi(v)\|}{V^r(v)}.\]
By definition of $\r_r$ as in \eqref{familymetrics},  \[\rho_r(u, v)\leq \int_{0}^{1}V^r\big((1-\tau)u+\tau v\big)\|u-v\|d\tau,\]
hence we have
\[\sup_{u: \|u-v\|<\varepsilon}\frac{|\phi(u)-\phi(v)|}{\rho_r(u, v)} \geq \sup_{u: \|u-v\|<\varepsilon}\frac{|\phi(u)-\phi(v)|}{\|u-v\|}\left( \sup_{u: \|u-v\|<\varepsilon}\int_{0}^{1}V^r\big((1-\tau)u+\tau v\big)d\tau\right)^{-1}. \]
Therefore by taking limit,
\begin{align}\label{LHS>RHS}
\lim_{\varepsilon\rightarrow 0}\sup_{u: \|u-v\|<\varepsilon}\frac{|\phi(u)-\phi(v)|}{\rho_r(u, v)} \geq \frac{\|\nabla\phi(v)\|}{V^r(v)}.
\end{align}
Next we prove the reverse inequality of \eqref{LHS>RHS}. For fixed $v\in H$ and any $u$ satisfying $\|u-v\|<\varepsilon$, let $R>0$ large such that $u, v \in B_R(0)$, the ball in $(H, \|\cdot\|)$ with radius $R$ centered at $0$. For any $w_1, w_2\in B_R(0) $, one has the equivalence of metrics
\[\|w_1-w_2\|\leq \rho_r(w_1, w_2)\leq V(R)\|w_1-w_2\|.\]
Let $K = V(R)$. Then for any $\delta>0$, there exists a differentiable path $\gamma$ connecting $u, v$ such that
\[\rho_r(u, v)\leq \int_{0}^1V^r(\gamma(\tau))\|\dot{\gamma}(\tau)\|d\tau \leq \rho_r(u, v)+ \delta\leq K\varepsilon+\delta.\]
Now for any $t\in [0,1]$,
\[\|\gamma(t) -v\| =\left \|\int_{0}^t\dot{\gamma}(\tau)d\tau\right\|\leq \int_{0}^1V^r(\gamma(\tau))\|\dot{\gamma}(\tau)\|d\tau\leq K\varepsilon +\delta,\]
which means that $\gamma(t)$ never leaves the ball of radius $K\varepsilon + \delta$ centered at $v$. Therefore,
\[\rho_r(u,v)\geq \int_0^1 V^r(\gamma(\tau))\|\dot{\gamma}(\tau)\|d\tau -\delta\geq \inf_{w:\|w-v\|\leq K\varepsilon +\delta}V^r(w)\|u-v\|-\delta.\]
Hence  by taking $\delta = \varepsilon\|u-v\|$ above, we have
\begin{align*}
\frac{|\phi(u)-\phi(v)|}{\rho_r(u, v)}\leq \frac{|\phi(u)-\phi(v)|}{\|u-v\|}\left(\inf_{w:\|w-v\|\leq (K+\varepsilon)\varepsilon}V^r(w)-\varepsilon\right) ^{-1}.
\end{align*}
By taking limit we have
\[\lim_{\varepsilon\rightarrow 0}\sup_{u: \|u-v\|<\varepsilon}\frac{|\phi(u)-\phi(v)|}{\rho_r(u, v)} \leq \frac{\|\nabla\phi(v)\|}{V^r(v)},\]
which finishes the proof of the claim.

\vskip0.05in

It then follows from the claim that
\begin{align}\label{eqLlessH}
\mathrm{Lip}_{\rho_r}(\phi) = \sup_{u\neq v}\frac{|\phi(u)-\phi(v)|}{\rho_r(u, v)}\geq \sup_{w\in H}\frac{ \|\nabla\phi(w)\|}{V^r(w)}.
\end{align}
Hence $\|\phi\|_{\rho_r,s} \geq \sup_{w\in H} \frac{\|\nabla\phi(w)\|}{V^r(w)} + \left|\mu_s(\phi)\right|$.  It remains to show the reverse inequality of \eqref{eqLlessH}. Without loss of generality we can assume $\phi(0)=0$ and $\mathrm{Lip}_{\rho_r}(\phi) =1$. There is nothing to show if for some $w$, $1\leq  \frac{\|\nabla\phi(w)\|}{V^r(w)}$. So we assume $\|\nabla\phi(w)\| \leq V^r(w)$ for all $w$.  Then for any $w_1, w_2\in H$,
\[|\phi(w_1)-\phi(w_2)| = \int_0^1\left\langle \nabla\phi(\gamma(\tau)), \dot{\gamma}(\tau)\right\rangle d\tau\leq \sup_{w\in H}\frac{\|\nabla\phi(w)\|}{V^r(w)}\int_0^1V^r(\gamma(\tau))\|\dot{\gamma}(\tau)\| d\tau. \]
By taking infimum over all differentiable $\gamma$ connecting $w_1, w_2$, we have
\[\frac{|\phi(w_1)-\phi(w_2)| }{\rho_r(w_1, w_2)}\leq \sup_{w\in H}\frac{\|\nabla\phi(w)\|}{V^r(w)}. \]
The proof is then complete.
\end{proof}
\begin{proof}[Proof of Proposition \ref{boundedbothsides}]
By Theorem \ref{fixedpoint}, there is a constant $C>0$, independent of the initial time $s$ such that
\begin{align}\label{mu_s_phi}
|\mu_s(\phi)| \leq \int_{H}V^r(w)\frac{|\phi(w)|}{V^r(w)}\mu_s(dw)\leq \|\phi\|_{V^r}\mu_s(V^r)\leq C\|\phi\|_{V^r}.
\end{align}
Combining \eqref{mu_s_phi} with Lemma \ref{lemmaderivative} and definition of the norm $\|\cdot\|_{r}$ as \eqref{normVr}, we have
\[\|\phi\|_{\rho_r,s}\leq \sup_{w\in H} \frac{\|\nabla\phi(w)\|}{V^r(w)}+ C\|\phi\|_{V^r} \leq \widetilde{C}\|\phi\|_{V^r}.\]
To show the first inequality of \eqref{eqboundbothsides}, we fix $\phi$ with $\|\phi\|_{\rho_r,s} =1.$ Then by Proposition \ref{verifyLyapunov},
\[|\phi(w) - \phi(0)|\leq \mathrm{Lip}_{\rho_r}(\phi)\rho_r(w, 0) \leq \int_0^1V^r(\tau w)\|w\|d\tau\leq \|w\|V^r(w)\leq CV^{\kappa r}(w). \]
Also by Theorem \ref{fixedpoint} and noting that $r\leq 1$,  we have
\[\int_{H}\rho_r(w,0)\mu_s(dw)\leq\int_{H}\rho(w,0)\mu_s(dw)\leq \int_{H} C V^{\kappa}(w)\mu_s(dw)\leq \widetilde{C}.\]

Hence
\begin{align*}
|\phi(0)|&\leq \left|\int_{H} \phi(w) \mu_s(dw) - \phi(0)\right| + \left|\int_{H} \phi(w) \mu_s(dw) \right|\\
&\leq \int_{H}|\phi(w)-\phi(0) |\mu_s(dw)+\left|\int_{H} \phi(w) \mu_s(dw) \right| \leq \widetilde{C}+1,
\end{align*}
where $\left|\int_{H} \phi(w) \mu_s(dw) \right|\leq 1$ since $\|\phi\|_{\rho_r,s} =1.$ It then follows that
\[|\phi(w)|\leq |\phi(0)|+|\phi(w) - \phi(0)| \leq \widetilde{C} V^{\kappa r}(w).\]
Note that $\|\phi\|_{\rho_r,s} =1$ also implies that $\sup_{w\in H} \frac{\|\nabla\phi(w)\|}{V^r(w)}\leq 1$, therefore
\[\|\phi\|_{V^{\kappa r}}\leq \sup_{w\in H} \frac{\|\nabla\phi(w)\|}{V^{\kappa r}(w)} + \sup_{w\in H} \frac{|\phi(w)|}{V^{\kappa r}(w)}\leq 1+\widetilde{C}\leq C\|\phi\|_{\rho_{r},s}.\]
The proof is complete.
\end{proof}

The following result shows that the transition operator has a contraction property under the norms $\|\cdot\|_{V^r}$  and $\|\cdot\|_{\rho_r, s}$.
\begin{theorem}\label{adaptedcontraction}
There are constants $C, \gamma>0$ such that
\begin{align}\label{eqcontraction1}
\|\mathcal{P}_{s, s+t}\phi\|_{V^{r(t)}}\leq Ce^{\gamma t}\|\phi\|_{V^r},\\
\|\mathcal{P}_{s, s+t}\phi\|_{\rho_{r(t),s}}\leq Ce^{\gamma t}\|\phi\|_{\rho_{r, s+t}},\label{eqcontraction2}
\end{align}
for $r\in[r_0, 2\kappa]$, $s\in \mathbb{R}, t>0$ and $\phi\in C^{1}(H)$, where $r(t) = \max\{r\alpha(t), r_0\}$ and $\alpha(t) $ is from Proposition \ref{verifyLyapunov} .
\end{theorem}
\begin{proof}
We first prove the inequalities for $t\in[0, \mathcal{T}]$. By Proposition \ref{pre_Lya}, we have
\begin{align}\notag
\|\nabla \mathcal{P}_{s, s+t}\phi(w)\|&\leq \mathbf{E}\|\nabla \phi(\Phi_{s, s+t}(w))\|\|\nabla\Phi_{s,s+t}(w)\|\\\notag
&\leq \sup_{w\in H} \frac{\|\nabla\phi(w)\|}{V^{ r}(w)}\mathbf{E}V^r(\Phi_{s, s+t}(w))\|\nabla\Phi_{s,s+t}(w)\|\\\label{adaptedcontraction-1}
& \leq C\sup_{w\in H} \frac{\|\nabla\phi(w)\|}{V^{ r}(w)}V^{r\alpha(t)}(w)\leq C\|\phi\|_{V^r} V^{r\alpha(t)}(w).
\end{align}
It follows from the penultimate step that
\begin{align}\label{adaptedcontraction1}
\sup_{w\in H}\frac{\|\nabla \mathcal{P}_{s, s+t}\phi(w)\|}{V^{r\alpha(t)}(w)}\leq C \sup_{w\in H} \frac{\|\nabla\phi(w)\|}{V^{ r}(w)}.
\end{align}
Also
\begin{align}\label{adaptedcontraction-3}
|\mathcal{P}_{s,s+t}\phi(w)|\leq \mathbf{E}|\phi|(\Phi_{s,s+t}(w))\leq \sup_{w\in H} \frac{|\phi(w)|}{V^{ r}(w)}\mathbf{E}V^r(\Phi_{s, s+t}(w))\leq C\|\phi\|_{V^r} V^{r\alpha(t)}(w).
\end{align}
Combining the above two estimates \eqref{adaptedcontraction-1} and \eqref{adaptedcontraction-3}, we see that $\|\mathcal{P}_{s, s+t}\phi\|_{V^{r(t)}}\leq C\|\phi\|_{V^r}$ with $r(t) = r\alpha(t)$. Lemma \ref{lemmaderivative} and the invariance of the unique periodic invariant measure, together with inequality \eqref{adaptedcontraction1} imply that
\begin{align*}
\|\mathcal{P}_{s, s+t}\phi\|_{\rho_{r(t)}, s}& = \sup_{w\in H} \frac{\|\nabla\mathcal{P}_{s, s+t}\phi(w)\|}{V^{r(t)}(w)} + \left|\int_{H}\mathcal{P}_{s, s+t}\phi(w)\mu_s(dw)\right|\\
&\leq C \sup_{w\in H} \frac{\|\nabla\phi(w)\|}{V^{ r}(w)} + \left|\int_{H}\phi(w)\mu_{s+t}(dw)\right|\leq C\|\phi\|_{\rho_r, s+t}, \quad t\in [0, \mathcal{T}].
\end{align*}

\vskip 0.05in

The case for $t> \mathcal{T}$ follows by iteration. From Proposition \ref{pre_Lya}, for $n\in\mathbb{N}$, one has $\alpha(n\mathcal{T}) = \alpha(\mathcal{T})^n$, so that $r(n\mathcal{T}) = \max\{r\alpha(\mathcal{T})^n, r_0\}$. By induction, we can show that
\begin{align}\label{contractionNtimes}
\|\mathcal{P}_{s, s+n\mathcal{T}}\phi\|_{V^{r(n\mathcal{T})}}\leq C^n\|\phi\|_{V^r}.
\end{align}
Indeed, the base case for $n=1$ has been proved. In particular,   replacing $s$ by $s+k\mathcal{T}$, it follows that
\begin{align}\label{contraction1times}
\|\mathcal{P}_{s+k\mathcal{T}, s+(k+1)\mathcal{T}}\phi\|_{V^{r(\mathcal{T})}}\leq C\|\phi\|_{V^r}.
\end{align}
Assume that for $n=k$, inequality \eqref{contractionNtimes} is true for all $s\in\mathbb{R}$ and $r\in[r_0, 2\kappa]$. Then since $r(\mathcal{T}) \in [r_0, 2\kappa]$, one has
\begin{align}\label{contractionktimes}
\|\mathcal{P}_{s, s+k\mathcal{T}}\phi\|_{V^{\max\{r(\mathcal{T})\alpha(\mathcal{T})^{k}, r_0\}}}\leq C^k\|\phi\|_{V^{r(\mathcal{T})}}.
\end{align}
It follows from \eqref{contraction1times}, \eqref{contractionktimes} and the evolution property of the transition operator  that
\begin{align*}\label{precontractionk+1times}
&\|\mathcal{P}_{s, s+(k+1)\mathcal{T}}\phi\|_{V^{\max\{r(\mathcal{T})\alpha(\mathcal{T})^{k}, r_0\}}} = \|\mathcal{P}_{s, s+k\mathcal{T}}\mathcal{P}_{s+k\mathcal{T}, s+(k+1)\mathcal{T}}\phi\|_{V^{\max\{r(\mathcal{T})\alpha(\mathcal{T})^{k}, r_0\}}}\\
&\leq C^k \|\mathcal{P}_{s, s+k\mathcal{T}}\phi\|_{V^{r(\mathcal{T})}}\leq C^{k+1}\|\phi\|_{V^r}.
\end{align*}
Since $r(\mathcal{T}) = \max\{r\alpha(\mathcal{T}), r_0\}$, we find that $r((k+1)\mathcal{T})=\max\{r\alpha(\mathcal{T})^{k+1}, r_0\}\geq {\max\{r(\mathcal{T})\alpha(\mathcal{T})^{k}, r_0\}}$ always holds. Hence
\[\|\mathcal{P}_{s, s+n\mathcal{T}}\phi\|_{V^{r((k+1)\mathcal{T})}}\leq \|\mathcal{P}_{s, s+(k+1)\mathcal{T}}\phi\|_{V^{\max\{r(\mathcal{T})\alpha(\mathcal{T})^{k}, r_0\}}}\leq C^{k+1}\|\phi\|_{V^r}.\]
This completes the induction step. Hence \eqref{contractionNtimes} is true for all $n\in\mathbb{N}$.  For any $t\geq \mathcal{T}$, there are unique $k\in \mathbb{N}$ and $\beta\in [0,\mathcal{T})$ such that $t = k\mathcal{T}+\beta$. Since
$r(\beta) \in[r_0, 2\kappa]$, it follows from \eqref{contractionNtimes} that
\begin{align}\label{contractionktimeswithbeta}
\|\mathcal{P}_{s, s+k\mathcal{T}}\phi\|_{V^{\max\{r(\beta)\alpha(\mathcal{T})^{k},r_0\}}}\leq C^{k}\|\phi\|_{V^{r(\beta)}}.
\end{align}
Combining \eqref{contractionktimeswithbeta} with \eqref{eqcontraction1} for $\beta\in [0, \mathcal{T})$, and the fact that
\[r(t) = \max\{r\alpha(\beta)\alpha(\mathcal{T})^{k}, r_0\}\geq \max\{r(\beta)\alpha(\mathcal{T})^{k},r_0\}, \]
we obtain \eqref{eqcontraction1} for $t\geq \mathcal{T}$,
\begin{align*}
\|\mathcal{P}_{s, s+t}\phi\|_{V^{r(t)}}&\leq \|\mathcal{P}_{s, s+k\mathcal{T}}\mathcal{P}_{s+k\mathcal{T},s+k\mathcal{T}+\beta }\phi\|_{\max\{r(\beta)\alpha(\mathcal{T})^{k},r_0\}}\\
&\leq C^{k}\|\mathcal{P}_{s+k\mathcal{T},s+k\mathcal{T}+\beta }\phi\|_{V^{r(\beta)}}\\
&\leq C^{k+1}\|\phi\|_{V^r}\leq Ce^{\gamma t}\|\phi\|_{V^r},
\end{align*}
by choosing appropriate constants $C, \gamma>0$ since $k = \frac{t-\beta}{\mathcal{T}}$. The proof for \eqref{eqcontraction2} in the case $t\geq \mathcal{T}$ is similar.

\end{proof}
\begin{corollary}\label{coroadaptedcontraction}
There exist $m>0$, $C = C(m)>0$ such that
\[\|\mathcal{P}_{s, s+m\mathcal{T}}\phi\|_{V^r}\leq C\|\phi\|_{\rho_{r},\tau}\]
for all $\phi\in C^1(H)$, $r\in (1-\alpha(\mathcal{T}), 1]$ and $s, \tau \in \mathbb{R}$.
\end{corollary}
\begin{proof}
Let  $r_n = r_0 + \alpha(\mathcal{T})^n\kappa r$, where $r_0$ is from Proposition \ref{pre_Lya} and can be chosen to be arbitrarily close to $0$.  Since $\alpha(\mathcal{T})<1$, we can choose a large $m$ such that $ \alpha(\mathcal{T})^n\kappa r<1$. Fix such an $m$. Then we can choose $r_0$ small such that $r_m\leq r$ and $r_0< \alpha(\mathcal{T})^m\kappa r$.  As a result, we have by Theorem \ref{adaptedcontraction} and Proposition \ref{boundedbothsides}, that for $r\in (1-\alpha(\mathcal{T}), 1]$,
\[\|\mathcal{P}_{s, s+m\mathcal{T}}\phi\|_{V^r}\leq \|\mathcal{P}_{s, s+m\mathcal{T}}\phi\|_{V^{r_m}}\leq \|\mathcal{P}_{s, s+m\mathcal{T}}\phi\|_{V^{\alpha(\mathcal{T})^m\kappa r}}\leq C(m)\|\phi\|_{\kappa r}\leq C(m)\|\phi\|_{\rho_r,\tau},\]
where in the penultimate step we use inequality \eqref{eqcontraction1} with
\[r(t) = \max\{\kappa r\alpha(m\mathcal{T}), r_0\} = \max\{\alpha(\mathcal{T})^m\kappa r, r_0\} = \alpha(\mathcal{T})^m\kappa r.\]
\end{proof}

We are now in a position to prove Theorem \ref{mixingob}.
\begin{proof}[Proof of Theorem \ref{mixingob}]
By Corollary \ref{Lipmixing}, and Proposition \ref{boundedbothsides}, we have
\[\|\mathcal{P}_{s, s+t}\phi-\mu_{s+t}(\phi)\|_{\rho, s}\leq Ce^{-\gamma t}\|\phi-\mu_{s}(\phi)\|_{\rho, s}\leq Ce^{-\gamma t}\|\phi-\mu_{s}(\phi)\|_{\eta},\]
for $s\in\mathbb{R}$ and $t\geq 0$.  By Corollary \ref{coroadaptedcontraction}, there exists $m>0$ such that
\[\|\mathcal{P}_{s, s+m\mathcal{T}}\phi\|_{\eta}\leq C(m)\|\phi\|_{\rho,\tau}.\]
Replacing $\phi$ by $\mathcal{P}_{s+m\mathcal{T}, s+m\mathcal{T}+t}\phi - \mu_{s+m\mathcal{T}+t}(\phi)$ and letting $\tau = s+m\mathcal{T}$ on the right hand side of the above inequality, we  have
\begin{align*}
\|\mathcal{P}_{s, s+m\mathcal{T}+t}\phi - \mu_{s+m\mathcal{T}+t}(\phi)\|_{\eta}&\leq C(m)\|\mathcal{P}_{s+m\mathcal{T}, s+m\mathcal{T}+t}\phi - \mu_{s+m\mathcal{T}+t}(\phi)\|_{\rho, s+m\mathcal{T}}\\
&\leq C(m)e^{-\gamma t}\|\phi - \mu_{s+m\mathcal{T}}(\phi)\|_{\eta}.
\end{align*}
Combining the above estimate with \eqref{mu_s_phi}, one has for $t\geq m\mathcal{T}$,
\begin{align*}
\|\mathcal{P}_{s, s+t}\phi - \mu_{s+t}(\phi)\|_{\eta}\leq C e^{-\gamma(t-m\mathcal{T})}\left(\|\phi\|_{\eta}+|\mu_{s+m\mathcal{T}}(\phi)|\right)\leq Ce^{-\gamma t}\|\phi\|_{\eta},
\end{align*}
where $C$ depends on the constants $m, \mathcal{T}$.

By Theorem \ref{adaptedcontraction}, we have for all $t>0$,
\[\|\mathcal{P}_{s, s+t}\phi\|_{\eta} = \|\mathcal{P}_{s, s+t}\phi\|_{V^1}\leq \|\mathcal{P}_{s, s+t}\phi\|_{V^{\max\{\alpha(t), r_0\}}}\leq Ce^{\gamma t}\|\phi\|_{\eta}.\]
So for $0\leq t\leq m\mathcal{T}$, $\|\mathcal{P}_{s, s+t}\phi\|_{\eta}\leq  Ce^{\gamma m\mathcal{T}}\|\phi\|_{\eta}$. Replacing $\phi$ by $\phi-\mu_{s+t}(\phi)$, we have
\begin{align*}
\|\mathcal{P}_{s, s+t}\phi - \mu_{s+t}(\phi)\|_{\eta}\leq Ce^{\gamma m\mathcal{T}}\|\phi-\mu_{s+t}(\phi)\|_{\eta}\leq Ce^{\gamma m\mathcal{T}}\|\phi\|_{\eta}\leq Ce^{-\gamma t}\|\phi\|_{\eta},
\end{align*}
by choosing the last constant $C$ larger. The proof is complete.
\end{proof}

\subsection{The weak irreducibility}\label{weakLya}
In this subsection, we show that the Navier-Stokes system \eqref{NS} has the following form of weak irreducibility,
which is necessary to obtain the desired  contraction on $\mathcal{P}(H)$.
 \begin{proposition}[Weak Irreducibility]\label{contractionirreducible}
If the Grashof number $G_1<1/c_0$,  then for any $R, \varepsilon>0, r\in (0,1]$, there exists $n_0\in\mathbb{N}$ such that for any $n\in\mathbb{N}$ and $n\geq n_0$, there exits $a>0$ so that
\[\inf_{\|w_1\|, \|w_2\|\leq R}\sup_{\Gamma\in \mathcal{C}\left(\mathcal{P}_{s, s+n\mathcal{T}}^{*}\delta_{w_1},\mathcal{P}_{s, s+n\mathcal{T}}^{*}\delta_{w_2}\right)}\Gamma\left\{(w_1', w_2')\in H\times H : \rho_{r}(w_1',w_2')<\varepsilon\right\}\geq a,\]
for all $s\in\mathbb{R}$, where the metric $\rho_r$ is given by \eqref{familymetrics}.
\end{proposition}

\vskip0.05in

As we mentioned in the introduction, to ensure the weak irreducibility in our work, we impose a condition on the Grashof number $G_1$ instead of the range condition. Under this condition $G_1<1/c_0$,  there is a unique stable  periodic solution $z(t)$ of the deterministic equation corresponding  to equation \eqref{NS}, see \cite{CV02b}. From this fact we can show that any two solutions of equation \eqref{NS} starting from different initial points at the same initial time $s$  have a positive probability  to be attracted by the periodic solution $z(t)$.  Since the driven force $f(t, x)$ is periodic, the solutions will recurrently enter into smaller neighborhood of the point $z(s)$ at periodic times $s+n\mathcal{T}$ for large $n\in\mathbb{N}$.  This gives us the weak irreducibility as in Proposition \ref{contractionirreducible}, which is an adaptation of the weak irreducibility introduced in \cite{HM08} to the time inhomogeneous and periodic setting. The proof  is based on the following Lemma \ref{usmall} and Lemma \ref{access}.

\vskip0.05in

For $\delta>0, T>s$, we let
\[\mathrm{\Omega}_{\delta, T} = \left\{\omega\in \mathrm{\Omega}: |W(\omega,t)|\leq\delta, \quad\text{for all}\quad  t\in[s, T]\right\},\]
where $|\cdot|$ denote the usual norm in $\mathbb{R}^d$.
Then by the property of Wiener process, there exists $\delta^* = \delta^*(\delta,T)$ such that
\begin{equation}
\mathbf{P}(\mathrm{\Omega}_{\delta, T} )\geq \delta^*>0.
\end{equation}\label{PB}
Let $V_t$ be the Ornstein-Uhlenbeck process given by
\[\partial_tV_t = \nu \mathrm{\Delta} V_t + GdW_t, \quad V(s) = 0.\]
By viewing the process $V_t$ as a stochastic convolution,  one has  the following
\begin{lemma}\cite{DZ14}\label{usmall}
For any $\delta>0, T>s$, there exists a deterministic constant $\varepsilon_{\delta,T}>0$ such that $\varepsilon_{\delta,T}\rightarrow 0 $ as $\delta\rightarrow0$ for $T$ fixed, and
\begin{equation}
\sup_{t\in[s, T]}\left\|V(t,\omega)\right\|_1\leq\varepsilon_{\delta,T},\text{ for all} \quad \omega\in \mathrm{\Omega}_{\delta, T}.
\end{equation}\label{smallconvolution}
\end{lemma}

Recall that $w_{s, t}: = w(t,\omega; s, w_0)$ denotes the solution to the Navier-Stokes equation \eqref{NS}.  Observe that  $Z_t := w_{s, t} - V_t$ solves the random PDE perturbed by the Ornstein-Uhlenbeck process $V_t$
\[\partial_t Z = \nu\mathrm{\Delta} Z - B\left(\mathcal{K}\left(Z + V_t\right), Z + V_t\right) +f(t), \quad Z(s)= w_0.\]
When the Grashof number $G_1<1/c_0$, there exists  a unique $\mathcal{T}$-periodic solution $z_t$ (see \cite{CV02b}) of the corresponding unperturbed deterministic equation
\begin{align}\label{eqperiodic}
\partial_t z = \nu\mathrm{\Delta} z - B(\mathcal{K}z, z) + f(t), \quad z(s) = z_s.
\end{align}
Then $u_t := Z_t - z_t$ solves the equation
\[ \partial_t u = \nu\mathrm{\Delta} u - B\left(\mathcal{K}\left(Z_t + V_t\right), Z_t + V_t\right) + B(\mathcal{K}z_t, z_t), \quad u(s) = w_0 - z_s \]
Note $- B\left(\mathcal{K}\left(Z_t + V_t\right), Z_t + V_t\right) + B(\mathcal{K}z_t, z_t) =- B\left(\mathcal{K}\left(u + V_t \right), u+ V_t + z_t\right) - B(\mathcal{K}z_t, u+V_t)$. So the equation for $u_t$ is actually
\begin{align}\label{equ}
 \partial_tu = \nu\mathrm{\Delta} u - B\left(\mathcal{K}\left(u + V_t \right), u+ V_t + z_t\right) - B(\mathcal{K}z_t, u+V_t), \quad u(s) = w_0 - z_s.
\end{align}
In what follows, we will estimate the $H$ norm of $u$ from the above equation so that one can see the attraction property of $z_t$ when the perturbation $V_t$ is sufficiently small.

\vskip0.05in

The following lemma shows that there is a positive probability such that, uniformly for initial values in any bounded region,  the process $u_t$  can eventually enter the ball  centered at the origin  with arbitrary small radius.
\begin{lemma}\label{access}
Suppose $G_1<1/c_0$. For arbitrary $R, \sigma>0$, there exists $\delta = \delta(R,\sigma)>0, T_0 = T_0(R, \sigma, \delta)>s$ such that, for $T\geq T_0$,
\begin{equation}
\left\|u(T, u_s)\right\|\leq \frac{\sigma}{2},  \quad  \text{on} \quad \mathrm{\Omega}_{\delta, T},
\end{equation}\label{decay}
for any $u_s\in B_R(0)\subset H$.
\end{lemma}
\begin{proof}
Taking $H$ inner product of equation \eqref{equ} with $u$, and using the fact that $\langle B(\mathcal{K}u, v),v\rangle = 0$  and $\left|\langle B(\mathcal{K}u, v),w\rangle\right|\leq c_0\|\mathcal{K}u\|_1\|v\|_1\|w\|_1=c_0\|u\|\|v\|_1\|w\|_1$ for any $u, v, w\in H$, and the Poincare inequality $\|u\|\leq \|u\|_1$ we have
\begin{align*}
\frac{1}{2}\partial_t\|u\|^2 &= -\nu\|\nabla u\|^2- \langle B(\mathcal{K}u,V_t+z_t),u\rangle - \langle B(\mathcal{K}V_t,V_t+z_t),u)-\langle B(\mathcal{K}z_t,V_t),u)\\
&\leq-\nu ||u||_1^2+ c_0\|u\|\|u\|_1\|z_t\|_1+c_0\|u\|_1^2\|V_t\|_1+c_0\|V_t\|_1^2\|u\|_1+2c_0\|V_t\|_1\|z_t\|_1\|u\|_1.
\end{align*}
By Lemma \ref{usmall}, $||V_t||_1\leq\varepsilon :=\varepsilon_{\delta,T} $ for $\delta, T$ that will be determined later. And from \cite{CV02a}, it is known that $\|z_t\|_1\leq R' $ for any $t\in \mathbb{R}$ , where $R'$ is a constant depending on $\nu$ and $\|f\|_{\infty}$. Therefore there exists a constant $C = C(\nu,c_0,\|f\|_{\infty})$ such that for fixed $\alpha\in(0,1)$,
\begin{align*}
\partial_t\|u\|^2 +2\nu ||u||_1^2&\leq 2c_0\|u\|\|u\|_1\|z_t\|_1 + 2c_0\varepsilon^2\|u\|_1+2 c_0\varepsilon\|u\|_1^2+4c_0\varepsilon\|z_t\|_1\|u\|_1\\
&\leq c_0^2(\alpha\nu)^{-1}\|u\|^2\|z_t\|_1^2 + \left(\alpha\nu+C\varepsilon \right)\|u\|_1^2+ C\varepsilon.
\end{align*}
Choosing $\delta$ sufficiently small so that $\alpha\nu+C\varepsilon\leq \nu$, then we have
\begin{align*}
\partial_t\|u\|^2 \leq -\left(\nu -  c_0^2(\alpha\nu)^{-1}\|z_t\|_1^2 \right)\|u\|^2+ C\varepsilon.
\end{align*}
By Gronwall's inequality,
\begin{align*}
\|u(t)\|^2&\leq \|u_s\|^2 \exp\left(-\int_s^t\left(\nu -   c_0^2(\alpha\nu)^{-1}\|z_\tau\|_1^2\right)d\tau\right) + C \varepsilon \int_s^t\exp\left(-\int_\tau^t\left(\nu -   c_0^2(\alpha\nu)^{-1}\|z_r\|_1^2\right)dr\right) d\tau\\
&: = I_1 + I_2.
\end{align*}
Multiplying equation \eqref{eqperiodic} by $z_t$ and integrating over the torus $\mathbb{T}^2$, one has
\begin{align}\label{periodicestimate}
\partial_t\|z_t\|^2\leq -\nu\|z_t\|_1^2 + \nu^{-1}\|f(t)\|^2.
\end{align}
Hence we deduce the energy inequality
\begin{align}\label{energyinequality}
\|z_t\|^2+\nu\int_s^t\|z_\tau\|_1^2d\tau\leq \|z_s\|^2+\nu^{-1}\int_s^t\|f(\tau)\|^2d\tau.
\end{align}
Again applying  Gronwall's inequality to \eqref{periodicestimate}, one obtains
\[\|z_t\|^2\leq e^{-\nu(t-s)}\|z_s\|^2+\frac{\|f\|_{\infty}^2}{\nu^2}\left(1- e^{-\nu(t-s)}\right),\]
which implies by using the periodicity of $z(t)$ that $\|z_t\|^2\leq \frac{\|f\|_{\infty}^2}{\nu^2}$ for every $t$.
Now it follows from the energy inequality \eqref{energyinequality}  that
\begin{align*}
\int_s^t\|z_\tau\|_1^2d\tau \leq \nu^{-1}\|z_s\|^2+\nu^{-2}(t-s)\|f\|_{\infty}^2 \leq \nu^{-3}\|f\|_{\infty}^2+\nu^{-2}(t-s)\|f\|_{\infty}^2.
\end{align*}
Therefore
\[I_1 \leq C \|u_s\|^2 e^{-\beta(t-s)},\]
where the constant $C = \exp\left(\frac{c_0^2\|f\|_{\infty}^2}{\alpha\nu^4}\right)$, and $\beta = \nu -\frac{c_0^2\|f\|_{\infty}^2}{\alpha\nu^3}$. Since we have assumed  $G_1= \|f\|_{\infty}/\nu^2<1/c_0$, one can choose $\alpha\in(0,1)$ such that $G_1<\sqrt{\alpha}/c_0$, hence $\beta>0$.
The estimate for $I_1$ in turn implies that
\[I_2\leq \frac{C\varepsilon}{\beta}\left(1- e^{-\beta(t-s)}\right).\]
The proof  is  then complete by letting $\delta$ sufficiently small and $T$ sufficiently large so that $I_1+ I_2<\frac{\sigma}{2}$ for all $t\geq T$.
\end{proof}

Now we can show the weak irreducibility.
\begin{proof}[Proof of Proposition \ref{contractionirreducible}.]
By Lemma \ref{usmall} and Lemma \ref{access}, we know that  for any $\sigma> 0$ and $R>0$,  there is small enough $\delta>0$  and  $T_0(\sigma, R, \delta)>s$ such that for $T>T_0$, we have for any $w_0\in B_R(0)$,
\[\mathbf{P}\left(\|w(T,s,w_0)-z(T)\|<\sigma\right)\geq \mathbf{P}\left(\|u(T,u_0)\|<\frac{\sigma}{2} \quad\text{and}\quad \|V_{T}\|<\frac{\sigma}{2}\right )\geq \mathbf{P}(\mathrm{\Omega}_{\delta,T})>0,\]
since $\left\|w_{s, t} - z_t\right\|\leq \left\|w_{s, t} - z_t - V_t\right\| + \left\|V_t\right\| = \left\|u_t\right\| + \left\|V_t\right\|$. Let $v = z(s)$. Then $z(s+nT) = v$ for $n\in\mathbb{Z}$ since $z_t$ is a $\mathcal{T}$-periodic solution. By replacing $T$ with $s + n\mathcal{T}$ for  $n\in\mathbb{R}$, there is $n_0>0$ such that for all $n\geq n_0$,  by choosing $\delta^* = \mathbf{P}(\mathrm{\Omega}_{\delta,s+n\mathcal{T}})>0$, we have
\begin{align}\label{prop3.3}
\inf_{\|w_0\|\leq R}\mathcal{P}_{s, s+n\mathcal{T}}(w_0, B_{\sigma}(v))\geq \delta^*>0,
\end{align}
where $B_{\sigma}(v) $ denotes the ball in $(H, \|\cdot\|)$ with radius $\sigma$ and centered at $v$.
Now for any $\varepsilon>0$,  we choose $\sigma = \frac{\varepsilon}{2V(\|v\|+\varepsilon)}$, here $V(w) = e^{\eta\|w\|^2}$ for $w\in H$ is the Lyapunov function. Then  $\widetilde{B}: = B_{\sigma}(v)\times B_{\sigma}(v)\subset \{(w_1', w_2')\in H\times H|\rho_r(w_1', w_2')<\varepsilon\}$. It now suffices to show that there is a $n_0>0$  such that for $n>n_0$, there is $a>0$, and
\[\inf_{\|w_1\|, \|w_2\|\leq R}\mathcal{P}_{s, s+n\mathcal{T}}^{*}\delta_{w_1}\otimes\mathcal{P}_{s, s+n\mathcal{T}}^{*}\delta_{w_2}(\widetilde{B})\geq a.\]
But this is equivalent to $\inf_{\|w_1\|, \|w_2\|\leq R}\mathcal{P}_{s, s+n\mathcal{T}}(w_1,B(v, \sigma))\mathcal{P}_{s, s+n\mathcal{T}}(w_2, B(v, \sigma))\geq a$, which is a consequence of \eqref{prop3.3}.
\end{proof}
As a byproduct, we have the following corollary, which is another form of weak irreducibility. By combining it with the asymptotic strong Feller property for time inhomogeneous system proposed in \cite{DD08}, one can show uniqueness of periodic invariant measures. This is a generalization of the Doob and Khasminskii type argument from the seminal work \cite{HM06,HM11} to the time inhomogenunous case.
\begin{corollary}\label{DoobKha}
Assume $G_1<1/c_0$. Then for any $w_1, w_2\in H$ and $\tau\in\mathbb{R}$, there exists a $v\in H$ such that  for any  $\varepsilon>0$, there exist $n_1, n_2\in\mathbb{N}$  such that $\mathcal{P}_{\tau, \tau+n_i\mathcal{T}}(w_i, B_{\varepsilon}(v))>0$. Furthermore, if $\mathcal{P}_{s, t}$ satisfies the asymptotic strong Feller property as in \cite{DD08}, then system \eqref{NS} has at most one periodic invariant measure.
\end{corollary}
\begin{proof}

Suppose there are more than one $\mathcal{T}$-periodic invariant measure. Then there are at least two invariant measures for the associated homogenized semigroup. Let $\mu^1$ and $\mu^2$ be two distinct ergodic measures of $\left(\mathcal{P}_{\tau}\right)_{\tau \geq 0}$ and denote the corresponding $\mathcal{T}$-periodic continuous disintegrations by $\left(\mu_{t}^1\right)_{t \in \mathbb{R}}$ and  $\left(\mu_{t}^2\right)_{t \in \mathbb{R}}$. Since $\left(\mathcal{P}_{s, t}\right)_{s \leq t}$ is asymptotically strong Feller, it follows from Theorem 4.32 in \cite{RM13} that $v \notin \operatorname{supp}\left(\mu_{t}^1\right) \cap \operatorname{supp}\left(\mu_{t}^2\right)$ for any $t \in \mathbb{R}$ and $v \in H,$ i.e., $\operatorname{supp}\left(\mu_{t}^1\right) \cap$ $\operatorname{supp}\left(\mu_{t}^2\right)=\emptyset$ for all $t \in \mathbb{R}$.
Now fix a $\tau\in \mathbb{R}$ and let $w_i\in \operatorname{supp}\left(\mu_{\tau}^i\right)$ for $i=1, 2$. By the weak irreducibility, there is a $v\in H $ such that for any $\varepsilon>0$, there exist $n_1,n_2\in\mathbb{N}$ with $\mathcal{P}_{\tau, \tau+n_i\mathcal{T}}(w_i, B_{\varepsilon}(v))>0$. Since $\left(\mathcal{P}_{s, t}\right)_{s \leq t}$ is Feller, and $B_{\varepsilon}(v)$ is open, there exists $\delta>0$ such that \[\inf_{u\in B_{\delta}(w_i)}\mathcal{P}_{\tau, \tau+ n_i\mathcal{T}}(u, B_{\varepsilon}(v))>0.\]
Since $w_i\in \operatorname{supp}\left(\mu_{\tau}^i\right)$, we have $\mu_{\tau}^i(B_{\delta}(w_i))>0$. Note $\mu_{\tau+n_i\mathcal{T}}^i = \mu_{\tau}^i$ by the $\mathcal{T}$-periodicity. Therefore
\begin{align*}
\mu_{\tau}^i(B_{\varepsilon}(v))& = \mu_{\tau+n_i\mathcal{T}}^i(B_{\varepsilon}(v)) = \int_{H}\mathbbm{I}_{B_{\varepsilon}(v)}(u)\mu_{\tau+n_i\mathcal{T}}^i(du)=\int_{H}\mathcal{P}_{\tau, \tau+n_i\mathcal{T}}\mathbbm{I}_{B_{\varepsilon}(v)}(u)\mu_{\tau}^i(du)\\
&\geq \int_{B_{\delta}(w_i)}\mathcal{P}_{\tau, \tau+n_i\mathcal{T}}\mathbbm{I}_{B_{\varepsilon}(v)}(u)\mu_{\tau}^i(du)\geq \mu_{\tau}^i(B_{\delta}(w_i))\inf\limits_{u\in B_{\delta}(w_i)}\mathcal{P}_{\tau, \tau+n_i\mathcal{T}}(u, B_{\varepsilon}(v))>0
\end{align*}
Since $\varepsilon$ is arbitrary, this shows that $v \in \operatorname{supp}\left(\mu_{\tau}^1\right) \cap \operatorname{supp}\left(\mu_{\tau}^2\right)$, which gives the desired contradiction.
\end{proof}
\begin{remark}
 The weak irreducibility in this subsection is inspired by that in the time homogeneous case \cite{HM08,HM11}. The difference is that here we need to work with the time inhomogeneity and periodicity. The asymptotic strong Feller property proposed in \cite{DD08} will be proved in Corollary \ref{ASF}.
\end{remark}

\subsection{Contraction on $\mathcal{P}(H)$}\label{section-contraction}
In this subsection, we prove Theorem \ref{contractiontransition}, namely, the strong contraction property \eqref{eqcontraction} of  the transition operator $\mathcal{P}_{s, t}^{*}$  when acting on $\mathcal{P}(H)$.  The result can be regarded as an extension of the work of Hairer and Mattingly \cite{HM08} to the time inhomogeneous setting, where they proved the existence of a unique invariant measure that is exponentially mixing under the Wasserstein metric $\rho$. The idea behind their result  dates back to the early work of Dobelin \cite{Doe37} and Harris \cite{Har56} for finite dimensional systems.
The proof here will be accomplished through a combination of the weak irreducibility from subsection \ref{weakLya}, the Lyapunov structure as in Proposition \ref{pre_Lya}, and the following gradient estimate.
\begin{proposition}\label{pre_gradientinequalityprop}
Assume $A_{\infty}=H$. Then there exists constant $\eta_0>0$ so that for every $\eta\in (0, \eta_0]$ and  $a>0$
there exists constants $ C = C(\eta, a)>0$ and $p\in(0,1)$ such that
\begin{align*}
\left\|\nabla \mathcal{P}_{s, s+t} \phi(w)\right\| \leq C \exp(p\eta\|w\|^2)\left(\sqrt{\left(\mathcal{P}_{s, s+t}|\phi|^{2}\right)(w)}+ e^{-at} \sqrt{\left(\mathcal{P}_{s, s+t}\|\nabla \phi\|^{2}\right)(w)}\right)
\end{align*}
for every  Frechet differentiable function $\phi$, every $w \in H, s\in\mathbb{R}$ and  $t\geq 0$. Here $ C(\eta, a)$ does not depend on initial condition $(s, w)$ and $\phi$.
\end{proposition}
The proof of this proposition is much involved, which requires an infinite dimensional version of the Malliavin calculus and an infinite dimensional Hörmander type theorem. In the time homogeneous case, it has been well established in the works from \cite{BM07,HM06,HM11,MP06}. Since the bounds for the solution of equation \eqref{NS} behave in a uniform way with respect to the initial time, the proof of Proposition \ref{pre_gradientinequalityprop} in the time inhomogeneous case is essentially the same as that in \cite{HM06,HM11}. For the reader's convenience, we supply a proof in Appendix \ref{gradientinequality}.

\vskip0.05in

As in \cite{HM08}, to prove Theorem \ref{contractiontransition}, we use a metric $d$ on $H$ that is equivalent to $\rho$ but easier to handle with the estimates from previous subsections.  Recall that $r_0$ is the constant from the Lyapunov structure in Proposition \ref{pre_Lya}, and $\rho_r$ is the metric defined as in \eqref{familymetrics}. For constants $r\in [r_0, 1)$,
$\delta>0$ and $\beta\in (0, 1)$, the metric $d$ is defined as
\begin{align}\label{thedistanced}
d(w_1, w_2) = \left(1\land \frac{\rho_r(w_1, w_2)}{\delta}\right)+\beta\rho(w_1, w_2),
\end{align}
which is equivalent to $\rho$ since $\beta\rho(w_1, w_2)\leq d(w_1, w_2)\leq (\delta^{-1}+\beta)\rho(w_1, w_2).$

\vskip0.05in

We first give a lemma that can reduce the contraction \eqref{eqcontraction} to a relatively simpler case. The first part of the  lemma allows us to extend the contraction  of the transition operator on $H$ (embedded in $\mathcal{P}(H)$) to a contraction on $\mathcal{P}(H)$. And the second part asserts that one can obtain the contraction for all times once the transition operator is a contraction at a particular time.
\begin{lemma}\label{contractlemma0} We have
\begin{enumerate}
\item For  $s\in \mathbb{R}, t\geq 0$, and any given distance $d$ on $H$, if $d(\mathcal{P}_{s, s+t}^{*}\delta_{w_1},\mathcal{P}_{s, s+t}^{*}\delta_{{w_2}} )\leq \alpha d(w_1, {w_2})$,  for any $w_1, {w_2}\in H$, then $d(\mathcal{P}_{s, s+t}^{*}\mu_1,\mathcal{P}_{s, s+t}^{*}\mu_2)\leq \alpha d(\mu_1, \mu_2)$ for any $\mu_1, \mu_2\in \mathcal{P}(H)$.
\item If there are  $N\in\mathbb{N}$, $\alpha\in (0, 1)$, such that for any $w_1, {w_2}\in H$, $r\in [r_0, 1]$ and $s\in\mathbb{R}$, $\rho_r(\mathcal{P}_{s, s+ N\mathcal{T}}^{*}\delta_{w_1},\mathcal{P}_{s, s+N\mathcal{T}}^{*}\delta_{w_1})\leq \alpha \rho_r(w_1, {w_2})$. Then there are $C>0, \gamma>0$ such that
\[\rho_r(\mathcal{P}_{s, s+t}^{*}\mu_1,\mathcal{P}_{s, s+t}^{*}\mu_2)\leq Ce^{-\gamma t} \rho_r(\mu_1, \mu_2),\]
 for any $t\geq 0$.
\end{enumerate}
\end{lemma}
\begin{proof}
Fix $s\in\mathbb{R}$ and $t\geq 0$. By Theorem 4.4.3 from \cite{Kul17} on the existence of optimal couplings for probability kernels, it follows that for the transition probability kernel
\[\mathcal{P}_{s, s+t}(\cdot, \cdot): H\times\mathcal{B}(H)\rightarrow\mathbb{R},\]
there is an optimal coupling kernel $Q$  in the sense that
\[Q: \left(H\times H\right) \times \left(\mathcal{B}(H)\otimes\mathcal{B}(H)\right)\rightarrow \mathbb{R}, \]
is a probability kernel on $\left(H\times H, \mathcal{B}(H)\otimes\mathcal{B}(H)\right)$ such that
for every $(w_1, {w_2})\in H\times H$, $Q((w_1, {w_2})$ is an optimal coupling of the transition probabilities $\mathcal{P}_{s, s+t}^{*}\delta_{w_1}$ and $\mathcal{P}_{s, s+t}^{*}\delta_{{w_2}}$:
\begin{align}\label{P_Q}
d(\mathcal{P}_{s, s+t}^{*}\delta_{w_1}, \mathcal{P}_{s, s+t}^{*}\delta_{{w_2}}) = \inf_{\mu\in\mathcal C}\int_{H\times H}d(u, v)\mu(dudv) =\int_{H\times H}d(u, v)Q((w_1, {w_2}), dudv),
\end{align}
where $\mathcal{C} = \mathcal{C}(\mathcal{P}_{s, s+t}^{*}\delta_{w_1},\mathcal{P}_{s, s+t}^{*}\delta_{{w_2}})$ is the set of all couplings of the transition probabilities $\mathcal{P}_{s, s+t}^{*}\delta_{w_1}$ and $\mathcal{P}_{s, s+t}^{*}\delta_{{w_2}}$.

\vskip0.05in

Define the operator $P_{Q}$ acting on $B_b(H\times H)$ by
\[P_{Q}\phi(w_1, {w_2}) = \int_{H\times H}\phi(u, v)Q((w_1, {w_2}), du,dv),\]
which induces an operator $P_{Q}^*$ on $\mathcal{P}(H\times H)$ by duality, $P_{Q}^*\mu(A\times B) = \int_{H\times H}Q((w_1, {w_2}), A\times B)\mu(dw_1d{w_2})$. One can verify that if $\mu$ is a coupling of $\mu_1, \mu_2$, then $P_{Q}^*\mu$ is a coupling of $\mathcal{P}_{s, s+t}^{*}\mu_1$ and $\mathcal{P}_{s, s+t}^{*}\mu_2$.  Suppose $\mu_{*}$ is an optimal coupling of $\mu_1, \mu_2$. Hence by \eqref{P_Q},
\begin{align*}
d(\mathcal{P}_{s, s+t}^{*}\mu_1, \mathcal{P}_{s, s+t}^{*}\mu_2) &= \inf_{\mu\in \mathcal{C}(\mathcal{P}_{s, s+t}^{*}\mu_1,\mathcal{P}_{s, s+t}^{*}\mu_2)}\int_{H\times H}d(u, v)\mu(dudv)\leq \int_{H\times H}d(u, v)P_{Q}^{*}\mu_{*}(dudv)\\
& = \int_{H\times H}\int_{H\times H}d(u, v)Q((w_1, {w_2}), dudv)\mu_{*}(dw_1d{w_2})\\
& = \int_{H\times H} d(\mathcal{P}_{s, s+t}^{*}\delta_{w_1},\mathcal{P}_{s, s+t}^{*}\delta_{{w_2}} )\mu_{*}(dw_1d{w_2})\\
& \leq \alpha \int_{H\times H} d(w_1, {w_2})\mu_{*}(dw_1d{w_2}) = \alpha d(\mu_1, \mu_2).
\end{align*}
This completes the proof for the first part of the lemma. From Lemma \ref{twopoint}, one has
\[\rho_r(\mathcal{P}_{s, s+t}^*\delta_{w_1},\mathcal{P}_{s, s+t}^*\delta_{w_2})\leq \mathbf{E}\rho_r(\Phi_{s, s+t}(w_1), \Phi_{s, s+t}({w_2}))\leq C\rho_r(w_1, {w_2}),\]
for any $s\in \mathbb{R}$ and $t\in [0, \mathcal{T}].$ Observe that
\begin{align*}
\rho_r(\mathcal{P}_{s, s+2\mathcal{T}}^*\delta_{w_1},\mathcal{P}_{s, s+2\mathcal{T}}^*\delta_{w_2}) &= \rho_r(\mathcal{P}_{s+\mathcal{T}, s+2\mathcal{T}}^*\mathcal{P}_{s, s+\mathcal{T}}^*\delta_{w_1}, \mathcal{P}_{s+\mathcal{T}, s+2\mathcal{T}}^*\mathcal{P}_{s, s+\mathcal{T}}^*\delta_{w_2})\\
&\leq C\rho_r(\mathcal{P}_{s, s+\mathcal{T}}^*\delta_{w_1},\mathcal{P}_{s, s+\mathcal{T}}^*\delta_{w_2})\leq C^2 \rho_{r}(w_1, {w_2}).
\end{align*}
So by iteration we have for any $n\in \mathbb{N}$,
\[\rho_r(\mathcal{P}_{s, s+n\mathcal{T}}^*\delta_{w_1},\mathcal{P}_{s, s+n\mathcal{T}}^*\delta_{w_2})\leq C^n \rho_r(w_1, {w_2}).\]
Now for any $0\leq t\leq N\mathcal{T}$, we can write $t = k\mathcal{T}+\beta$ for unique integer $k\geq 0$ such that $k\leq N$ and $\beta\in [0, \mathcal{T})$. Therefore
\[\rho_r(\mathcal{P}_{s, s+t}^*\delta_{w_1},\mathcal{P}_{s, s+t}^*\delta_{w_2}) =\rho_r\left(\mathcal{P}_{s+\beta, s+\beta+k\mathcal{T}}^*\mathcal{P}_{s, s+\beta}^*\delta_{w_1}, \mathcal{P}_{s+\beta, s+\beta+k\mathcal{T}}^*\mathcal{P}_{s, s+\beta}^*\delta_{w_2}\right)\leq C^{N+1}\rho_r(w_1, {w_2}), \]
Hence for any $\gamma>0$, choosing $\widetilde{C} = C^{N+1}e^{\gamma N}, $ we have
\begin{align}\label{t<NT}
\rho_r(\mathcal{P}_{s, s+t}^*\delta_{w_1},\mathcal{P}_{s, s+t}^*\delta_{w_2})\leq \widetilde{C}e^{-\gamma t}\rho_r({w_1}, {w_2}).
\end{align}
While for $t>N\mathcal{T}$, one has $t = k N\mathcal{T}+\beta$, where $k\in \mathbb{N}$, and $0\leq \beta<N\mathcal{T}$. By assumption, $\rho_r(\mathcal{P}_{s, s+ N\mathcal{T}}^{*}\delta_{w_1},\mathcal{P}_{s, s+N\mathcal{T}}^{*}\delta_{w_1})\leq \alpha \rho_r(w_1, {w_2})$ for all $s\in\mathbb{R}$.  So for any $k\in\mathbb{N}$, by iteration,
\begin{align}\label{kNT}
\rho_r(\mathcal{P}_{s, s+ kN\mathcal{T}}^{*}\delta_{w_1},\mathcal{P}_{s, s+kN\mathcal{T}}^{*}\delta_{w_1})\leq \alpha^k\rho_r(w_1, {w_2}).
\end{align}
It then follows from inequality \eqref{t<NT}, \eqref{kNT} and the first part of Lemma \ref{contractlemma0} that
\begin{align*}
\rho_r(\mathcal{P}_{s, s+t}^*\delta_{w_1},\mathcal{P}_{s, s+t}^*\delta_{w_2}) &= \rho_r(\mathcal{P}_{s+\beta, s+\beta+kN\mathcal{T}}^*\mathcal{P}_{s, s+\beta}^*\delta_{w_1}, \mathcal{P}_{s+\beta, s+\beta+kN\mathcal{T}}^*\mathcal{P}_{s, s+\beta}^*\delta_{w_2})\\
&\leq \alpha^k\rho_r(\mathcal{P}_{s, s+\beta}^*\delta_{w_1}, \mathcal{P}_{s, s+\beta}^*\delta_{w_2})\leq \alpha^{\frac{t-\beta}{N\mathcal{T}}}C^{N+1}\rho_r({w_1}, {w_2})\leq Ce^{-\gamma t}\rho_r({w_1}, {w_2}),
\end{align*}
for appropriate constants $C, \gamma>0$. The proof is then complete by invoking again the first part of Lemma \ref{contractlemma0} .
\end{proof}

The weak irreducibility, Lyapunov structure and gradient inequality give us the contractions on the state space at different scales. This fact is summarized in the following lemmas. Lemma \ref{contractlemma1} deals with those points that are far apart, where the contraction is guaranteed by the Lyapunov structure. The gradient inequality  and Lyapunov structure give the contraction at small scales in Lemma \ref{contractlemma2}.  And the contraction of the intermediate scale in Lemma \ref{contractlemma3} is given by the weak irreducibility and Lyapunov structure.

\vskip0.05in

Under the conditions of Theorem \ref{contractiontransition}, the following three lemmas hold for all $s\in \mathbb{R}$. The proof for the lemmas will be given in Appendix \ref{lemmasforcontraction} since it is almost the same as that in \cite{HM08} once we obtain the weak irreducibility, Lyapunov structure and the gradient inequality that are uniform in the initial time. Recall that the metric $d$ is defined in \eqref{thedistanced}, which depends on $\delta, \beta$ and $r$.
\begin{lemma}\label{contractlemma1}
There is a constant $L>0$ such that for any $\delta>0$, $\beta\in(0, 1)$ and $r\in [r_0, 1)$, there is $\alpha_1\in (0, 1)$ such that
\[\left.\begin{array}{l}
\rho({w_1}, {w_2}) \geq L \\
\rho_{r}({w_1}, {w_2}) \geq \delta
\end{array}\right\} \quad \Longrightarrow \quad d\left(\mathcal{P}_{s, s+ n\mathcal{T}}^{*} \delta_{{w_1}}, \mathcal{P}_{s, s+ n\mathcal{T}}^{*} \delta_{{w_2}}\right) \leq \alpha_{1} d({w_1}, {w_2}),\]
for all $n\in\mathbb{N}$.
\end{lemma}
\begin{lemma}\label{contractlemma2}
For any $\alpha_2\in (0, 1)$ there exist $n_0>0$, and $r\in [r_0, 1)$, $\delta>0$ such that
\[
\rho_{r}({w_1}, {w_2})<\delta \Longrightarrow d\left(\mathcal{P}_{s, s+n\mathcal{T}}^{*} \delta_{{w_1}}, \mathcal{P}_{s, s+n\mathcal{T}}^{*} \delta_{{w_2}}\right) \leq \alpha_{2} d({w_1}, {w_2}),
\]
for all $n>n_0$ and $\beta\in(0, 1)$.
\end{lemma}
\begin{lemma}\label{contractlemma3}
For any $L, \delta>0$, $r\in (0,1]$, there is some $n_1>0$ such that  for any $n>n_1$, there are $\beta, \alpha_3\in (0, 1)$ such that
\[
\left.\begin{array}{l}
\rho({w_1}, {w_2})<L \\
\rho_{r}({w_1}, {w_2}) \geq \delta
\end{array}\right\} \quad \Longrightarrow \quad d\left(\mathcal{P}_{s, s+n\mathcal{T}}^{*} \delta_{{w_1}}, \mathcal{P}_{s, s+n\mathcal{T}}^{*} \delta_{{w_2}}\right) \leq \alpha_{3} d({w_1}, {w_2})
\]
\end{lemma}
Now we prove Theorem \ref{contractiontransition} with the help of the above lemmas.
\begin{proof}[Proof of  Theorem \ref{contractiontransition}]
By Lemma \ref{contractlemma0} and the equivalence of the two metrics $\rho$ and $d$, it suffices to show that
\[d(\mathcal{P}_{s, s+N\mathcal{T}}^{*}\delta_{w_1},\mathcal{P}_{s, s+N\mathcal{T}}^{*}\delta_{w_2})\leq \alpha d({w_1}, {w_2}), \]
for some $N\in \mathbb{N}$ and $0<\alpha<1$ and for every $({w_1}, {w_2})\in H\times H$. By Lemma \ref{contractlemma2}, fixing an $\alpha_2\in (0, 1)$, then there are $n_0, r, \delta$ such that for those $({w_1}, {w_2})$ with  $ \rho_{r}({w_1}, {w_2})<\delta $, one has
\[d\left(\mathcal{P}_{s, s+n\mathcal{T}}^{*} \delta_{{w_1}}, \mathcal{P}_{s, s+n\mathcal{T}}^{*} \delta_{{w_2}}\right) \leq \alpha_{2} d({w_1}, {w_2})\] for all $n>n_0$ and $\beta\in(0,1).$
Now fixing $L$ as in Lemma \ref{contractlemma1}, then by Lemma \ref{contractlemma3}, for the fixed $L, \delta, r$, there is some $n_1$ such that for $n>n_1$, there exist $\beta, \alpha_3$ such that the implication in Lemma \ref{contractlemma3} holds true.  Now for fixed $\delta, \beta, r, L$, there is $\alpha_1$ such that the implication of Lemma \ref{contractlemma1} holds true. So the conclusion follows by taking $N>\max\{n_0, n_1\}$ and $\alpha = \max\{\alpha_1,\alpha_2,\alpha_3\}<1. $
\end{proof}

\section{Weak law of large numbers and central limit theorem}\label{sectionlimittheorem}
In this section, we prove  limit theorems Theorem \ref{thmlimittheorem} for the solution process of the Navier-Stokes equation \eqref{NS}. The essential ingredients to obtain the results are the contraction \eqref{eqcontraction} of the Markov transition operator $\mathcal{P}_{s ,t}^*$ on $\mathcal{P}(H)$, together with a Lyapunov structure. To the best of our knowledge, there are little known results of limit theorems for continuous time inhomogeneous Markov processes. Our first result is a weak law of large numbers for the continuous time inhomogeneous solution process which is proved in subsection \ref{subsectionWLLNinhomo}.  However we are not able to prove a central limit theorem for the continuous time inhomogeneous solution process. In the theory of dynamical systems, when a system has a periodic structure, one usually consider its dynamics on the Poincare section. Therefore following the same idea, we turn to consider the restriction of the time inhomogeneous solution process to periodic times, which results in a discrete time homogeneous Markov chain. And in subsection \eqref{subsectionlimitthm} we show that both the weak law of large numbers and the central limit theorem hold for this Markov chain.

\subsection{Weak law of large numbers for the inhomogeneous solution process}\label{subsectionWLLNinhomo}
In this subsection, we will prove the weak law of large numbers for the continuous time inhomogeneous solution process of Navier-Stokes equation \eqref{NS}. To be more specific, we will show the following theorem. Recall that $\mathbb{S}_{\mathcal{T}} = \mathbb{R}/\mathcal{T}\mathbb{Z}$, and we denote by $\lambda$ the normalized Lebesgue measure on $\mathbb{S}_{\mathcal{T}}$.
\begin{theorem}\label{thm6.8}
Assume $G_1<1/c_0$, $A_{\infty} = H$.  Then for any $s\in\mathbb{R}$, $w_0\in H$ and $\psi\in\mathrm{Lip}_{\rho}(H)$,
\[\lim_{T\rightarrow\infty}\frac{1}{T}\int_{0}^T\psi(w_{s, s+t}(w_0))dt = \int_{Y}\int_{H}\psi(w)\mu_s(dw)\lambda(ds)\]
in probability.
\end{theorem}
\begin{proof}
Let $Y_T = \frac{1}{T}\int_{0}^T\psi(w_{s, s+t}(w_0))dt$ and $v_* =  \int_{Y}\int_{H}\psi(w)\mu_s(dw)\lambda(ds)$. By Chebyshev's inequality, to show Theorem \ref{thm6.8},  it suffices to show that $\lim_{T\rightarrow\infty}\mathbf{E}Y_T = v_*$ and
$\lim_{T\rightarrow\infty}\mathbf{E}Y_T^2 = v_*^2$. We first proof the theorem for bounded $\psi\in\mathrm{Lip}_{\rho}(H)$.  Note that by mixing property of the periodic invariant measure $\mu_s$ in Theorem \ref{ergodicmixing} and the Monge-Kantorovich daulity, we have
\begin{align}\notag
&\left|\frac{1}{T}\int_{0}^{T}\mathcal{P}_{s, s+t}\psi(w_0) dt- \frac{1}{T}\int_{0}^{T}\int_{H}\psi(w)\mu_{s+t}(dw)dt\right|\\\notag
&\leq \mathrm{Lip_{\rho}}(\psi)\sup_{\mathrm{Lip_{\rho}}(\phi)\leq 1}\left|\frac{1}{T}\int_{0}^{T}\mathcal{P}_{s, s+t}\phi(w_0) dt- \frac{1}{T}\int_{0}^{T}\int_{H}\phi(w)\mu_{s+t}(dw)dt\right|\\\label{continuousWLLN01}
&\leq \mathrm{Lip_{\rho}}(\psi) \frac1T\int_{0}^{T} \rho(\mathcal{P}_{s, s+t}^*\delta_{w_0}, \mathcal{P}_{s, s+t}^*\mu_s)dt\leq C\mathrm{Lip_{\rho}}(\psi)\rho(\delta_{w_0}, \mu_s)\frac1T\int_{0}^{T}e^{-\gamma t}dt\rightarrow 0 \quad (T\rightarrow\infty).
\end{align}
For $\tau\in\mathbb{R}$, let $\nu_{\tau} =\int_{H}\psi(w)\mu_{\tau}(dw) $. By Theorem \ref{fixedpoint}, there is a constant $C>0$ such that $\int_H \exp\left(2\kappa\eta\|w\|^2\right)\mu_s(dw)\leq C$ for all $s\in\mathbb{R}$. Note that $\rho(0, w)\leq \|w\|e^{\eta\|w\|^2}\leq C e^{\kappa\eta\|w\|^2}$, so
\[|\nu_{\tau}|\leq |\psi(0)|+\mathrm{Lip}_{\rho}(\psi) \int_{H}\rho(0, w)\mu_s(dw)\leq|\psi(0)|+C\mathrm{Lip}_{\rho}(\psi). \]
The boundedness of $\nu_{\tau}$ implies that
\begin{align}\label{continuousWLLN02}
\frac{1}{T}\int_{0}^{T}\int_{H}\psi(w)\mu_{s+t}(dw)dt = \frac{1}{T}\int_{s}^{s+T}\nu_{\tau}d\tau =  \frac{1}{T}\int_{0}^{s+T}\nu_{\tau}d\tau -  \frac{1}{T}\int_{0}^{s}\nu_{\tau}d\tau\rightarrow v_* \quad (T\rightarrow\infty),
\end{align}
since
\begin{align*}
\lim_{T\rightarrow\infty} \frac{1}{T}\int_{0}^{s+T}\nu_{\tau}d\tau = \lim_{T\rightarrow\infty}\frac{1}{T}\int_{0}^{T}\nu_{\tau}d\tau = \lim_{T\rightarrow\infty}\left(\frac{1}{T}\int_{0}^{k\mathcal{T}}\nu_{\tau}d\tau + \frac{1}{T}\int_{k\mathcal{T}}^{T}\nu_{\tau}d\tau\right) = v_*
\end{align*}
by the representation $T = k\mathcal{T}+\beta$ for unique  $k\in\mathbb{N}$ and $\beta\in [0, \mathcal{T})$.
It follows from \eqref{continuousWLLN01} and \eqref{continuousWLLN02} that
\begin{align}\label{thm6.8-0}
\lim_{T\rightarrow\infty}\mathbf{E}Y_T =\lim_{T\rightarrow\infty}\frac{1}{T}\int_{0}^{T}\mathcal{P}_{s, s+t}\psi(w_0) dt= v_*.
\end{align}
Note that the convergence \eqref{thm6.8-0} is valid for all $\psi\in\mathrm{Lip}_{\rho}(H)$ since we did not use boundedness of $\psi$ in the proof.

\vskip0.05in

By Markov property, we can rewrite
\begin{align}
\mathbf{E}Y_T^2 \notag
&= \frac{2}{T^2}\mathbf{E}\int_{0}^T\int_{0}^t\psi(w_{s, s+t}(w_0))\psi(w_{s, s+\tau}(w_0))d\tau dt\\\notag
& = \frac{2}{T^2}\int_{0}^T\int_{0}^t\mathbf{E}\psi(w_{s, s+\tau}(w_0))\mathcal{P}_{s+\tau, s+t}\psi(w_{s, s+\tau}(w_0))d\tau dt \\
&= \frac{2}{T^2}\int_{0}^T\int_{0}^t  \mathcal{P}_{s, s+\tau}\left(\psi\mathcal{P}_{s+\tau, s+t}\psi)(w_0) \right)d\tau dt.\label{YTsquare}
\end{align}
We first show that the convergence in \eqref{thm6.8-0} is uniform on any compact set,  that is for any $\varepsilon>0$, and compact $K\subset H$, there is $T_0>0$ such that for all $T\geq T_0$,
\begin{align}\label{thm6.8-1}
\sup_{w\in K}\left|\frac{1}{T}\int_{0}^{T}\mathcal{P}_{s, s+t}\psi(w) dt - v_*\right|<\varepsilon.
\end{align}
Indeed, $\frac{1}{T}\int_{0}^{T}\mathcal{P}_{s, s+t}\psi(w) dt$ is uniformly bounded on $K$ since $\psi$ is  bounded. And it is also equicontinuous since by the Monge-Kantorovich daulity
\begin{align*}
&\left|\frac{1}{T}\int_{0}^{T}\mathcal{P}_{s, s+t}\psi(w_1) dt - \frac{1}{T}\int_{0}^{T}\mathcal{P}_{s, s+t}\psi(w_2) dt\right| \leq\mathrm{Lip}_{\rho}(\psi) \frac{1}{T}\int_{0}^{T}\rho(\mathcal{P}_{s, s+t}^*\delta_{w_1},\mathcal{P}_{s, s+t}^*\delta_{w_2})dt\\
&C\mathrm{Lip}_{\rho}(\psi)\rho(w_1, w_2)\frac1T\int_{0}^{T}e^{-\gamma t}dt \leq C\mathrm{Lip}_{\rho}(\psi)\rho(w_1, w_2).
\end{align*}
So \eqref{thm6.8-1} follows by Arzela-Ascoli theorem. From the proof  we see that the estimate \eqref{thm6.8-1} is also uniform with respect to the initial time $s$.

\vskip0.05in

Since for any $w\in H$ and  bounded $\psi\in\mathrm{Lip}_{\rho}(H)$, we have
\[\lim_{T\rightarrow\infty}\frac{1}{T}\int_{0}^{T}\mathcal{P}_{s, s+t}\psi(w) dt=\int_{Y}\int_{H}\psi(w)\mu_s(dw)\lambda(ds), \]
hence the family of probability measures $\frac{1}{T}\int_{0}^{T}\mathcal{P}_{s, s+t}^*\delta_wdt$ converges weakly to  $\int_{Y}\mu_s\lambda(ds)$.  Therefore $\frac{1}{T}\int_{0}^{T}\mathcal{P}_{s, s+t}^*\delta_wdt$ is tight, that is for $\varepsilon>0$, there are compact subset $K$ and $T_1>0$ such that for $T\geq T_1$,
\begin{align}\label{thm6.8tight}
\frac{1}{T}\int_{0}^{T}\mathcal{P}_{s, s+t}^*\delta_w(K^c)dt = \frac{1}{T}\int_{0}^{T}\mathcal{P}_{s, s+t}(w, K^c)dt <\varepsilon.
\end{align}
The proof of the following lemma will be given at the end of this subsection.
\begin{lemma}\label{thm6.8-lemma01}
For any $\varepsilon>0$, $w_0\in H$, there is a $T_1>0$ such that for $T\geq T_1$,
\[\left|\frac{2}{T^2}\int_{0}^T\int_{0}^t  \mathcal{P}_{s, s+\tau}\Big(\psi\big(\mathcal{P}_{s+\tau, s+t}\psi-v_*\big)\Big)(w_0) d\tau dt\right|<\varepsilon.\]
\end{lemma}
Observe that $\alpha(t) : = \frac1t\int_{0}^t\mathcal{P}_{s, s+r}\psi(w_0)dr\rightarrow v_*$ as $t\rightarrow \infty$ by \eqref{thm6.8-0}. Hence $\lim_{T\rightarrow\infty} \frac{2}{T^2}\int_0^T t\alpha(t)dt = v_*$.  Indeed, for $\varepsilon>0$, there is $T_0>0$ such that for $t\geq T_0$, $|\alpha(t)- v_*|<\varepsilon$ and $t|\alpha(t)- v_*|< M$ for $t\in[0, T_0]$. Therefore for $T>T_0$,
\begin{align}\notag
\left| \frac{2}{T^2}\int_0^T t\alpha(t)dt  - v_*\right|& \leq \frac{2}{T^2}\int_0^T t\left| \alpha(t)  - v_*\right|dt\\\label{continuousWLLN03}
&\leq \frac{2}{T^2}\int_{0}^{T_0}t|\alpha(t)- v_*|dt + \frac{2}{T^2}\int_{T_0}^{T}t\varepsilon dt\leq \frac{MT_0^2}{T^2}+\frac{T^2-T_0^2}{T^2}\varepsilon<\varepsilon
\end{align}
by choosing  large enough $T$.

\vskip0.05in

Combining \eqref{continuousWLLN03} with Lemma \ref{thm6.8-lemma01}, one has
\begin{align*}
\lim_{T\rightarrow\infty}\frac{2}{T^2}\int_{0}^T\int_{0}^t  \mathcal{P}_{s, s+\tau}\left(\psi\mathcal{P}_{s+\tau, s+t}\psi)(w_0) \right)d\tau dt &= \lim_{T\rightarrow\infty}\frac{2}{T^2}\int_{0}^T\int_{0}^t  \mathcal{P}_{s, s+\tau}\left(\psi v_*\right)d\tau dt\\
& = v_*\lim_{T\rightarrow\infty} \frac{2}{T^2}\int_0^T t\alpha(t)dt = v_*^2.
\end{align*}
In view of \eqref{YTsquare}, we see that $\lim_{T\rightarrow\infty}\mathbf{E}Y_T^2 = v_*^2$, which completes the proof of Theorem \ref{thm6.8} for bounded Lipschitz observables.

\vskip0.05in

Now for $\psi\in \mathrm{Lip}_{\rho}(H)$, we define for $L>1$
\begin{align}\label{psiL}
\psi_{L}(w):=\left\{\begin{array}{ll}
\psi(w), & \text { when }|\psi(w)| \leq L, \\
L, & \text { when } \psi(w)>L, \\
-L, & \text { when } \psi(w)<-L.
\end{array}\right.
\end{align}
set $v_*^L =  \int_{Y}\int_{H}\psi_L(w)\mu_s(dw)\lambda(ds)$, $\psi^L = |\psi - \psi_L|$.
Again by Theorem \ref{fixedpoint}, there is a constant $C>0$ such that $\int_H \exp\left(2\kappa\eta\|w\|^2\right)\mu_s(dw)\leq C$ for all $s\in\mathbb{R}$. Note that $\rho(0, w)\leq \|w\|e^{\eta\|w\|^2}\leq C e^{\kappa\eta\|w\|^2}$, so
\[\int_{Y}\int_{H}|\psi(w)|\mu_s(dw)\lambda(ds)\leq |\psi(0)|+\mathrm{Lip}_{\rho}(\psi) \int_{Y}\int_{H}\rho(0, w)\mu_s(dw)\lambda(ds)\leq|\psi(0)|+C\mathrm{Lip}_{\rho}(\psi). \]
Hence by the dominated convergence theorem,
\begin{align}\label{unbounded3}
\lim_{L\rightarrow\infty} \int_{Y}\int_{H}\psi^L(w)\mu_s(dw)\lambda(ds) = 0.
\end{align}
Since $\psi_L$ is bounded, by Theorem \ref{thm6.8} for bounded Lipschitz observables that we have proved, one has
\begin{align}\label{unbounded2}
\lim_{T\rightarrow\infty}\frac{1}{T}\int_{0}^T\psi_L(w_{s, s+t}(w_0))dt  = v_*^L.
\end{align}
By combining \eqref{unbounded3} with  \eqref{thm6.8-0}, we have that
\begin{align*}
\mathbf{E} \frac{1}{T}\int_{0}^T \psi^L(w_{s. s+t})dt \leq
&\left|\mathbf{E} \frac{1}{T}\int_{0}^T \psi^L(w_{s. s+t})dt -  \int_{Y}\int_{H}\psi^L(w)\mu_s(dw)\lambda(ds)\right| \\
&+ \int_{Y}\int_{H}\psi^L(w)\mu_s(dw)\lambda(ds) \rightarrow 0 \quad (T\rightarrow \infty, L\rightarrow \infty).
\end{align*}
Therefore by Markov inequality,
\begin{align}\label{unbounded1}
\lim_{L\rightarrow \infty}\lim_{T\rightarrow \infty} \frac{1}{T}\int_{0}^T \psi^L(w_{s. s+t})dt \rightarrow 0 \quad \text{in probability}.
\end{align}
Now observe that
\begin{align*}
&\left|\frac{1}{T}\int_{0}^T\psi(w_{s, s+t})dt-v_*\right|\\
&\leq \left|\frac{1}{T}\int_{0}^T\psi(w_{s, s+t})dt-\frac{1}{T}\int_{0}^T\psi_L(w_{s, s+t})dt\right| + \left|\frac{1}{T}\int_{0}^T\psi_L(w_{s, s+t})dt - v_*^L\right| + |v_*^L - v_*|\\
&\leq \frac{1}{T}\int_0^T\psi^L(w_{s, s+t}) dt + \left|\frac{1}{T}\int_{0}^T\psi_L(w_{s, s+t})dt - v_*^L\right| + \int_{Y}\int_{H}\psi^L(w)\mu_s(dw)\lambda(ds),
\end{align*}
whence $\lim_{T\rightarrow\infty}\frac{1}{T}\int_{0}^T\psi(w_{s, s+t})dt=v_*$  in probability by \eqref{unbounded3}-\eqref{unbounded1}.  The proof of Theorem \ref{thm6.8} for any Lipschitz observable is then completed. It remains to supply a proof for Lemma \ref{thm6.8-lemma01}.
\begin{proof}[Proof of Lemma \ref{thm6.8-lemma01}]
Let $K$ be the compact set and $T_1>0$ as in \eqref{thm6.8tight}. Note
\begin{align*}
\left|\frac{2}{T^2}\int_{0}^T\int_{0}^t  \mathcal{P}_{s, s+\tau}\Big(\psi\big(\mathcal{P}_{s+\tau, s+t}\psi-v_*\big)\Big)(w_0) d\tau dt\right|\leq K_1 + K_2,
\end{align*}
where
\[K_1 =  \left|\frac{2}{T^2}\int_{0}^T\int_{0}^t  \int_{H}\Big(\psi\big(\mathcal{P}_{s+\tau, s+t}\psi-v_*\big)\Big)\mathbb{I}_{K}(w) \mathcal{P}_{s, s+\tau}^*\delta_{w_0}(dw)d\tau dt\right|,\]
\[K_2 = \left|\frac{2}{T^2}\int_{0}^T\int_{0}^t  \int_{H}\Big(\psi\big(\mathcal{P}_{s+\tau, s+t}\psi-v_*\big)\Big)\mathbb{I}_{K^c}(w) \mathcal{P}_{s, s+\tau}^*\delta_{w_0}(dw)d\tau dt\right|.\]
We first estimate $K_1$.  For the compact set $K$, choose $T_0>T_1$ as in \eqref{thm6.8-1}, then for $T>T_0$
\begin{align*}
K_1 &= \left|\frac{2}{T^2}\int_{0}^T\int_{\tau}^T  \int_{H}\Big(\psi\big(\mathcal{P}_{s+\tau, s+t}\psi-v_*\big)\Big)\mathbb{I}_{K}(w) \mathcal{P}_{s, s+\tau}^*\delta_{w_0}(dw)dtd\tau \right|\\
&= \left|\frac{2}{T^2}\int_{0}^T(T-\tau)\int_{H}\psi(w)\left(\frac{1}{T-\tau}\int_{0}^{T-\tau}\big(\mathcal{P}_{s+\tau, s+\tau+r}\psi-v_*\big)(w)\right)\mathbb{I}_{K}(w) dr\mathcal{P}_{s, s+\tau}^*\delta_{w_0}(dw)d\tau \right|\\
&= \left|\frac{2}{T^2}\int_{0}^Tt\int_{H}\psi(w)\left(\frac{1}{t}\int_{0}^{t}\big(\mathcal{P}_{s+T-t, s+T-t+r}\psi-v_*\big)(w)\right)\mathbb{I}_{K}(w) dr\mathcal{P}_{s, s+T-t}^*\delta_{w_0}(dw)dt \right|\\
&\leq K_{11} + K_{12},
\end{align*}
where
\begin{align*}
K_{11}& = \left|\frac{2}{T^2}\int_{0}^{T_0}t\int_{H}\psi(w)\left(\frac{1}{t}\int_{0}^{t}\big(\mathcal{P}_{s+T-t, s+T-t+r}\psi-v_*\big)(w)\right)\mathbb{I}_{K}(w) dr\mathcal{P}_{s, s+T-t}^*\delta_{w_0}(dw)dt \right|\\
&\leq \frac{2}{T^2}\int_{0}^{T_0}t\|\psi\|_{\infty}\big(\|\psi\|_{\infty}+v_*\big)d\tau \leq 2\left(\frac{T_0}{T}\right)^2\|\psi\|_{\infty}\big(\|\psi\|_{\infty}+v_*\big),
\end{align*}
And
\begin{align*}
K_{12}& = \left|\frac{2}{T^2}\int_{T_0}^{T}t\int_{H}\psi(w)\left(\frac{1}{t}\int_{0}^{t}\big(\mathcal{P}_{s+T-t, s+T-t+r}\psi-v_*\big)(w)\right)\mathbb{I}_{K}(w) dr\mathcal{P}_{s, s+T-t}^*\delta_{w_0}(dw)dt \right|\\
&\leq\frac{2}{T^2}\|\psi\|_{\infty}\int_{T_0}^Tt\sup_{w\in K} \left|\frac{1}{t}\int_{0}^{t}\mathcal{P}_{s+T-t, s+T-t+r}\psi(w)dr-v_*\right|dt\leq \varepsilon\|\psi\|_{\infty}.
\end{align*}
For $K_2$, we apply \eqref{thm6.8tight},
\begin{align*}
K_2&\leq \left|\frac{4}{T^2}\|\psi\|_{\infty}^2\int_{0}^T\int_{0}^t  \int_{H}\mathbb{I}_{K^c}(w) \mathcal{P}_{s, s+\tau}^*\delta_{w_0}(dw)d\tau dt\right|\\
&\leq \frac{4}{T^2}\|\psi\|_{\infty}^2\left(\int_{0}^{T_0}\int_{0}^t \mathcal{P}_{s, s+\tau}(w_0,{K^c}) d\tau dt+\int_{T_0}^Tt\left(\frac1t\int_{0}^t \mathcal{P}_{s, s+\tau}(w_0,{K^c}) d\tau\right) dt\right)\\
&\leq \frac{4}{T^2}\|\psi\|_{\infty}^2\left(\frac12 T_0^2+\frac{T^2}{2}\varepsilon\right).
\end{align*}
Combining the above estimates and choosing large enough $T$, we then arrive at Lemma \ref{thm6.8-lemma01}.
\end{proof}
The proof of Theorem \ref{thm6.8} is then completed.
\end{proof}
\begin{remark}
Note that the convergence \eqref{thm6.8-0} is actually a special corollary of the Von Neumann Theorem (formula 6.28) mentioned  in \cite{DD08}.
\end{remark}

\subsection{Limit theorems for the restriction of the inhomogeneous solution process
to periodic times}\label{subsectionlimitthm}
In this subsection, we will prove the weak law of large numbers and central limit theorem for the time homogeneous Markov chain  obtained by restricting the solution process starting at $s\in\mathbb{R}$ at periodic times $\{s+n\mathcal{T}\}_{n\geq 0}$.

\vskip0.05in

For $w_0\in H, s\in\mathbb{R}$,  recall that  $w_{s, t}(w_0) = w(t, \omega; s, w_0)$ is the solution of  system \eqref{NS} starting from $w_0$ at time $s$. Let $\{\mu_s\}_{s\in\mathbb{R}}$  be the unique periodic invariant measure of the transition operator $\mathcal{P}_{s, t}$. For integer $k\geq 0$,  $\mathfrak{F}_{k} = \mathcal{F}_{s+k\mathcal{T}}$. For $0\leq i\leq j$, and  any Borel set $A\subset H$, we have transition probabilities $\mathfrak{P}_{i, j}(w_0, A): = \mathcal{P}_{s+i\mathcal{T},s+j\mathcal{T}}(w_0, A) = \mathbf{P}(w(s+j\mathcal{T}; s+i\mathcal{T}, w_0)\in A)$  of the Markov chain $\{w_{s, s+k\mathcal{T}}(w_0)\}_{k\geq 0} $ such that
\begin{enumerate}
\item $\mathbf{P}(w_{s, s+j\mathcal{T}}(w_0)\in A|\mathfrak{F}_i) = \mathfrak{P}_{i, j}(w_{s, s+i\mathcal{T}}(w_0), A)$, $\mathfrak{P}_{i, j}(w_0, A) = \delta_{w_0}(A)$;
\item It is time homogeneous,  i.e. $\mathfrak{P}_{i, j} = \mathfrak{P}_{0, j - i}$ since $\mathcal{P}_{s+i\mathcal{T},s+j\mathcal{T}} = \mathcal{P}_{s, s+(j-i)\mathcal{T}}$ by periodicity.
\end{enumerate}
Let $\mathfrak{P}_{k} = \mathcal{P}_{s, s+k\mathcal{T}}$. Then $\mathfrak{P}_{k}$ is a Markov semigroup and $\mu_s$ is the unique invariant measure. What's more, by Theorem \ref{fixedpoint}, one has $\rho(\mathfrak{P}_{k}^{*}\mu, \mu_s)\leq Ce^{-\gamma k\mathcal{T}}\rho(\mu,\mu_s)$ for any $\mu\in\mathcal{P}(H)$.

\vskip0.05in

The main result of this subsection is the following limit theorem for the Markov chain $\{w_{s, s+k\mathcal{T}}\}_{k\geq 0}.$ We also denote the integration of a function $\psi$ with respect to the probability measure $\mu$ by $\langle \mu, \psi\rangle =  \int_{H}\psi(z)\mu(dz)$.
\begin{theorem}
Assume $G_1<1/c_0$, $A_{\infty} = H$. Then for any $s\in\mathbb{R}$, $w_0\in H$ and any Lipschitz function $\psi: H \rightarrow \mathbb{R}$, we have the following:
\begin{enumerate}
\item Weak law of large numbers (WLLN),
\begin{align}\label{WLLN}
\lim_{N\rightarrow\infty}\frac{1}{N}\sum_{k=0}^{N-1}\psi\left(w_{s, s+k\mathcal{T}}(w_0)\right) = \langle \mu_s, \psi\rangle
\end{align}
in probability.
\item Let $\widetilde{\psi}_s(w_0): = \psi(w_0) - \langle \mu_s, \psi\rangle$. Then the central limit theorem (CLT) holds:
\begin{align}\label{CLT}
\lim_{N\rightarrow\infty}\mathbf{P}\left(\frac{1}{\sqrt{N}}\sum_{k=0}^{N-1}\widetilde{\psi}_s\left(w_{s, s+k\mathcal{T}}(w_0)\right)<\xi\right) = \Phi_{\sigma}(\xi), \forall \xi\in\mathbb{R},
\end{align}
where $\Phi_{\sigma}(\cdot)$ is the distribution function of a normal random variable with variance $\sigma^2$, and
\[\sigma^2 = \lim_{N\rightarrow\infty}\frac1N\mathbf{E}\left[\sum_{k=0}^{N-1}\widetilde{\psi}_s\left(w_{s, s+k\mathcal{T}}(w_0)\right)\right]^2,\]
where $\sigma: = \sigma(s)$ is $\mathcal{T}$-periodic in $s$.
\end{enumerate}
\end{theorem}
\subsubsection{Proof of WLLN}
Let $Y_N = \frac{1}{N}\sum_{k=0}^{N-1}\psi\left(w_{s, s+k\mathcal{T}}(w_0)\right) $ and $v_s = \langle \mu_s, \psi\rangle$. Then by Chebyshev's inequality, it suffices to show that $\lim_{N\rightarrow\infty}\mathbf{E}Y_N = v_s$ and $\lim_{N\rightarrow\infty}\mathbf{E}Y_N^2 = v_s^2$. Assume first that $\psi\in\mathrm{Lip}_{\rho}(H)$ is bounded. Then
\begin{align}\label{convergentexpectation}
\mathbf{E}Y_N = \frac{1}{N}\sum_{k=0}^{N-1} \mathfrak{P}_{k}\psi(w_0) =  \frac{1}{N}\sum_{k=0}^{N-1} \langle\mathfrak{P}_{k}^{*}\delta_{w_0},\psi\rangle \rightarrow v_s,
\end{align}
 since the average $ \frac{1}{N}\sum_{k=0}^{N-1} \mathfrak{P}_{k}^{*}\delta_{w_0}$ converges to $\mu_s$ in the Wasserstein metric. Indeed, for any $\psi\in\mathrm{Lip}_{\rho}(H)$ with $\mathrm{Lip}_{\rho}(\psi)\leq 1$, and $\mu\in\mathcal{P}_1(H)$, by Theorem \ref{fixedpoint}, the invariance of $\mu_s$, and the Monge-Kantorovich daulity, one has
\begin{align*}
\left|\frac{1}{N}\sum_{k=0}^{N-1} \langle \mathfrak{P}_{k}^{*}\mu, \psi\rangle - \langle\mu_s,\psi\rangle\right|& = \left|\frac{1}{N}\sum_{k=0}^{N-1} \langle \mathfrak{P}_{k}^{*}\mu, \psi\rangle -\frac{1}{N}\sum_{k=0}^{N-1} \langle \mathfrak{P}_{k}^{*}\mu_s, \psi\rangle\right|\\
&\leq \frac{1}{N}\sum_{k=0}^{N-1}\left| \langle \mathfrak{P}_{k}^{*}\mu, \psi\rangle - \langle \mathfrak{P}_{k}^{*}\mu_s, \psi\rangle\right|\\
&\leq \frac{1}{N}\sum_{k=0}^{N-1}Ce^{-\gamma k\mathcal{T}}\rho(\mu, \mu_s)\\
&\leq \frac{1}{N}\sum_{k=0}^{N-1}\rho(\mathfrak{P}_{k}^{*}\mu, \mathfrak{P}_{k}^{*}\mu_s)\leq \frac{C}{N(1-e^{-\gamma\mathcal{T}})}\rho(\mu, \mu_s).
\end{align*}
In addition, the convergence \eqref{convergentexpectation} is uniform on any compact set by Arzelà–Ascoli theorem since $\psi_{N}(w): =\frac{1}{N}\sum_{k=0}^{N-1} \mathfrak{P}_{k}\psi(w)  $ is uniformly bounded by boundedness of $\psi$ and it is equicontinuous by noting that
\begin{align*}
|\psi_{N}(w_1) - \psi_{N}(w_2)| \leq \frac{1}{N}\sum_{k=0}^{N-1}\left|\mathfrak{P}_{k}\psi(w_1)-\mathfrak{P}_{k}\psi(w_2)\right|&
\leq \frac{1}{N}\sum_{k=0}^{N-1}Ce^{-\gamma k\mathcal{T}}\mathrm{Lip}_{\rho}(\psi)\rho(w_1, w_2)\\
&\leq C\mathrm{Lip}_{\rho}(\psi)\rho(w_1, w_2).
\end{align*}
Hence for any $\varepsilon>0$ and compact set $K\subset H$, there is some $N_0$ such that for all $N\geq N_0$, one has
\begin{align}\label{uniformcompact}
\sup_{w\in K} \left|\frac{1}{N}\sum_{k=0}^{N-1} \mathfrak{P}_{k}\psi(w)  - v_s\right|<\varepsilon.
\end{align}
Note that
\begin{align}\notag
\mathbf{E}Y_N^2 &= \frac{2}{N^2}\sum_{m=0}^{N-1}\sum_{n=0}^{m}\mathbf{E}\left[\psi\left(w_{s, s+m\mathcal{T}}(w_0)\right)\psi\left(w_{s, s+n\mathcal{T}}(w_0)\right)\right]-\frac{1}{N^2}\sum_{m=0}^{N-1}\mathbf{E}\left[\psi(w_{s, s+m\mathcal{T}}(w_0))^2\right]\\\notag
& = \frac{2}{N^2}\sum_{m=0}^{N-1}\sum_{n=0}^{m}\mathbf{E}\left[\psi\left(w_{s, s+n\mathcal{T}}(w_0)\right)\mathfrak{P}_{m-n}\psi\left(w_{s, s+n\mathcal{T}}(w_0)\right)\right]-\frac{1}{N^2}\sum_{m=0}^{N-1}\mathfrak{P}_{m}\psi^2(w_0)\\\label{discreteWLLN01}
&= \frac{2}{N^2}\sum_{m=0}^{N-1}\sum_{n=0}^{m}\mathfrak{P}_n(\psi\mathfrak{P}_{m-n}\psi)(w_0)-\frac{1}{N^2}\sum_{m=0}^{N-1}\mathfrak{P}_{m}\psi^2(w_0).
\end{align}
The second term on the last equality \eqref{discreteWLLN01} tends to $0$ since $ \frac{1}{N}\sum_{m = 0}^{N-1} \mathfrak{P}_{m}^{*}\delta_{w_0}$ converges to $\mu_s$ weakly.  To estimate the first term, we need the following lemma.
\begin{lemma}\label{WLLN1}
For any $\varepsilon>0$, there is $N_0>0$ such that for all $N\geq N_0$,
\[\left|\frac{2}{N^2}\sum_{m=0}^{N-1}\sum_{n=0}^{m}\mathfrak{P}_n(\psi(\mathfrak{P}_{m-n}\psi-v_s))(w_0)\right|<\varepsilon.\]
\end{lemma}
The proof of this lemma will be given at the end of this subsection.
Note that
\begin{align}\label{discreteWLLN02}
\lim_{N\rightarrow\infty}\frac{2}{N^2}\sum_{m=0}^{N-1}\sum_{n=0}^{m}\mathfrak{P}_n\psi(w_0) = v_s.
\end{align}
Indeed,
\[\frac{2}{N^2}\sum_{m=0}^{N-1}\sum_{n=0}^{m}\mathfrak{P}_n\psi(w_0) = \frac{2}{N^2}\sum_{m=0}^{N-1}m\left(\frac{1}{m}\sum_{n=0}^{m}\mathfrak{P}_n\psi(w_0)\right), \]
and $\lim_{m\rightarrow\infty}\frac{1}{m}\sum_{n=0}^{m}\mathfrak{P}_n\psi(w_0) = v_s.$ Therefore, if we set $\alpha_m = \frac{1}{m}\sum_{n=0}^{m}\mathfrak{P}_n\psi(w_0)$, then for any $\varepsilon>0, $ there is some $M_0>0$ such that for any $m\geq M_0$, $|\alpha_m - v_s|<\varepsilon$. Hence
\begin{align*}
\left|\frac{2}{N^2}\sum_{m=0}^{N-1}m\alpha_m - v_s\right|& \leq \frac{2}{N^2}\sum_{m=0}^{N-1}m|\alpha_m - v_s| + \frac{|v_s|}{N}\\
&\leq  \frac{2}{N^2}\sum_{m=0}^{M_0-1}m|\alpha_m - v_s| + \frac{2}{N^2}\sum_{m=M_0}^{N-1}m\varepsilon + \frac{|v_s|}{N}<\varepsilon,
\end{align*}
by taking $N$ large enough. Combining \eqref{discreteWLLN01} and \eqref{discreteWLLN02} with Lemma \ref{WLLN1}, one has the equality \eqref{WLLN} for bounded Lipschitz observables. For unbounded Lipschitz function $\psi$,  define $\psi_L$ as \eqref{psiL}.
Let $v_s^L = \langle\mu_s,\psi_L\rangle$ and $\psi^{L} = |\psi-\psi_L|$. Then
\begin{align}\notag
&\left|\frac{1}{N}\sum_{k=0}^{N-1}\psi\left(w_{s, s+k\mathcal{T}}(w_0)\right)- v_s\right| \\\label{discreteWLLN03}
&\leq \frac{1}{N}\sum_{k=0}^{N-1}\psi^{L}(w_{s, s+k\mathcal{T}}(w_0))+\left|\frac{1}{N}\sum_{k=0}^{N-1}  \psi_L\left(w_{s, s+k\mathcal{T}}(w_0)\right)  - v_s^{L} \right|+\left|v_s^L - v_s\right|.
\end{align}
Since $\mu_s\in\mathcal{P}_1(H)$, we have $\langle\mu_s, |\psi|\rangle\leq \langle\mu_s, |\psi(0)|+\rho(0, w)\mathrm{Lip}_{\rho}(\psi)\rangle<\infty$. So $\lim_{L\rightarrow\infty}\left|v_s^L - v_s\right|=0$ by the dominated convergence theorem. Also the second term of \eqref{discreteWLLN03} converges to $0$ in probability by the already proved \eqref{WLLN} for bounded Lipschitz observables.  It follows from Theorem \ref{fixedpoint}, the invariance of $\mu_s$, and the Monge-Kantorovich daulity that
\begin{align*}
\mathbf{E} \frac{1}{N}\sum_{k=0}^{N-1}\psi^{L}(w_{s, s+k\mathcal{T}}(w_0))  &\leq \left|\mathbf{E} \frac{1}{N}\sum_{k=0}^{N-1}\psi^{L}(w_{s, s+k\mathcal{T}}(w_0)) - \langle\mu_s, \psi^L\rangle\right|+ \langle\mu_s, \psi^L\rangle\\
& \leq\frac{1}{N}\sum_{k=0}^{N-1}\left|\langle\mathfrak{P}_{k}^*\delta_{w_0}, \psi^L\rangle -  \langle\mathfrak{P}_{k}^*\mu_s, \psi^L\rangle\right|+\langle\mu_s, \psi^L\rangle\\
&\leq \frac{1}{N}\sum_{k=0}^{N-1}Ce^{-\gamma k \mathcal{T}}\mathrm{Lip}_{\rho}(\psi)\rho(\delta_{w_0}, \mu_s)+\langle\mu_s, \psi^L\rangle.
\end{align*}
Hence
\[\lim_{L\rightarrow\infty}\lim_{N\rightarrow\infty}\mathbf{E} \frac{1}{N}\sum_{k=0}^{N-1}\psi^{L}(w_{s, s+k\mathcal{T}}(w_0)) = 0.\]
As a result, by Markov inequality, $ \frac{1}{N}\sum_{k=0}^{N-1}\psi^{L}(w_{s, s+k\mathcal{T}}(w_0)) $ converges to $0$ in probability as $N\rightarrow\infty$ and $L\rightarrow\infty$. Combining this with \eqref{discreteWLLN03}, we obtain \eqref{WLLN} for unbounded Lipschitz observables. Now we supply a proof for Lemma \ref{WLLN1}.
\begin{proof}[Proof of Lemma \ref{WLLN1}]
By mixing property of $\mu_s$, $\rho(\mathfrak{P}_{k}^{*}\mu, \mu_s)\leq Ce^{-\gamma k\mathcal{T}}\rho(\mu,\mu_s)$, hence the family of probability measures $\{\mathfrak{P}_{k}^{*}\delta_{w_0}\}_{k\geq 0}$ is tight. Then for any $\varepsilon>0$, there is a compact set $K$ and $N_0>0$ such that for all $N\geq N_0$,
\begin{align*}
\frac{1}{N}\sum_{k=0}^{N-1}\mathfrak{P}_{k}^{*}\delta_{w_0}(K^c)<\varepsilon.
\end{align*}
Then choose $N_0^*\geq N_0$ as in \eqref{uniformcompact} for given $\varepsilon$ and $K$.  Now
\begin{align*}
&\left|\frac{2}{N^2}\sum_{m=0}^{N-1}\sum_{n=0}^{m}\mathfrak{P}_n(\psi(\mathfrak{P}_{m-n}\psi-v_s))(w_0)\right|\\
&\leq \left|\frac{2}{N^2}\sum_{m=0}^{N-1}\sum_{n=0}^{m}\mathfrak{P}_n(\mathbb{I}_{K}\psi(\mathfrak{P}_{m-n}\psi-v_s))(w_0)\right|+\left|\frac{2}{N^2}\sum_{m=0}^{N-1}\sum_{n=0}^{m}\mathfrak{P}_n(\mathbb{I}_{K^c}\psi(\mathfrak{P}_{m-n}\psi-v_s))(w_0)\right|
& := S_1 +S_2.
\end{align*}
Note that
\begin{align*}
S_1 &= \left|\frac{2}{N^2}\sum_{n=0}^{N-1}\sum_{m=n}^{N-1}\mathfrak{P}_n(\mathbb{I}_{K}\psi(\mathfrak{P}_{m-n}\psi-v_s))(w_0)\right| = \left|\frac{2}{N^2}\sum_{n=0}^{N-1}\sum_{m= 0}^{N-1-n}\mathfrak{P}_n(\mathbb{I}_{K}\psi(\mathfrak{P}_{m}\psi-v_s))(w_0)\right|\\
&\leq  \left|\frac{2}{N^2}\sum_{n=0}^{N_0^*-1}(N-n)\mathfrak{P}_n\left(\mathbb{I}_{K}\psi\left(\frac{1}{N-n}\sum_{m= 0}^{N-1-n}\mathfrak{P}_{m}\psi-v_s\right)\right)(w_0)\right|\\
&+\left|\frac{2}{N^2}\sum_{n=N_0^*}^{N-1}(N-n)\mathfrak{P}_n\left(\mathbb{I}_{K}\psi\left(\frac{1}{N-n}\sum_{m= 0}^{N-1-n}\mathfrak{P}_{m}\psi-v_s\right)\right)\left(w_0\right)\right|\\
&\leq \frac{2N_0^*}{N}\|\psi\|_{\infty}(\|\psi\|_{\infty}+|v_s|) + C\|\psi\|_{\infty}\varepsilon\leq C\|\psi\|_{\infty}\varepsilon,
\end{align*}
for large $N$. And
\begin{align*}
S_2 & \leq \frac{4\|\psi\|_{\infty}^2}{N^2}\sum_{m=0}^{N-1}\sum_{n=0}^{m}\mathfrak{P}_n\mathbb{I}_{K^c}(w_0)\\
& \leq \frac{4\|\psi\|_{\infty}^2}{N^2}\left(\sum_{m=0}^{N_0^*-1}\sum_{n=0}^{m}\mathfrak{P}_n\mathbb{I}_{K^c}(w_0) +\sum_{m=N_0^*}^{N - 1}\frac{m+1}{m+1}\sum_{n=0}^{m}\mathfrak{P}_n\mathbb{I}_{K^c}(w_0)\right)\\
&\leq C\|\psi\|_{\infty}^2\varepsilon,
\end{align*}
for $N$ large enough.  The proof is then complete.
\end{proof}
\subsubsection{Proof of CLT}
Throughout this subsection, we denote by $B_R(0) = \{w\in H: \rho(0,w)<R\}$ the ball under the metric $\rho$. Denote the initial condition as $w$.  And denote $\mathbf{E}[X, Y\geq Z] = \mathbf{E}X\mathbb{I}_{\{Y\geq Z\}}$ for given random variables $X, Y, Z$. The following estimate shows a Lyapunov structure of system \eqref{NS}, which will be useful in the proof.
\begin{lemma}\label{CLTLyapunov}
There is $\delta>0$ so that for all $R>0$, there is constant $C =C(R,\delta)$ independent of initial time $s$ such that
\[\sup_{t\geq0}\sup_{w\in B_R(0)}\mathbf{E}\rho(0, w_{s, s+t}(w))^{2+\delta} \leq C.\]
\end{lemma}
\begin{proof}
Recall that $V(w) = e^{\eta\|w\|^2}$, and $\kappa> 1$ as in Proposition \ref{verifyLyapunov}. From the definition of the metric $\rho$, we have
\begin{align*}
\mathbf{E}\rho(0, w_{s, s+t}(w))^{2+\delta}&\leq \mathbf{E}\|w_{s, s+t}(w)\|^{2+\delta}V^{2+\delta}(w_{s, s+t}(w))\\
&\leq C\mathbf{E}V^{\kappa(2+\delta)}(w_{s, s+t}(w))\leq C\exp(\kappa(2+\delta)\eta\|w\|^2),
\end{align*}
by taking $\eta$ small so that $\kappa(2+\delta)\eta\leq\eta_0$ and one can apply the estimate  \eqref{eq: est1} in the last step. The proof is complete.
\end{proof}

The proof of  \eqref{CLT} is based on the central limit theorem for martingales proved in \cite{KW12} by a martingale approximation approach, where the idea dates back to the early works of  Gordin \cite{Gor69}, Kipnis and Varadhan \cite{KV86}. Suppose that $\left\{\mathfrak{F}_{n}, n \geq 0\right\}$ is a filtration over $(\Omega, \mathcal{F}, \mathbf{P})$ such that $\mathfrak{F}_{0}$ is trivial and $\left\{Z_{n}, n \geq 1\right\}$ is a sequence of square integrable martingale differences, i.e. it is $\left\{\mathfrak{F}_{n}, n \geq 1\right\}$ adapted, $\mathbf{E} Z_{n}^{2}<+\infty$ and $\mathbf{E}\left[Z_{n} \,\middle\vert\, \mathfrak{F}_{n-1}\right]=0$ for all $n \geq 1 .$ Define also the martingale
$$M_{N}:=\sum_{j=1}^{N} Z_{j}, \quad N \geq 1, \quad M_{0}:=0$$
Its quadratic variation equals $[M]_{N}:=\sum_{j=1}^{N} \mathbf{E}\left[Z_{j}^{2} \,\middle\vert\, \mathfrak{F}_{j-1}\right]$ for $N \geq 1 .$ Assume also that:\\
(M1) for every $\varepsilon>0,$
\[\lim _{N \rightarrow+\infty} \frac{1}{N} \sum_{j=0}^{N-1} \mathbf{E}\left[Z_{j+1}^{2},\left|Z_{j+1}\right| \geq \varepsilon \sqrt{N}\right]=0,\]
(M2) the condition
\[\sup _{n \geq 1} \mathbf{E} Z_{n}^{2}<+\infty\]
and there exists $\sigma \geq 0$ such that
\[\lim _{K \rightarrow \infty} \limsup _{\ell \rightarrow \infty} \frac{1}{\ell} \sum_{m=1}^{\ell} \mathbf{E}\left|\frac{1}{K} \mathbf{E}\left[[M]_{m K}-[M]_{(m-1) K} \,\middle\vert\, \mathfrak{F}_{(m-1) K}\right]-\sigma^{2}\right|=0\]
and\\
(M3) for every $\varepsilon>0$
\[\lim _{K \rightarrow \infty} \limsup _{\ell \rightarrow \infty} \frac{1}{\ell K} \sum_{m=1}^{\ell} \sum_{j=(m-1) K}^{m K-1} \mathbf{E}\left[1+Z_{j+1}^{2},\left|M_{j}-M_{(m-1) K}\right| \geq \varepsilon \sqrt{\ell K}\right]=0.\]
Then one has
\begin{theorem}\label{martingaleclt}\cite{KW12}
Under  conditions (M1)-(M3), one has
\[\lim _{N \rightarrow\infty} \frac{\mathbf{E}[M]_{N}}{N}=\sigma^{2}\]
and
\[\lim _{N \rightarrow \infty} \mathbf{E} e^{i \theta M_{N} / \sqrt{N}}=e^{-\sigma^{2} \theta^{2} / 2}, \quad \forall \theta \in \mathbb{R}.\]
\end{theorem}
Assume $\psi\in\mathrm{Lip}_{\r}(H)$ and $\langle\mu_s, \psi\rangle = 0$, otherwise one can replace $\psi$ by $\psi- \langle\mu_s, \psi\rangle$.  In this subsection, $\left\{\mathfrak{F}_{n}, n \geq 0\right\}$ denotes the natural filtration generated by the Markov chain $\{w_{s, s+n\mathcal{T}}(w)\}_{n\geq 0}$.  Let $\alpha_n(w) = \sum_{k=0}^{n}\mathfrak{P}_k\psi(w)$. We have
\begin{lemma}\label{corrector}
The sequence of functions $\alpha_n$ converges uniformly on bounded sets. The point wise limit $\alpha := \lim_{n\rightarrow\infty}\alpha_n\in\mathrm{Lip}_{\rho}(H) $ and
for any $n\geq m$,
\[\mathbf{E}\left[\alpha(w_{s, s+n\mathcal{T}})|\mathfrak{F}_m\right] = \lim_{N\rightarrow\infty}\mathbf{E}\left[\alpha_N(w_{s, s+n\mathcal{T}})|\mathfrak{F}_m\right].\]
\end{lemma}
\begin{proof}
By $\langle\mu_s, \psi\rangle = 0$, Theorem \ref{fixedpoint}, the invariance of $\mu_s$, and the Monge-Kantorovich daulity, we have
\begin{align*}
|\alpha_n(w) - \alpha_m(w)| &= \left|\sum_{k=m}^{n}\left(\mathfrak{P}_{k}\psi(w) - \langle \mu_s, \mathfrak{P}_k\psi\rangle\right)\right|\\
&\leq \sum_{k=m}^{n}\mathrm{Lip}_{\rho}(\psi)\rho(\mathfrak{P}_{k}^*\delta_w, \mathfrak{P}_{k}^*\mu_s)
\leq C\mathrm{Lip}_{\rho}(\psi)\rho(\delta_w, \mu_s)\sum_{k=m}^{n}e^{-\gamma k\mathcal{T}}<\varepsilon
\end{align*}
for $n>m>N$ and $N$ large enough.  Hence $\alpha_n$ converges uniformly on bounded sets.
Note that
\[|\alpha_N(w_1)-\alpha_N(w_2)|\leq \sum_{k = 0}^{N}|\langle\mathfrak{P}_k^*\delta_{w_1} - \mathfrak{P}_k^*\delta_{w_2}, \psi\rangle|\leq C\mathrm{Lip}_{\rho}(\psi)\rho(w_1, w_2)\sum_{k=0}^{N}e^{-\gamma k\mathcal{T}},\]
hence $|\alpha(w_1)-\alpha(w_2)|\leq C\mathrm{Lip}_{\rho}(\psi)\rho(w_1, w_2)$, where $C$ does not depend on $w_1, w_2\in H$. So $\alpha\in\mathrm{Lip}_{\r}(H).$  Since each $\alpha_N$ and $\alpha$ are Lipschitz continuous and $\mathfrak{P}_k^*\delta_w\in\mathcal{P}_1(H)$, so they are $\mathfrak{P}_k^*\delta_w$ integrable. By the dominated convergence theorem, one has for any $w\in H$,
\begin{align*}
\lim_{N\rightarrow\infty}\mathfrak{P}_{n- m}\alpha_N(w) = \lim_{N\rightarrow\infty}\mathbf{E}\left[\alpha_N(w_{s, s+(n-m)\mathcal{T}}(w))\right] = \lim_{N\rightarrow\infty}\langle\mathfrak{P}_{n-m}^*\delta_w, \alpha_N\rangle = \langle\mathfrak{P}_{n-m}^*\delta_w, \alpha\rangle = \mathfrak{P}_{n- m}\alpha(w).
\end{align*}
Hence
\[\lim_{N\rightarrow\infty}\mathbf{E}\left[\alpha_N(w_{s, s+n\mathcal{T}})|\mathfrak{F}_m\right] = \lim_{N\rightarrow\infty}\mathfrak{P}_{n- m}\alpha_N(w_{s, s+m\mathcal{T}}) =\mathfrak{P}_{n- m}\alpha(w_{s, s+m\mathcal{T}}) =\mathbf{E}\left[\alpha(w_{s, s+n\mathcal{T}})|\mathfrak{F}_m\right].  \]
\end{proof}
To apply Theorem \ref{martingaleclt}, we decompose the average in \eqref{CLT} into the sum of a martingale and a small remainder term
\begin{align*}
\frac{1}{\sqrt{N}}\sum_{k=0}^{N-1}\psi(w_{s, s+k\mathcal{T}}) = \frac{1}{\sqrt{N}}
M_N+R_N,
\end{align*}
where $M_N = \alpha({w_{s, s+N\mathcal{T}}})- \alpha(w) + \sum_{k=0}^{N-1}\psi(w_{s, s+k\mathcal{T}}) $, $M_0 = 0$ and $R_N = \frac{1}{\sqrt{N}}(\alpha(w) - \alpha({w_{s, s+N\mathcal{T}}}))$.
\begin{proposition}
 $\{M_N\}_{N\geq 0}$ is a martingale with respect to the natural filtration $\{\mathfrak{F}_N\}_{N\geq 0}$.
\end{proposition}
\begin{proof}
First note that $|\alpha(w)|\leq |\alpha(w)- \alpha(0)| + |\alpha(0)| \leq C(1+\mathrm{Lip}_{\rho}(\psi)\rho(w, 0)). $ Hence
\begin{align*}
\mathbf{E}|M_N|&\leq \mathbf{E}|\alpha({w_{s, s+N\mathcal{T}}})| +  \mathbf{E}|\alpha(w)| + \sum_{k=0}^{N-1} \mathbf{E}|\psi(w_{s, s+k\mathcal{T}})|\\
&\leq C(1+\mathbf{E}\rho(0,w_{s, s+N\mathcal{T}} )) + |\alpha(w)|+\sum_{k=0}^{N-1}\mathfrak{P}_{k}|\psi|(w)\\
&\leq C(1+\mathbf{E}\rho(0,w_{s, s+N\mathcal{T}} )) + |\alpha(w)|+C(1+\rho(w, 0))<\infty.
\end{align*}
Also  for $m\leq N$,
\[\mathbf{E}[M_N|\mathfrak{F}_m] =  \mathbf{E}[\alpha({w_{s, s+N\mathcal{T}}})|\mathfrak{F}_m]- \alpha(w) + \sum_{k=0}^{m-1}\mathbf{E}[\psi(w_{s, s+k\mathcal{T}})|\mathfrak{F}_m] + \sum_{k=m}^{N-1}\mathbf{E}[\psi(w_{s, s+k\mathcal{T}})|\mathfrak{F}_m]\]
and
\begin{align*}
\sum_{k=m}^{N-1}\mathbf{E}[\psi(w_{s, s+k\mathcal{T}})|\mathfrak{F}_m] &= \sum_{k=m}^{\infty}\mathbf{E}[\psi(w_{s, s+k\mathcal{T}})|\mathfrak{F}_m] - \sum_{k=N}^{\infty}\mathbf{E}[\psi(w_{s, s+k\mathcal{T}})|\mathfrak{F}_m] \\
&= \sum_{k=m}^{\infty} \mathfrak{P}_{k-m}\psi(w_{s, s+m\mathcal{T}}) - \sum_{k=N}^{\infty} \mathfrak{P}_{k-N}\mathfrak{P}_{N-m}\psi(w_{s, s+m\mathcal{T}}) \\
& = \alpha(w_{s, s+m\mathcal{T}}) -\mathfrak{P}_{N-m}\alpha(w_{s, s+m\mathcal{T}}) \\
& = \alpha(w_{s, s+m\mathcal{T}}) -\mathbf{E}[\alpha(w_{s, s+N\mathcal{T}})|\mathfrak{F}_m].
\end{align*}
Therefore $\mathbf{E}[M_N|\mathfrak{F}_m]=M_{m}.$
\end{proof}
\begin{lemma}
$\mathbf{E}R_N\rightarrow 0$ as $N\rightarrow\infty$.
\end{lemma}
\begin{proof}
We have by Theorem \ref{fixedpoint},
\begin{align*}
\frac{1}{\sqrt{N}}\mathbf{E}|\alpha(w_{s, s+N\mathcal{T}})| &= \frac{1}{\sqrt{N}}\left|\langle\mathfrak{P}_N^*\delta_{w} - \mathfrak{P}_N^*\mu_s, |\alpha|\rangle\right| + \frac{1}{\sqrt{N}}\langle  \mathfrak{P}_N^*\mu_s, |\alpha|\rangle. \\
&\leq \frac{C}{\sqrt{N}}\mathrm{Lip}_{\rho}(|\alpha|)e^{-\gamma N \mathcal{T}}\rho(\delta_{w}, \mu_s)+ \frac{1}{\sqrt{N}}\langle\mu_s, |\alpha|\rangle\rightarrow 0 \quad (N\rightarrow\infty).
\end{align*}
\end{proof}
Setting $Z_n = M_n - M_{n-1}$  for $n\geq1$, we will verify the conditions of Theorem \ref{martingaleclt} to finish the proof of \eqref{CLT}.
\begin{proof}[Verification of condition (M1)]
Set $F(w)  = \mathbf{E}\left[Z_1^2(w)\mathbb{I}_{|Z_1|\geq \varepsilon\sqrt{N}}\right]$. Note that by the periodicity one has for each $k\geq 0$,
\begin{align*}
F(w) =  \mathbf{E}\left[\left(\alpha({w_{s+k\mathcal{T}, s+(k+1)\mathcal{T}}})- \alpha(w) +\psi(w)\right)^2,\left|\alpha({w_{s+k\mathcal{T}, s+(k+1)\mathcal{T}}})-\alpha(w) +\psi(w)\right|\geq\varepsilon\sqrt{N}\right].
\end{align*}
Then by the Markov property,
\[\mathfrak{P}_{k-1}F(w) = \mathbf{E}\left[F(w_{s, s+(k-1)\mathcal{T}}(w))\right] =\mathbf{E}\left[ \mathbf{E}\left[Z_k^2\mathbb{I}_{|Z_k|\geq \varepsilon\sqrt{N}}|\mathfrak{F}_{k-1}\right]\right] =\mathbf{E}\left[Z_k^2\mathbb{I}_{|Z_k|\geq \varepsilon\sqrt{N}}\right] \]
\end{proof}
Hence
\[\frac{1}{N} \sum_{j=0}^{N-1} \mathbf{E}\left[Z_{j+1}^{2},\left|Z_{j+1}\right| \geq \varepsilon \sqrt{N}\right] =\left\langle\frac{1}{N} \sum_{j=1}^{N}\mathfrak{P}_{j-1}^{*}\delta_w, F\right\rangle. \]
We recall a lemma from \cite{KW12}.
\begin{lemma}\cite{KW12}\label{weakconvergencecompact}
If $\mu_N$ converges to $\mu$ weakly, and $G_N\rightarrow 0$ uniformly on compact sets and there is  $\eta>0$ such that $\limsup_{N\rightarrow\infty}\langle \mu_N, |G_N|^{1+\eta}\rangle<\infty$, then $\lim_{N\rightarrow\infty}\langle\mu_N, G_N\rangle = 0$.
\end{lemma}
Note that $\nu_N: = \frac{1}{N} \sum_{j=1}^{N}\mathfrak{P}_{j-1}^{*}\delta_w$ converges to $\mu_s$ weakly. To verify condition (M1), it remains to show that $F_N(w) = F(w)$ converges to $0$ uniformly on compact sets and there is some $\eta>0$ that meets the condition of the above Lemma \ref{weakconvergencecompact}.
Note that
\begin{align*}
F_N(w) = \mathbf{E}\left[Z_1^2\mathbb{I}_{|Z_1|\geq \varepsilon\sqrt{N}}\right]& \leq \left( \mathbf{E}|Z_1|^{2+\delta}\right)^{\frac{2}{2+\delta}}\mathbf{P}(|Z_1|\geq\varepsilon\sqrt{N})^{\frac{\delta}{2+\delta}}\\
&\leq \left( \mathbf{E}|Z_1|^{2+\delta}\right)^{\frac{2}{2+\delta}}\left(\frac{\mathbf{E}|Z_1|^{2+\delta}}{\varepsilon^{2+\delta} N^{\frac{2+\delta}{2}}}\right)^{\frac{\delta}{2+\delta}}= \frac{\mathbf{E}|Z_1|^{2+\delta}}{\varepsilon^{\delta}N^{\delta/2}}.
\end{align*}
And
\begin{align}\notag
\mathbf{E}|Z_1|^{2+\delta}& = \mathbf{E}|\alpha(w_{s, s+\mathcal{T}}(w)) - \alpha(w) + \psi(w)|^{2+\delta}\\\notag
&\leq C\left(\mathbf{E}|\alpha(w_{s, s+\mathcal{T}}(w)) - \alpha(w)|^{2+\delta}+|\psi(w)|^{2+\delta}\right)\\\label{discreteCLT01}
&\leq C\left(\mathrm{Lip}_{\rho}(\alpha)^{2+\delta}\mathbf{E}\rho(w_{s, s+\mathcal{T}}(w), w)^{2+\delta}+|\psi(0)|^{2+\delta}+\mathrm{Lip}_{\rho}(\psi)^{2+\delta}\rho(0, w)^{2+\delta}\right).
\end{align}
Therefore for any $R>0$, by Lemma \ref{CLTLyapunov}, one has
\begin{align*}
\sup_{w\in B_R(0)} F_{N}(w) = \sup_{w\in B_R(0)} \mathbf{E}\left[Z_1^2\mathbb{I}_{|Z_1|\geq \varepsilon\sqrt{N}}\right]\leq \sup_{w\in B_R(0)}\mathbf{E}|Z_1|^{2+\delta}\frac{1}{\varepsilon^{\delta}N^{\delta/2}}\rightarrow 0 \quad (N\rightarrow \infty).
\end{align*}
Hence $F_N(w)$ converges to $0$ uniformly on any compact set. By \eqref{discreteCLT01},
\begin{align*}
&\langle\nu_N, |F_N|^{1+\delta/2}\rangle\leq \left\langle\nu_N, \mathbf{E}|Z_1|^{2+\delta}\right\rangle\\
&\leq C\left(\mathrm{Lip}_{\rho}(\alpha)^{2+\delta}\int_{H} \mathbf{E}\rho(w_{s, s+\mathcal{T}}(z), z)^{2+\delta}\nu_N(dz)+|\psi(0)|^{2+\delta}+\mathrm{Lip}_{\rho}(\psi)^{2+\delta}\int_{H} \rho(0, z)^{2+\delta}\nu_N(dz)\right),
\end{align*}
where by Lemma \ref{CLTLyapunov}, we have
\begin{align*}
\int_{H} \rho(0, z)^{2+\delta}\nu_N(dz)= \frac{1}{N} \sum_{j=1}^{N}\mathfrak{P}_{j-1}\rho(0, w)^{2+\delta} =  \frac{1}{N} \sum_{j=1}^{N}\mathbf{E}\rho(0, w_{s, s+(j-1)\mathcal{T}}(w))^{2+\delta}<\infty,
\end{align*}
and
\begin{align*}
\int_{H} \mathbf{E}\rho(w_{s, s+\mathcal{T}}(z), 0)^{2+\delta}\nu_N(dz) = \frac{1}{N} \sum_{j=1}^{N}\int_H\mathfrak{P}_1\rho(0, z)^{2+\delta}\mathfrak{P}_{j-1}^{*}\delta_w(dz) = \frac{1}{N} \sum_{j=1}^{N}\mathfrak{P}_{j}\rho(0, w)^{2+\delta}<\infty.
\end{align*}
Hence $\limsup_{N\rightarrow\infty}\langle\nu_N, |F_N|^{1+\delta/2}\rangle<\infty$. The condition $(M1)$ then follows from Lemma \ref{weakconvergencecompact}.
\begin{proof}[Verification of condition (M2)]We first show that $Z_n$ is a square integrable martingale. Indeed, Lemma \ref{CLTLyapunov} implies that
\begin{align*}
\mathbf{E}Z_n^2&\leq 2\left(\mathbf{E}|\alpha(w_{s, s+n\mathcal{T}})-\alpha(w_{s, s+(n-1)\mathcal{T}})|^2 + \mathbf{E}|\psi(w_{s, s+(n-1)\mathcal{T}})|^2\right)\\
&\leq C \mathbf{E}\left[\mathrm{Lip}_{\rho}(\alpha)^{2}\left(\rho(0, w_{s, s+n\mathcal{T}})^2 + \rho(0, w_{s, s+(n-1)\mathcal{T}})^2\right) + |\psi(0)|^2+\mathrm{Lip}_{\rho}(\psi)^{2}\rho(0, w_{s, s+(n-1)\mathcal{T}})^2\right]\\
&<\infty.
\end{align*}
By periodicity and Markov property, for any $\sigma\geq 0$,
\begin{align*}
\frac{1}{\ell} \sum_{m=1}^{\ell} \mathbf{E}\left|\frac{1}{K} \mathbf{E}\left[[M]_{m K}-[M]_{(m-1) K} \,\middle\vert\, \mathfrak{F}_{(m-1) K}\right]-\sigma^{2}\right| =\frac{1}{\ell} \sum_{m=1}^{\ell}\mathfrak{P}_{(m-1)K}|H_K|(w),
\end{align*}
where $H_K(w) = \mathbf{E}\left[\frac{1}{K}[M]_K-\sigma^2\right] =\mathbf{E}\left[\frac{1}{K}M_K^2-\sigma^2\right] $.
Following the same strategy as the proof of Lemma 5.5 in \cite{KW12}, one can show that $H_K\in C(H)$ and
\[\limsup_{\ell\rightarrow \infty}\frac{1}{\ell} \sum_{m=1}^{\ell}\left\langle\mathfrak{P}_{(m-1)K}^{*}\delta_w, |H_K|\right\rangle<\infty.\]
Hence $\lim_{\ell\rightarrow\infty}\frac{1}{\ell} \sum_{m=1}^{\ell}\left\langle\mathfrak{P}_{(m-1)K}^{*}\delta_w, |H_K|\right\rangle = \langle \mu_s,  |H_K|\rangle. $ Note that $H_K(w) =\frac{1}{K}\sum_{j=0}^{K-1}\mathfrak{P}_{j}J(w) $, where $J = \mathbf{E}M_1^2 - \sigma^2$.  And $\langle \mu_s,  |H_K|\rangle= \mathbf{E}\left|\frac{1}{K}\sum_{j=0}^{K-1}\mathfrak{P}_{j}J(X_0) \right|$, and $X_0$ is an initial data with law $\mu_s$. Since $\mu_s$ is ergodic under the Markov semigroup $\{\mathfrak{P}_k\}_{k\geq 0}$, by the Birkhoff ergodic theorem, one has for the stationary Markov chain $X_j:= w(s+j\mathcal{T}, \omega; s, X_0)$,
\begin{align*}
\frac{1}{K}\sum_{j=0}^{K-1}\mathfrak{P}_{j}( \mathbf{E}M_1^2)(X_0) =  \frac{1}{K}\sum_{j=0}^{K-1} (\mathbf{E}M_1^2)(X_j)\rightarrow \int_{H}(\mathbf{E}M_1^2)(z)\mu_s(dz), \quad \mathbf{P}\text{-a.s.} \quad (K\rightarrow\infty).
\end{align*}
Hence if we choose  $\sigma^2 = \langle\mu_s, \mathbf{E}M_1^2\rangle$, then $\langle \mu_s,  |H_K|\rangle\rightarrow 0$ as $K\rightarrow\infty$, which completes the proof of condition (M2).
\end{proof}
\begin{proof}[Verification of condition (M3)]
By the periodicity and Markov property, we can rewrite the expression in condition (M3) as
\begin{align*}
 \sum_{m=1}^{\ell} \sum_{j=(m-1) K}^{m K-1} \mathbf{E}\left[1+Z_{j+1}^{2},\left|M_{j}-M_{(m-1) K}\right| \geq \varepsilon \sqrt{\ell K}\right] =  \frac{1}{K}\sum_{j = 0}^{K-1}\left\langle Q_{\ell, K}^{*}\delta_w,G_{\ell, j} \right\rangle ,
 \end{align*}
 where $ Q_{\ell, K}^{*}\delta_w =\frac{1}{\ell}\sum_{m =1}^{\ell}\mathfrak{P}_{(m-1)K}^{*}\delta_w $ and $G_{\ell, j}(w) = \mathbf{E}\left[\left(1+Z_{j+1}^2\right)\mathbb{I}_{|M_j|\geq \varepsilon\sqrt{\ell K}}\right].$  To show that the limit in the condition (M3) vanishes, it suffices to prove that
\begin{align}\label{M3}
\limsup_{\ell\rightarrow \infty}\left\langle Q_{\ell, K}^{*}\delta_w,G_{\ell, j} \right\rangle=0,  \text{ for } j = 0, 1, \cdots, K-1.
\end{align}
Let $\widetilde{\alpha} = \sum_{k = 0}^{\infty}\mathfrak{P}_{k}|\psi|$, then $\widetilde{\alpha}\in\mathrm{Lip}_{\rho}(H)$ as in Lemma \ref{corrector}. Noting that
\begin{align*}
\mathbf{E}|M_j| &\leq \mathbf{E}|\alpha(w_{s, s+j\mathcal{T}}) - \alpha(w)| + \sum_{k = 0}^{j-1}\mathbf{E}|\psi(w_{s, s+k\mathcal{T}})|\\
&\leq \mathrm{Lip}_{\rho}(\alpha)\mathbf{E}\rho(w_{s, s+j\mathcal{T}}, w) + \widetilde{\alpha}(w)\leq \mathrm{Lip}_{\rho}(\alpha)\mathbf{E}\rho(w_{s, s+j\mathcal{T}}, w) +  \mathrm{Lip}_{\rho}(\widetilde{\alpha})\rho(w, 0) + \widetilde{\alpha}(0).
\end{align*}
Therefore by Markov inequality and Lemma \ref{CLTLyapunov}, there is a constant $C>0$ such that
\begin{align*}
\sup_{w\in B_{R}(0)}\mathbf{P}(|M_j|\geq \varepsilon\sqrt{\ell K})\leq \sup_{w\in B_{R}(0)}\frac{\mathbf{E}|M_j| }{\varepsilon\sqrt{\ell K}}\leq \frac{C}{\varepsilon\sqrt{\ell K}}.
\end{align*}
Observe that
\[|\alpha(w_{s, s+(j+1)\mathcal{T}})-\alpha(w_{s, s+j\mathcal{T}})|\leq \mathrm{Lip}_{\rho}(\alpha)\left(\rho(0, w_{s, s+(j+1)\mathcal{T}})+\rho(0, w_{s, s+j\mathcal{T}})\right), \]
\[|\psi(w_{s, s+j\mathcal{T}})|\leq  \mathrm{Lip}_{\rho}(\psi)\rho(0, w_{s, s+j\mathcal{T}})+|\psi(0)|.\]
Therefore
\begin{align*}
&\sup_{w\in B_{R}(0)}\mathbf{E}\left[Z_{j+1}^2\mathbb{I}_{|M_j|\geq \varepsilon\sqrt{\ell K}}\right]\\
&\leq 2\sup_{w\in B_{R}(0)}\left(\mathbf{E}\left[|\alpha(w_{s, s+(j+1)\mathcal{T}})-\alpha(w_{s, s+j\mathcal{T}})|^2\mathbb{I}_{|M_j|\geq \varepsilon\sqrt{\ell K}}\right]+ \mathbf{E}\left[\psi(w_{s, s+j\mathcal{T}})^2\mathbb{I}_{|M_j|\geq \varepsilon\sqrt{\ell K}}\right]\right)\\
&\leq C\sup_{0\leq j\leq K-1}\sup_{w\in B_{R}(0)}\mathbf{E}\left[\rho(0, w_{s, s+j\mathcal{T}})^2\mathbb{I}_{|M_j|\geq \varepsilon\sqrt{\ell K}}\right]\\
&\leq C\sup_{0\leq j\leq K-1}\left(\sup_{w\in B_{R}(0)}\mathbf{E}\rho(0, w_{s, s+j\mathcal{T}})^{2+\delta}\right)^{\frac{2}{2+\delta}}\left(\sup_{w\in B_{R}(0)}\mathbf{P}\left(|M_j|\geq \varepsilon\sqrt{\ell K}\right)\right)^{\frac{\delta}{2+\delta}},
\end{align*}
where the above constant does not depend on $\ell$. Hence for any $R>0$,
\begin{align}\label{discreteCLT02}
\lim_{\ell\rightarrow\infty}\sup_{w\in B_{R}(0)}G_{\ell,j} = 0.
\end{align}
Note that
\begin{align*}
\left\langle Q_{\ell, K}^{*}\delta_w,G_{\ell, j}^{1+\frac{2}{\delta}} \right\rangle&\leq \int_{H}\mathbf{E} (1+Z_{j+1}^2(z))^{1+\frac{\delta}{2}}Q_{\ell, K}^{*}\delta_w(dz)\\
&\leq C(1+ \langle Q_{\ell, K}^{*}\delta_w,\mathbf{E} Z_{j+1}^{2+\delta}\rangle)\leq C(1+ \sup_{j\geq 0}\mathbf{E}\rho(0, w_{s, s+j\mathcal{T}})^{2+\delta}).
\end{align*}
Again by Lemma \ref{CLTLyapunov},
\begin{align}\label{discreteCLT03}
\limsup_{\ell\rightarrow\infty}\left\langle Q_{\ell, K}^{*}\delta_w,G_{\ell, j}^{1+\frac{2}{\delta}} \right\rangle<\infty.
\end{align}
The conclusion \eqref{M3} then follows from \eqref{discreteCLT02}, \eqref{discreteCLT03} and Lemma \ref{weakconvergencecompact}, which in turn completes the verification of condition (M3).
\end{proof}

\section{Larger viscosity implies a trivial periodic attractor}\label{sectiontrivialattractor}
We have shown that  for periodic force $f(t)$, if $G_1<1/c_0$, then the system \eqref{NS} has a unique periodic invariant measure. The aim of this section is to show that if the viscosity is  larger, then the dynamics is actually trivial in the sense that the unique invariant measure is supported on a unique exponentially stable random periodic solution.  We take the approach from \cite{Mat98}, where the case when $f=0$ has been proved.

\vskip0.05in

For $\alpha\in(0,1]$, let  $\delta_0 := \nu - \frac{c_0^2}{(2-\alpha)\nu^2}\left(\frac{\|f\|_{\infty}^2}{\alpha\nu}+\mathcal{B}_0\right)$.  We will prove the following result in this section.
\begin{theorem}\label{contractionagain}
Assume $\delta_0>0$. Then there exists a random $\mathcal{T}$-periodic solution $w^*(t, x)$ of equation \eqref{NS} in the sense of  the definition given by equation \eqref{randomperiodic}. Moreover, $w^*(t, x)$ attracts all other solutions both in forward and pullback times. The law of $w^*(t, x)$ gives the unique $\mathcal{T}$-periodic invariant measure. \end{theorem}
In particular, when $\alpha=1$, the condition $\delta_0>0$ is equivalent to the condition $G_2<1/c_0$, so that we obtain Theorem \ref{thmperiodicattraction} from Theorem \ref{contractionagain}.
The proof of Theorem \ref{contractionagain} is a combination of Proposition \ref{rps1} and Theorem \ref{rps2} below.

\vskip0.05in

We begin by comparing any two solutions that start from different initial data. Let $w_0, \widetilde{w}_0\in H$, and $\mathfrak{e}_t = w_{s, t}(w_0)-\widetilde{w}_{s, t}(\widetilde{w}_0)$. In view of equation \eqref{NS}, we see that $\mathfrak{e}_t$ solves the following equation
\begin{align*}
\partial_t\mathfrak{e}_t&=\nu\mathrm{\Delta}\mathfrak{e}_t - B(\mathcal{K}w_{s, t}, w_{s, t})+B(\mathcal{K}\widetilde{w}_{s, t},\widetilde{w}_{s, t})\\
&=\nu\mathrm{\Delta}\mathfrak{e}_t - B(\mathcal{K}w_{s, t}, w_{s, t})+B(\mathcal{K}(w_{s, t}-\mathfrak{e}_t),w_{s, t}-\mathfrak{e}_t)\\
&= \nu\mathrm{\Delta}\mathfrak{e}_t - B(\mathcal{K}w_{s, t},\mathfrak{e}_t)+B(\mathcal{K}\mathfrak{e}_t, w_{s, t})+B(\mathcal{K}\mathfrak{e}_t,\mathfrak{e}_t).
\end{align*}
It then follows from  $\langle B(\mathcal{K}w, v),v\rangle = 0$ and basic estimates of the nonlinear term (see Proposition 6.1 in \cite{CF88}) that
\begin{align*}
\partial_t\|\mathfrak{e}_t\|^2& = 2\langle\mathfrak{e}_t, \partial_t\mathfrak{e}_t\rangle = -2\nu\|\mathfrak{e}_t\|_1^2 + 2\langle B(\mathcal{K}\mathfrak{e}_t, w_{s, t}), \mathfrak{e}_t\rangle \\
&\leq -2\nu\|\mathfrak{e}_t\|_1^2 + 2c_0 \|\mathfrak{e}_t\|\|w_{s, t}\|_1\|\mathfrak{e}_t\|_1\leq -\nu\|\mathfrak{e}_t\|^2 + \frac{c_0^2}{\nu}\|\mathfrak{e}_t\|^2\|w_{s, t}\|_1^2.
\end{align*}
Hence by the Gronwall's inequality,
\[\|\mathfrak{e}_t\|^2\leq \|\mathfrak{e}_s\|^2\exp\left(-\nu(t-s)+\frac{c_0^2}{\nu}\int_s^t\|w_{s, r}\|_1^2dr\right).\]
From equation \eqref{Ito1} one obtains that for any $\alpha\in(0, 1]$,
\begin{align*}
\frac{c_0^2}{\nu}\int_{s}^t\|w_{s, r}\|_1^2dr\leq \frac{c_0^2}{(2-\alpha)\nu^2}\left(\|w_0\|^2+M(s, t)+(t-s)\left(\frac{\|f\|_{\infty}^2}{\alpha\nu}+\mathcal{B}_0\right)\right),
\end{align*}
where $M(s, t) = 2\int_{s}^t\langle w_{s, r},GdW_r\rangle$. As a result,
\begin{align}\label{contraction}
\|\mathfrak{e}_t\|^2\leq \|\mathfrak{e}_s\|^2e^{-\left(t-s\right)\left(\delta_0 - \mathrm{\Gamma}(s, t)\right)},
\end{align}
where $\mathrm{\Gamma}(s, t) = \frac{c_0^2}{(2-\alpha)\nu^2(t-s)}\left(\|w_0\|^2+M(s,t)\right)$. Just as in \cite{Mat98}, it turns out that the average $\mathrm{\Gamma}(s, t)$ can be sufficiently small for  fixed $s$ and large (random) $t$. Hence as long as $\delta_0>0$, or equivalently if $\nu$ is large, then the dissipation dominates so that the dynamics is contracting in the state space. In what follows, the inequality \eqref{contraction} will be used to construct the unique random periodic solution that attracts all other solutions. The virtue of the proof is based on controlling the term $\mathrm{\Gamma}(s,t)$.  Throughout the proof, we  fix  $\alpha\in(0,1]$.

\vskip0.05in

The following proposition allows us to construct a solution that starts from $-\infty$. Since the proof is essentially the same as that in \cite{Mat99}, we omit it here.
\begin{proposition}\label{rps1}
Fix a $\delta\in(0,\delta_0)$ and $t_1\in \mathcal{T}\mathbb{Z}$. Then for any $\varepsilon>0$,  there exists a  $\mathcal{T}\mathbb{Z}^{+}$ valued random time $n^*(\varepsilon, \delta, t_1)$, such that with probability one, for any $\tau\geq 0$ and $n_1, n_2\in \mathcal{T}\mathbb{Z}$,
\[n_1, n_2<t_1-n^*\implies \|w(t_1+\tau; n_1,0)-w(t_1+\tau; n_2,0)\|^2\leq \varepsilon e^{-\delta\tau}.\]
\end{proposition}
We now construct the random periodic solution that has the desired attraction property.
\begin{theorem}\label{rps2}
Assume $\delta_0>0$ and let $\delta\in(0,\delta_0)$. Then there exists a unique random periodic solution $w^*: \mathbb{R}\times\mathrm{\Omega}\rightarrow H$ of system \eqref{NS} such that for any $s\in\mathbb{R}$, there exists positive random times $n^*(s,\delta)$ and $n_*(s,\delta)$ have all moments finite such that
\begin{align}
\sup_{w_0\in B_r(w^*(s))}\|w(t; s, w_0)-w^*(t)\|^2\leq r^2e^{-\delta (t-s)},\label{forward}\\
\sup_{w_{\tau}\in B_r(w^*(\tau))}\|w(s; \tau, w_{\tau})-w^*(s)\|^2\leq r^2e^{-\delta (s-\tau)},\label{backward}
\end{align}
for all $r>0$, $t>s+n^*$ and $\tau<s-n_*$.
\end{theorem}
\begin{proof}
Let $n_1\in \mathcal{T}\mathbb{Z}$. Consider the sequence of solutions $w_n(t, \omega) = w(t,\omega; n_1-n, 0)$ for $n\in \mathcal{T}\mathbb{Z}^{+}$ and $t\geq n_1$.  Proposition \ref{rps1} states that  there exists a random time $n^*(\varepsilon, \delta, n_1)$ such that for every $t\geq n_1$  and every $m_1, m_2\in \mathcal{T}\mathbb{Z}$ satisfying
$m_1, m_2> n^*$, one has
\[\|w(t; n_1-m_1, 0)-w(t; n_1-m_2, 0)\|\leq \varepsilon e^{-\delta(t-n_1)}.\]
This implies that $\{w_n\}$ is a Cauchy sequence in $C\left([n_1,\infty), H\right)$, which is complete with the norm $\displaystyle\|w\|_{\infty}:=\sup_{t\geq n_1}\|w(t)\|$. Define $w^*(t,\omega)$ to be the limit for $t\geq n_1$. Since $n_1$ is arbitrary, we obtain a process $w^*(t,\omega)$ defined for $t\in\mathbb{R}$. For any fixed $T>n_1$, it is well known (see \cite{CDF97, Fla94} for example) that there exists a random variable $K(\omega, T)$ such that $\limsup_{n}\|w_n(t)\|<K(\omega, T)$ almost surely for all $t\in[n_1,T]$.  Therefore $w^*\in C([n_1, \infty], H_1)$, and $\{w_n\}$ converges to $w^*$ weakly in $H_1$. Hence Lemma B.6 from \cite{Mat99} shows that $w^*$ is a strong solution of equation \eqref{NS}.

\vskip0.05in

We now show that $w^*$ is a random periodic solution in the sense of equation \eqref{randomperiodic}. By continuity of the stochastic flow $\mathrm{\Phi}(t,\omega; s, w_0)$, we have for $\tau\geq 0$,
\begin{align*}
\mathrm{\Phi}(t+\tau,\omega; t, w^*(t,\omega))&=\mathrm{\Phi}(t+\tau,\omega; t, \lim_{n\rightarrow \infty} w_n(t,\omega; n_1-n,0))\\
& =\lim_{n\rightarrow \infty} \mathrm{\Phi}(t+\tau,\omega; t,  w_n(t,\omega; n_1-n,0))\\
&=\lim_{n\rightarrow \infty}w_n(t+\tau,\omega; n_1-n,0)=w^*(t+\tau,\omega).
\end{align*}
Denote $w_{n}^{\mathcal{T}}(t) = w_{n}(t, \theta_{\mathcal{T}}\omega)$. By the $\mathcal{T}$ periodicity of the deterministic force $f(t)$, we have that
\begin{align*}
w_{n}^{\mathcal{T}}(t)= &\int_{n_1-n}^{t}\left[\nu\mathrm{\Delta}w_{n}^{\mathcal{T}}(r) - B(\mathcal{K}w_{n}^{\mathcal{T}}(r) ,w_{n}^{\mathcal{T}}(r) )+f(r)\right]dr+G\left[W(t,\theta_{\mathcal{T}}\omega)-W(t-n,\theta_{\mathcal{T}}\omega)\right]\\
= &\int_{n_1-n+\mathcal{T}}^{t+\mathcal{T}}\left[\nu\mathrm{\Delta}w_{n}^{\mathcal{T}}(r-\mathcal{T}) - B(\mathcal{K}w_{n}^{\mathcal{T}}(r-\mathcal{T}) ,w_{n}^{\mathcal{T}}(r-\mathcal{T}) )+f(r)\right]dr\\
&+G\left[W(t+\mathcal{T},\omega)-W(n_1-n+\mathcal{T},\omega)\right].
\end{align*}
Also note that
\begin{align*}
w_{n+\mathcal{T}}(t+\mathcal{T},\omega) =&\int_{n_1-n+\mathcal{T}}^{t+\mathcal{T}} \left[\nu\mathrm{\Delta}w_{n+\mathcal{T}}(r)+B(\mathcal{K}w_{n+\mathcal{T}}(r),w_{n+\mathcal{T}}(r))+f(r)\right]dr\\
&+G\left[W(t+\mathcal{T},\omega)-W(n_1-n+\mathcal{T},\omega)\right].
\end{align*}
Therefore it follows from the uniqueness of solution that
\[w_{n}(t, \theta_{\mathcal{T}}\omega) =w_{n+\mathcal{T}}(t+\mathcal{T},\omega),\, \,\mathbf{P}-\mathrm{a.s.}. \]
By letting $n\rightarrow\infty$, we have
\[w^*(t, \theta_{\mathcal{T}}\omega)=w^*(t+\mathcal{T},\omega),\, \,\mathbf{P}-\mathrm{a.s.}.\]

\vskip0.05in

It remains to show that $w^*(t,\omega)$ has the attraction property.  Note that each $w_n$ starts from $0$, hence estimate \eqref{eq: est1} implies that for $p>0$,  $\mathbf{E}\|w_n(t)\|^p\leq C$ for some constant $C$ independent of $n$ and $t$. This in turn shows that $\mathbf{E}\|w^*(t)\|^p\leq C$ for $t\in\mathbb{R}$.  Let $\varepsilon = \delta_0 - \delta$. Define the random time $\tau_1= \frac{2c_0^2}{(2-\alpha)\varepsilon\nu^2}\|w^*(s)\|^2$, which has all moments finite.  Note that for $\tau>\tau_1$, we have $\frac{c_0^2}{(2-\alpha)\nu^2\tau}\|w^*(s)\|^2< \frac{\varepsilon}{2}$, which controls the first term in $\mathrm{\Gamma}(s, s+\tau)$ from \eqref{contraction}.  Based on the same reasoning as in Proposition \ref{rps1},  we find that there exists a $\mathcal{T}\mathbb{Z}^{+}$ valued random time $n_1(s, \delta)$ such that for all $\tau>n_1$, one has
\[\frac{c_0^2}{(2-\alpha)\nu^2\tau}|M(s, s+\tau)|<\frac{\varepsilon}{2}.\]
The estimate \eqref{forward} then follows by taking $n^* = \max\{\tau_1, n_1\}$. For $t>0$, inequality \eqref{contraction} states that
\[\|\mathfrak{e}_s\|^2\leq \|\mathfrak{e}_{s-t}\|^2e^{-t\left(\delta_0 - \mathrm{\Gamma}(s-t, s)\right)},\]
where  $\mathrm{\Gamma}(s-t, s) = \frac{c_0^2}{(2-\alpha)\nu^2t}\left(\|w^{*}(s-t)\|^2+M(s-t, s)\right)$. Now $s$ is fixed, so $M(s-t, s)$ is not a martingale since it runs backwards in time. Nonetheless, we can still use the same reasoning as above. By replacing the Doob $\mathrm{L}^p$ maximal inequality with the backwards maximal inequality (see Lemma A.6 in \cite{Mat99}), one obtains a $\mathcal{T}\mathbb{Z}^+$ valued random time $n_2(s, \delta)$ with all moments finite such that for all $t>n_2$,
\[ \frac{c_0^2}{(2-\alpha)\nu^2t}|M(s-t, s)|<\frac{\varepsilon}{2}.\]
To estimate the term $\frac{1}{t}\|w^*(s-t)\|^2$,  noting that  Theorem 3.13 in \cite{Mat99} remains true in our setting, hence there exists a random time $\tau_2>0$, such that for all $t>\tau_2$, we have $\frac{c_0^2}{(2-\alpha)\nu^2t}\|w^{*}(s-t)\|^2<\frac{\varepsilon}{2}$. Therefore the inequality \eqref{backward} holds by taking $n_* = \max\{\tau_2, n_2\}$.
\end{proof}
\begin{remark}
If $f(t, x)\equiv 0$, then we can take $\alpha =0$ so that $\delta_0 = \nu-\frac{c_0^2\mathcal{B}}{\nu^2}$ which is the same as in \cite{Mat99}.  And with the same proof as above, one can show that under the condition $\delta_0 >0$, there is a stationary solution that exponentially attracts other solutions as in \cite{Mat99}.
\end{remark}

\section{Appendix}\label{appendix}
This appendix contains certain results used in the previous sections regarding system \eqref{NS}. In particular, several useful bounds on the solution of equation \eqref{NS} are given in subsection \ref{appendix-bounds}.  The Lyapunov structure is given in subsection \ref{subsectionLyapunovstructure}. And in subsection \ref{gradientinequality}, we give the proof of the technical result Proposition \ref{pre_gradientinequalityprop} about the spectral property of the Malliavin matrix. The proof of Lemma \ref{contractlemma1}-\ref{contractlemma3} are gathered in subsection \ref{lemmasforcontraction}.

\subsection{Various bounds of the solution}\label{appendix-bounds}
Several estimates about the solution  $w_{s, t}(w_0)$ of the stochastic Navier-Stokes equation \eqref{NS} is collected in the following Lemma \ref{bounds}. And bounds on higher order Sobolev norms are given in the next Lemma \ref{hb}.
\begin{lemma}\label{bounds}
For the solution $w_{s, t}(w_0)$  of the Navier-Stokes equation, we have
\begin{enumerate}

\item There exist constants $\eta_{0}>0$ and $C>0$ depending on $f, \mathcal{B}_0, \nu$ such that for every $t>s$ and every $\eta \in\left(0, \eta_{0}\right],$ the bound
\begin{equation}\label{eq: est1}
\mathbf{E} \exp \left(\eta\left\|w_{s, t}\right\|^{2}\right) \leq C \exp \left(\eta e^{-\nu (t-s)}\left\|w_0\right\|^{2}\right)
\end{equation}
holds.

\item There exist positive constants $\eta_0$ and $C$,  depending  on $f, \mathcal{B}_0, \nu$,  such that
\begin{align}\label{enstrophy}
\quad \mathbf{E} \exp \left(\eta \sup _{t \geq \tau}\left(\left\|w_{s, t}\right\|^{2}+\nu \int_{\tau}^{t}\left\|w_{s, r}\right\|_{1}^{2} d r-C(f,\mathcal{B}_0)(t-\tau)\right)\right)\leq C \exp \left(\eta e^{-\nu \left(\tau-s\right)}\left\|w_0\right\|^{2}\right)
\end{align}
for every $\tau \geq s$ and for every $\eta\in\left(0, \eta_{0}\right] $.

\item For any $f_1, f_2\in H(f)$, the hull of the deterministic force, there exists a constant  $\eta_0>0$ depending  on $f, \mathcal{B}_0, \nu$, so that for every $\eta \in\left(0, \eta_{0}\right]$, there are positive constant $C, \gamma$ depending on $f, \mathcal{B}_0, \nu, \eta$, such that
\begin{align}\label{continuousonhull}
\mathbf{E}\left\|w_{s, t}^{f_1} - w_{s, t}^{f_2}\right\|^2\leq Ce^{\gamma (t-s)}\exp\left({\eta\|w_0\|^2}\right)\|f_1-f_2||_{\infty},
\end{align}
for every $\tau \geq s$, where for $i=1,2$,  $w_{s, t}^{f_i}$ is the solution to equation \eqref{NS} with $f$ replaced by $f_i$ and with initial condition $(s, w_0)$, where $\|f_1-f_2||_{\infty} = \sup_{t\in\mathbb{R}}\|f_1(t)- f_2(t)\|$.

\item There exist constants $\eta_{1}, a, \gamma>0,$ depending  on $f, \mathcal{B}_0, \nu,
\eta_0$,  such that
\begin{align}\label{boundsum}
\mathbf{E} \exp \left(\eta \sum_{n=0}^{N}\left\|w_{s, s+n}\right\|^{2}-\gamma N\right) \leq \exp \left(a \eta\left\|w_0\right\|^{2}\right)
\end{align}
holds for every $N>0,$ every $\eta \leq \eta_{1},$ and every initial condition $w_0 \in \mathcal{H}$.

\item For every $\eta>0$ and $t\geq\tau\geq s$, there exists a constant $C=C\left(\nu, \eta\right)>0$ such that the Jacobian $J_{\tau, t}$ satisfies almost surely
\begin{align}\label{Jacobian}
\left\|J_{\tau, t}\right\| \leq \exp \left(\eta \int_{\tau}^{t}\left\|w_{s, r}\right\|_{1}^{2} d r+C (t-\tau)\right)
\end{align}

\item For every $\eta>0$ and every $p>0,$ there exists $C=C\left(f, \mathcal{B}_0, \nu, \eta, p\right)>0$ such that the Hessian satisfies
\begin{align}\label{Hessian}
\mathbf{E}\left\|K_{\tau, t}\right\|^{p} \leq C \exp \left(\eta\left\|w_0\right\|^{2}\right)
\end{align}
for every $\tau\geq s$ and $t\in(\tau, \tau+1)$.
\end{enumerate}
\end{lemma}

\begin{proof}
\begin{enumerate}
\item Apply Ito's formula and use the fact that $\langle B(\mathcal{K}w,w),w\rangle = 0$.  We have
\begin{align}\label{Ito1}
\left\|w_{s, t}\right\|^{2}-\left\|w_0\right\|^{2}+2 \nu \int_{s}^{t}\left\|w_{s, r}\right\|_{1}^{2} d r=2\int_{s}^{t}\left\langle w_{s, r}, G d W(r)\right\rangle+2\int_{s}^{t}\left\langle w_{s, r}, f(r)\right\rangle dr+\mathcal{B}_0 (t-s)
\end{align}
Let $C(f,\mathcal{B}_0) = \frac{2}{\nu}\|f\|_{\infty}^2+\mathcal{B}_0$. By the inequality $\|w\|_1^2\geq \|w\|^2$ and  Young's product inequality, it follows that
\begin{align*}
\left\|w_{s, t}\right\|^{2} -  e^{-\nu(t-s)}\|w_0\|^2\leq \frac{C(f,\mathcal{B}_0)}{\nu}-\frac{\nu}{2}\int_s^te^{-\nu(t-r)}\|w_{s, r}\|^2dr +\int_s^t e^{-\nu(t-r)}\left\langle w_{s, r}, 2G d W(r)\right\rangle.
\end{align*}
Since there exists constant $\alpha$ such that $\frac{\nu}{2}\|w\|^2\geq \frac{\alpha}{2}\|2G^*w\|^2$, so Lemma A.1 from
\cite{Mat02a} implies that
\[\mathbf{P}\left(\left\|w_{s, t}\right\|^{2}-e^{-\nu (t-s)}\left\|w_0\right\|^{2}-\frac{C(f,\mathcal{B}_0)}{\nu}>\frac{K}{\alpha}\right) \leq e^{-K},\]
which is equivalent to
\[\mathbf{P}\left(\exp\left(\frac{\alpha}{2}\left\|w_{s, t}\right\|^{2}-\frac{\alpha}{2}e^{-\nu (t-s)}\left\|w_0\right\|^{2}-\frac{\alpha C(f,\mathcal{B}_0)}{2\nu}\right)>e^{\frac{K}{2}}\right) \leq e^{-K}.\]
Now if a random variable $X$ satisfies $\mathbf{P}\left(X\geq C\right)\leq \frac{1}{C^2}$ for every $C\geq 0$, then
\begin{align}\notag
\mathbf{E}X = \int_{\mathrm{\Omega}}Xd\mathbf{P}&\leq \int_{\{0\leq X \leq 1\}}Xd\mathbf{P}+\int_{\{X\geq 1\}}Xd\mathbf{P}\leq 1+\sum_{n=0}^{\infty}\int_{\{2^n\leq X\leq 2^{n+1}\}}Xd\mathbf{P}\\\label{appendixbounds00}
&\leq 1+ \sum_{n=0}^{\infty}2^{n+1}\frac{1}{2^{2n}} = 5.
\end{align}
So $\mathbf{E}\exp\left(\frac{\alpha}{2}\left\|w_{s, t}\right\|^{2}-\frac{\alpha}{2}e^{-\nu (t-s)}\left\|w_0\right\|^{2}-\frac{\alpha C(f,\mathcal{B}_0)}{2\nu}\right)\leq 5$.
By taking $\eta_0 = \frac{\alpha}{2}$ and $C = 5\exp\left(\frac{\alpha C\left(f,\mathcal{B}_0\right)}{2\nu}\right)$, we finish the proof.

\item  Again apply Ito's formula, for any $\eta>0$, $s\leq\tau<t$,
\begin{align*}
\eta\left\|w_{s, t}\right\|^{2}&+\eta \nu \int_{\tau}^{t}\left\|w_{s, r}\right\|_{1}^{2} d r-\eta\mathcal{B}_0 (t-\tau)\\
&=\eta\left\|w_{s, \tau}\right\|^{2}+2\eta\int_{\tau}^{t}\left\langle w_{s, r}, G d W(r)\right\rangle- \eta \nu \int_{\tau}^{t}\left\|w_{s, r}\right\|_{1}^{2} d r+2\eta\int_{\tau}^{t}\left\langle w_{s, r}, f(r)\right\rangle dr\\
&\leq \eta \left\|w_{s, \tau}\right\|^2+2\eta\int_{\tau}^{t}\left\langle w_{s, r}, G d W(r)\right\rangle - \frac{\eta\nu}{2}\int_{\tau}^{t}\left\|w_{s, r}\right\|_{1}^{2} d r+\frac{2\eta}{\nu}\|f\|_{\infty}^2 \left(t-\tau\right)
\end{align*}
Therefore
\begin{align}\notag
\eta\left\|w_{s, t}\right\|^{2}+\eta \nu \int_{\tau}^{t}\left\|w_{s,r}\right\|_{1}^{2} d r&-\eta C(f,\mathcal{B}_0) (t-\tau)\\\notag
&\leq  \eta \left\|w_{s, \tau}\right\|^2+2\eta\int_{\tau}^{t}\left\langle w_{s, r}, G d W(r)\right\rangle - \frac{\eta\nu}{2}\int_{\tau}^{t}\left\|w_{s, r}\right\|_{1}^{2} d r\\\label{appendixbounds01}
&\leq \eta \left\|w_{s, \tau}\right\|^2+2\eta\int_{\tau}^{t}\left\langle w_{s, r}, G d W(r)\right\rangle - \frac{\eta\alpha}{2}\int_{\tau}^{t}\left\|2G^*w\right\|^2 d r
\end{align}
Again the last inequality is obtained by choosing $\alpha>0$ so that $\frac{\nu}{2}\|w\|_1^2\geq \frac{\alpha}{2}\|2G^*w\|^2$. Setting $M(\tau, t) =\eta \left\|w_{s, \tau}\right\|^2+2\eta\int_{\tau}^{t}\left\langle w_{s, r}, G d W(r)\right\rangle$, then the right hand side of inequality \eqref{appendixbounds01} is $M(\tau, t) - \frac{\alpha}{2\eta}[M](\tau, t)$, where $[M](\tau, t)$ is  the quadratic variation of the continuous $\mathrm{L}^2$-martingale $M$.  Hence by the exponential martingale inequality, it follows that
\[\mathbf{P}\left(\sup _{t \geq \tau}\left( M(\tau, t) - \frac{\alpha}{2\eta}[M](\tau, t)\right) \geq K \,\middle\vert\, \mathcal{F}_{\tau}\right) \leq \exp \left(\eta\left\|w_{s, \tau}\right\|^{2}-\frac{\alpha K}{\eta}\right)\]
for all $\tau\geq s$.  As a consequence,
\[\mathbf{P}\left(\eta \sup _{t \geq \tau}\left(\left\|w_{s, t}\right\|^{2}+\nu \int_{\tau}^{t}\left\|w_{s, r}\right\|_{1}^{2} d r-C(f,\mathcal{B}_0)(t-\tau)\right)\geq K\,\middle\vert\,\mathcal{F}_{\tau} \right)\leq \exp \left(\eta\left\|w_{s, \tau}\right\|^{2}-\frac{\alpha K}{\eta}\right).\]
In view of \eqref{appendixbounds01}, we deduce that
\[\mathbf{E}\exp\left(\frac{\alpha}{2} \sup _{t \geq \tau}\left(\left\|w_{s, t}\right\|^{2}+\nu \int_{\tau}^{t}\left\|w_{s, r}\right\|_{1}^{2} d r-C(f,\mathcal{B}_0)(t-\tau)\right) \right)\leq 5\mathbf{E}\exp\left(\eta\|w_{s, \tau}\|^2\right)\]
The conclusion follows from \eqref{eq: est1} by taking $\eta_0 =\frac{\alpha}{2}  $.

\item Let $\mathcal{R}_t = w_{s, t}^{f_1} - w_{s, t}^{f_2}$ then
\begin{align*}
\partial_t\mathcal{R}_t = &\nu\Delta \mathcal{R}_t + B\left(\mathcal{K}w_{s, t}^{f_2}, w_{s, t}^{f_2}\right) - B\left(\mathcal{K}w_{s, t}^{f_1}, w_{s, t}^{f_1}\right)  + f_2 - f_1\\
&= \nu\Delta \mathcal{R}_t + B\left(\mathcal{K}\mathcal{R}_t, \mathcal{R}_t\right) - B\left(\mathcal{K}\mathcal{R}_t, w_{s, t}^{f_1}\right) - B\left(\mathcal{K}w_{s, t}^{f_1}, \mathcal{R}_t\right) + f_2 - f_1.
\end{align*}
Therefore
\begin{align*}
\partial_t\|\mathcal{R}_t\|^2 & = \left\langle \nu\Delta \mathcal{R}_t , 2\mathcal{R}_t\right\rangle - \left\langle B\left(\mathcal{K}\mathcal{R}_t, w_{s, t}^{f_1}\right) , 2\mathcal{R}_t\right\rangle + \left\langle f_2 -f_1 , 2\mathcal{R}_t\right\rangle\\
&\leq -2\nu\|\mathcal{R}_t\|_1^2 + C \|\mathcal{R}_t\|_{1/2}\|w_{s, t}^{f_1}\|_{1}\|\mathcal{R}_t\| + 2\|f_2-f_1\|_{\infty}\|\mathcal{R}_t\|\\
&\leq -2\nu\|\mathcal{R}_t\|_1^2 + \nu\|\mathcal{R}_t\|_1^2+ C(\eta, \nu)\|\mathcal{R}_t\|^{2} + \eta\nu \|w_{s, t}^{f_1}\|_{1}^2\|\mathcal{R}_t\|^2 + \nu\|\mathcal{R}_t\|_1^2 + \frac{1}{\nu}\|f_2-f_1\|_{\infty}^2\\
&\leq \left(C(\eta, \nu) +  + \eta\nu \|w_{s, t}^{f_1}\|_{1}^2 \right) \|\mathcal{R}_t\|^{2} +  \frac{1}{\nu}\|f_2-f_1\|_{\infty}^2.
\end{align*}
By the Gronwall's inequality, the bounds \eqref{eq: est1} and \eqref{enstrophy}, we have
\begin{align*}
\mathbf{E}\|\mathcal{R}_t\|^2&\leq \frac{1}{\nu}\|f_2-f_1\|_{\infty}^2\int_{s}^{t}\exp\Big(C(\eta, \nu)(t-\tau)\Big) \mathbf{E}\left[\exp\left(\int_{\tau}^{t}\eta\nu\|w_{s, r}^{f_1}\|_1^2\right)dr\right]d\tau\\
 &\leq\frac{C}{\nu}\|f_2-f_1\|_{\infty}^2\int_{s}^{t}\exp\Big(C(\eta, \nu)(t-\tau)\Big) \exp\left(\eta e^{-\nu(\tau-s)}\|w_0\|^2\right)d\tau\\
 &\leq C e^{\gamma(t-s)}\exp\left(\eta \|w_0\|^2\right)\|f_2-f_1\|_{\infty}^2
\end{align*}
with $\gamma = C(\eta, \nu)$.
\item The proof of the inequality \eqref{boundsum} is the same as that in \cite{HM06}, hence we omit it here.
\item  For any $\tau\geq s$ and initial condition $\xi\in H$, the evolution of $\xi_t: = J_{\tau,t}\xi$ is given by equation \eqref{eq: Jacobian}, which is a PDE with random coefficients. Taking $H$ inner product with $\xi_t$ and using the fact that $\langle B(\mathcal{K}w,\xi),\xi\rangle = 0$, we have
\[\partial_t\left\|\xi_t\right\|^2 = -2\nu \left\|\nabla\xi_t\right\|^2 + 2\langle B(\mathcal{K}\xi_t, w_{s, t}),\xi_t\rangle. \]
Then note
\begin{align*}
&\left|2\langle B(\mathcal{K}\xi_t, w_{s, t}),\xi_t\rangle\right|\leq C\left\|w_{s, t}\right\|_1\left\|\xi_t\right\|\left\|\xi_t\right\|_{\frac12}\leq C_{\eta}\left\|\xi_t\right\|_{\frac12}^2 + \frac{\eta}{2}\left\|w_{s, t}\right\|_1^2\left\|\xi_t\right\|^2\\
&\leq C_{\eta}\left(\varepsilon\left\|\xi_t\right\|^2+\frac{1}{\varepsilon^2}\left\|\xi_t\right\|_1^2\right)+ \frac{\eta}{2}\left\|w_{s, t}\right\|_1^2\left\|\xi_t\right\|^2 =C\left\|\xi_t\right\|^2+\nu\left\|\xi_t\right\|_1^2+ \frac{\eta}{2}\left\|w_{s, t}\right\|_1^2\left\|\xi_t\right\|^2,
\end{align*}
by choosing $\varepsilon = \sqrt{C_{\eta}/\nu}$ , where $C$ depends on $\eta , \nu$.
Therefore
\[\partial_t\left\|\xi_t\right\|^2 \leq -\nu\left\|\nabla\xi_t\right\|^2 +C\left\|\xi_t\right\|^2 + \frac{\eta}{2}\left\|w_{s, t}\right\|_1^2\left\|\xi_t\right\|^2\leq C\left\|\xi_t\right\|^2 + \frac{\eta}{2}\left\|w_{s, t}\right\|_1^2\left\|\xi_t\right\|^2. \]
And the result follows by the Gronwall's inequality.
\item Define $\zeta_{t}=\left\|\xi_{t}\right\|^{2}+\nu(t-\tau) \left\|\xi_{t}\right\|_{1}^{2}$.  From the equation \eqref{eq: Jacobian} for the Jacobian $\xi_t$,  one has
\begin{align*}
\frac12 \partial_{t}\left\|\xi_{t}\right\|_{1}^{2} &= -\nu \left\|\xi_{t}\right\|_2^{2}+\langle B(\mathcal{K}\xi_t, w_{s, t}),-\mathrm{\Delta} \xi_t \rangle+\langle B(\mathcal{K}w_{s, t}, \xi_t),-\mathrm{\Delta} \xi_t \rangle\\
&\leq -\nu \left\| \xi_{t}\right\|_{2}^{2}+C\left\|\xi_{t}\right\|_{1}\left\|w_{s, t}\right\|_{1}\left\|\xi_{t}\right\|_{3/2}.
\end{align*}
Therefore
\begin{align*}
\partial_t\zeta_t\leq&C\left\|\xi_t\right\|^2 + \frac{\eta}{2}\left\|w_{s, t}\right\|_1^2\left\|\xi_t\right\|^2-2(t-\tau)\nu^2 \left\| \xi_{t}\right\|_{2}^{2}+2(t-\tau) C\left\|\xi_{t}\right\|_{1}\left\|w_{s, t}\right\|_{1}\left\|\xi_{t}\right\|_{3/2}.
\end{align*}
From interpolation inequalities, one has
\begin{align*}
 2 C\|w\|_{1}\|\xi\|_{1}\|\xi\|_{3/ 2}
 \leq C_{\eta}\left\|\xi\right\|_{3/2}^2 + \frac{\eta}{2}\left\|w\right\|_1^2\left\|\xi\right\|_1^2 \leq \nu\|\xi\|_{2}^{2}+C\|\xi\|_1^{2}+\frac{\eta}{2}\|w\|_{1}^{2}\|\xi\|_{1}^{2}.
\end{align*}
As a consequence,
\begin{align*}
\partial_{t}\zeta_t\leq C\left\|\zeta_{t}\right\|^{2}+\frac{\eta}{2}\left\|w_{s, t}\right\|_{1}^{2}\left\|\zeta_{t}\right\|^{2}.
\end{align*}
Hence by Gronwall's inequality, for $s\leq\tau<t\leq s+T$,
\begin{align}\label{xi_t1}
\left\|\xi_{t}\right\|_{1}^{2} \leq \frac{C}{t-\tau} \exp \left(\eta \int_{\tau}^{t}\left\|w_{s, r}\right\|_{1}^{2} d r\right)\left\|\xi\right\|^{2},
\end{align}
where $C$ depends on $\nu, \eta ,T$. From basic Sobolev inequalities and interpolation inequalities, one has
\begin{align*}
\|\widetilde{B}(u, w)\| & \leq C\left(\|u\|_{1 / 2}\|w\|_{1}+\|u\|_{1}\|w\|_{1 / 2}\right) \\
& \leq C\left(\|u\|^{1 / 2}\|u\|_{1}^{1 / 2}\|w\|_{1}+\|w\|^{1 / 2}\|w\|_{1}^{1 / 2}\|u\|_{1}\right).
\end{align*}
It then follows from the definition of $K_{\tau, t}$ in \eqref{eq: Hessian} that for $t\in(\tau, \tau+1)$,
\begin{align*}
\left\|K_{\tau, t}\right\|&\leq \int_{\tau}^{t}\left\| J_{r, t}\right\|\left\| J_{\tau, r} \right\|^{\frac{1}{2}}\left\| J_{\tau, r} \right\|_1^{\frac{3}{2}}d r\\
&\leq \exp \left(\eta \int_{\tau}^{\tau+1}\left\|w_{s, r}\right\|_{1}^{2} d r+C\right) \exp\left(\frac34 \eta \int_{\tau}^{\tau+1}\left\|w_{s, r}\right\|_{1}^{2} d r\right)\int_{\tau}^{\tau+1} \frac{\tilde{C}}{\left(r-\tau\right)^{3/4}} d r\\
&\leq C \exp\left(\eta \int_{\tau}^{\tau+1}\left\|w_{s, r}\right\|_{1}^{2} d r\right),
\end{align*}
where we used inequalities \eqref{Jacobian} and \eqref{xi_t1} in the second step.  The result then follows from estimate \eqref{enstrophy}.
\end{enumerate}
\end{proof}
The following lemma in the velocity formulation when the deterministic force $f(t, x)$ is independent of $t$,  has been proved in \cite{KS12}. Here we give a  proof in the vorticity formulation and for the time dependent force $f(t, x)$.
\begin{lemma}\label{hb}
(\textbf{Higher order Sobolev norms of solution}) For any integer $k\geq 0$, set $\mathcal{E}_{w_0}(k,t,s) = \left(t-s\right)^k\|w_{s, t}\|_k^2+\nu\int_s^{t}\left(r-s\right)^k\|w_{s, r}\|_{k+1}^2dr.$ Suppose that in the equation
\eqref{NS}, $f\in \mathrm{L}^2_{\mathrm{loc}}(\mathbb{R}, H_k)$ and $g_{i}\in H_k$ for $i =1,\cdots, d$.  Then for any $m\geq 1$ and  $T> s$,  there is a constant $C=C(k, m,T-s, \nu, \|f\|_{\mathrm{L}^2([s,T], H_k)}, \mathcal{B}_k)>0$, such that
\[\mathbf{E}\sup_{s\leq t\leq T} \mathcal{E}_{w_0}(k,t,s)^m\leq C\exp\left(\|w_0\|^2\right). \]
\end{lemma}
\begin{proof}
The proof proceeds by induction as in \cite{KS12}. Let $L = -\mathrm{\Delta}$, $F_{k}(w) = (t-s)^k\|w\|_k^2=(t-s)^k\langle L^kw, w\rangle$.  We first prove the base case when $k=0$. Applying Ito's formula to the functional $F_0(w_{s, t})$,  and noting $\langle B(\mathcal{K}w, w),w\rangle = 0$, we have for $s\leq t\leq T$,
\begin{align*}
\|w_{s, t}\|^2 &= \|w_0\|^2 - 2\nu\int_s^t\|w_{s, r}\|_1^2dr+2\int_s^t\langle w_{s, r}, f(r)\rangle dr +\mathcal{B}_0(t-s) +M_t\\
 &\leq \|w_0\|^2 - 2\nu\int_s^t\|w_{s, r}\|_1^2dr+\frac{\nu}{2}\int_s^t\|w_{s, r}\|^2dr +\frac{2}{\nu}\|f\|^2_{\mathrm{L}^2([s,T], H)}+\mathcal{B}_0(T-s) +M_t
\end{align*}
where $M_t = 2\int_s^t\langle w_{s, r}, \sum_{i=1}^d g_i dW_i(r)\rangle$.  Note that the quadratic variation process of $M_t$  satisfies
\[[M]_t = \int_s^t4\sum_{i=1}^d\langle w_{s, r},  g_i \rangle^2 dr\leq 4\mathcal{B}_0\int_s^t\|w_{s, r}\|^2dr.\]
If we let $\sigma_0 =\frac{1}{4\mathcal{B}_0}$, $C_0 = \frac{2}{\nu}\|f\|^2_{\mathrm{L}^2([s,T], H)}+\mathcal{B}_0(T-s) $, then it follows that
\begin{align*}
\mathcal{E}_{w_0}(0,t,s) \leq \|w_0\|^2 +C_0+M_t-\frac{\sigma_0\nu}{2}[M]_t
\end{align*}
By supermartingale inequality, one has for any $K>0$
\[\mathbf{P}\left(\sup_{t\geq s}\left(M_t - \frac{\sigma_0\nu}{2}[M]_t\right)\geq K\right)\leq e^{-\sigma_0\nu K}.\]
Therefore
\[\mathbf{P}\left(\sup_{t\geq s}\left(\mathcal{E}_{w_0}(0,t,s) - C_0 - \|w_0\|^2\right)\geq K\right)\leq e^{-{\sigma_0\nu K}}\]
Note for non-negative randoms variables $\xi$ and $\eta$, one has
\begin{align*}
\mathbf{E}\xi^m \leq 2^m \left(\mathbf{E}\left(\xi-\eta\right)^m\mathbb{I}_{\{\xi>\eta\}}+\mathbf{E}\eta^m\right) = 2^m\int_0^{\infty}\mathbf{P}\{\xi-\eta>\lambda^{1/m}\}d\lambda+2^m\mathbf{E}\eta^m.
\end{align*}
Therefore
\begin{align}\label{Hmoments}
\mathbf{E}\sup_{t\geq s}\mathcal{E}_{w_0}(0,t,s) ^m \leq  2^m\int_0^{\infty}e^{-\sigma_0\nu\lambda^{1/m}}d\lambda +2^m\mathbf{E}(C_0 + \|w_0\|^2)^m\leq C\exp\left(\|w_0\|^2\right).
\end{align}
This finishes the proof of the base case.

\vskip0.05in

We now assume that $k = n\geq 1$ and that for $k\leq n-1$ the inequality has been proved. Applying Ito's formula to $F_n(w_{s, t})$ we find that
\begin{align*}
F_n(w_{s, t}) =& \int_s^t n(r-s)^{n-1}\|w_{s, r}\|_n^2 + 2(r-s)^n\langle L^n w_{s, r}, -\nu L w_{s, r}-B(\mathcal{K}w_{s, r},w_{s, r})+f(r)\rangle dr \\
&+ \frac{(t-s)^{n+1}}{n+1}\mathcal{B}_n + \int_s^t 2(r-s)^n\langle L^n w_{s, r}, \sum_{i=1}^d g_idW_i(r)\rangle.
\end{align*}
Note the quadratic variation process $[M]_t$ of the martingale $M_t: = \int_s^t 2(r-s)^n\langle L^n w_{s, r}, \sum_{i=1}^d g_idW_i(r)\rangle$ satisfies
\begin{align*}
[M]_t = \sum_{i=1}^d \int_s^t 4(r-s)^{2n}\langle L^n w_{s, r},  g_i\rangle^2dr\leq 4\mathcal{B}_n\left(T-s\right)^n\int_s^t\left(r-s\right)^n\|w_{s,  r}\|_n^2dr.
\end{align*}
Also note
\[\int_s^t 2(r-s)^n\langle L^nw_{s, r}, f(r)\rangle dr\leq \int_s^t 2(r-s)^n\|w_{s, r}\|_n\|f(r)\|_n dr\leq \int_s^t (r-s)^{n-1}\|w_{s, r}\|_n^2+(r-s)^{n+1}\|f(r)\|_n^2 dr. \]
Applying the inequality $\|w\|_n\leq\|w\|_{n+1}$, and combining these estimates,  it follows that
\begin{align*}
F_n(w_{s, t})\leq &(n+1)\int_s^t (r-s)^{n-1}\|w_{s, r}\|_n^2dr - \frac{3\nu}{2}\int_s^t(r-s)^n\|w_{s, r}\|_{n+1}^2dr \\
&-2\int_s^t  (r-s)^n\langle L^n w_{s, r},B(\mathcal{K}w_{s, r},w_{s, r})\rangle dr +C_n+ N_n(t),
\end{align*}
where $C_n = \left(T-s\right)^{n+1}\left(\|f\|^2_{\mathrm{L}^2([s,T], H_n)}+\frac{\mathcal{B}_n}{n+1}\right)$, $N_n(t)=M_t-\frac{\nu\sigma_n}{2}[M]_t$ and $\sigma_n =\frac{1}{4\mathcal{B}_n(T-s)^n}$.\\
When $n=1$, the nonlinear term has the following bounds by interpolation inequality,
\[2\langle Lw,B(\mathcal{K}w,w)\leq C\|w\|\|w\|_1\|w\|_{2}\leq \frac{\nu}{2}\|w\|_{2}^2+C\|w\|^2\|w\|_{1}^2.\]
Then one has
\begin{align*}
\sup_{s\leq t\leq T}\mathcal{E}_{w_0}(1,t,s) &\leq \frac{2}{\nu}\sup_{s\leq t\leq T}\mathcal{E}_{w_0}(0,t,s)+C\sup_{s\leq t\leq T}\int_s^t (r-s)\|w_{s, r}\|^2\|w_{s, r}\|_1^2dr + C_1+\sup_{t\geq s}N_1(t)\\
&\leq \frac{2}{\nu}\sup_{s\leq t\leq T}\mathcal{E}_{w_0}(0,t,s)+\frac{C(T-s)}{\nu}\sup_{s\leq t\leq T}\mathcal{E}_{w_0}(0,t,s)^2 + C_1 +\sup_{t\geq s}N_1(t).
\end{align*}
And the result then follows from the supermartingale inequality and the same argument as in the case $n=0$.
If $n\geq 2$, we use the following inequality to bound the nonlinear term
\[\left|\langle L^n w, B(\mathcal{K}w,w)\rangle \right|\leq C_n\|w\|_{n+1}^{\frac{4n-1}{2n}} \|w\|_1^{\frac{n+1}{2n}}\|w\|^{\frac12}\leq \frac{\nu}{2}\|w\|_{n+1}^2+C\|w\|_1^{2(n+1)}\|w\|^{2n},\]
which can be proved as Lemma 2.1.20 in \cite{KS12}.
It then follows that
\begin{align*}
\sup_{s\leq t\leq T}&\mathcal{E}_{w_0}(n,t,s) \\
&\leq \frac{n+1}{\nu}\sup_{s\leq t\leq T}\mathcal{E}_{w_0}(n-1,t,s)+C\sup_{s\leq t\leq T}\int_s^t (r-s)^n\|w_{s, r}\|^{2n}\|w_{s, r}\|_1^{2n+1}dr + C_1+\sup_{t\geq s}N_n(t)\\
&\leq \frac{n+1}{\nu}\sup_{s\leq t\leq T}\mathcal{E}_{w_0}(n-1,t,s)+\frac{C}{\nu}\sup_{s\leq t\leq T}\mathcal{E}_{w_0}(1,t,s)^n \sup_{s\leq t\leq T}\mathcal{E}_{w_0}(0,t,s)^{n+1}+ C_1 +\sup_{t\geq s}N_1(t).
\end{align*}
Then the desired result follows by induction hypothesis and the same reasoning as above.
\end{proof}

\subsection{The Lyapunov structure}\label{subsectionLyapunovstructure}
In this subsection, we prove that the Navier-Stokes system \eqref{NS} has a Lyapunov Structure that shows the contraction of the dynamics at large scales. Namely, we will prove Proposition \ref{pre_Lya} and an it's corollary.
\begin{proposition}[The Lyapunov Structure]\label{verifyLyapunov}
Let $V(w) = \exp(\eta\|w\|)$ and $\Phi_{s,s+t}$ be the stochastic flow generated by the Navier-Stokes system \eqref{NS}. Then there exists $\eta_0>0$ such that for any $\eta\in(0,\eta_0]$, we have the following\\
1. There are $\kappa> 1$ and $C>0$ such that $\|w\|V(w)\leq CV^{\kappa}(w) $ for all $w\in H$. \\
2. There is a constant $C>0$ such that for any fixed $0< r_0 <1$
\begin{align}\label{eqLya01}
\mathbf{E}V^r(\Phi_{s,s+t}(w))(1+\|\nabla\Phi_{s, s+t}(w))\xi\|)\leq CV^{r\alpha(t)}(w)
\end{align}
for every $w, \xi\in H$ with $\|\xi\|=1$, $r\in[r_0, 2\kappa]$ and every $t\in [0,\mathcal{T}], s\in \mathbb{R}$, where $\alpha(t) = e^{-\frac{\nu}{4}t}.$ \\
3. There is a constant $C$ (possibly larger) such that
\begin{align} \label{eqLya02}
\mathbf{E}V^r(\Phi_{s,s+t}(w))\leq C V^{r\alpha(t)}(w),
\end{align}
for every $t\geq 0$ and $r\in [r_0, 2\kappa]$. Here $\alpha(t)= \alpha(\mathcal{T})^{k}\alpha(\beta)$ is an extension to $\mathbb{R}$ of the function $\alpha(t)$ in \eqref{eqLya01},  where $k\in \mathbb{N}$ and $\beta\in [0,\mathcal{T})$ are unique numbers such that $t = k\mathcal{T}+\beta$.
\end{proposition}
The proof of Proposition \ref{verifyLyapunov} is based on a combination of the bounds from Lemma \ref{bounds} and the the following lemma from \cite{HM08}. Although the Wiener process here is $d$-dimensional, which is different from that in \cite{HM08}, the proof is essentially the same.
\begin{lemma}\cite{HM08}\label{lemmasemimartingale}
Let $U$ be a real valued semimartingale
\[d U(t,\omega) = F(t, \omega)dt+ \langle G(t,\omega), dW(t,\omega)\rangle,\]
where $W(t)$ is a standard two sided $d$-dimensional Brownian motion,  $G(t, \omega)\in \mathbb{R}^d$ and $\langle\cdot, \cdot\rangle$ is the inner product in $\mathbb{R}^d$. Assume there is a process $Z$ and positive constants $b_1, b_2, b_3$ with $b_2>b_3$ such that $F(t, \omega)\leq b_1-b_2Z(t, \omega)$, $U(t,\omega)\leq Z(t, \omega)$ and
$\|G(t, \omega)\|^2\leq b_3Z(t,\omega)$, then
\[\mathbf{E}\exp\left(U(t) + \frac{b_2e^{-\frac{b_2(t-s)}{4}}}{4}\int_{s}^{t}Z(\tau)d\tau\right)\leq \frac{b_2e^{4b_1/b_2}}{b_2-b_3}\exp\left(U(s)e^{-\frac{b_2(t-s)}{4}}\right),\]
for all $s\in\mathbb{R}, t\geq s$ and deterministic initial data $U(s)$.
\end{lemma}

\begin{proof}[Proof of Proposition \ref{verifyLyapunov}.]
Let $w_{s, t} = \Phi(t, \omega; s, w_0)$ be the solution of equation \eqref{NS} with initial data $w_{s, s} = w_0$. By Ito's formula, for $U(t) = \eta\|w_{s, t}\|^2$ ,one has
\begin{align*}
dU(t) &= \eta\left(\mathcal{B}_0+2\langle w_{s, t}, f(t)\rangle - 2\nu\|w_{s, t}\|_1^2\right)dt +\eta dM(t)\\
&\leq \eta\left(\mathcal{B}_0 + \|f\|_{\infty}^2/\nu - \eta\nu\|w_{s, t}\|_1^2\right)dt+\eta dM(t)
\end{align*}
where $M(t) = 2\int_s^t\langle w_{s, r}, \sum_{i = 1}^dg_idW_i(r)\rangle$ whose quadratic variation satisfies $[M](t) = 4\sum_{i=1}^d\langle w_{s, t}, g_i\rangle^2\leq 4\mathcal{B}_0\|w_{s, t}\|_1^2$. Setting $Z(t) = \eta\|w_{s, t}\|^2_1$, $b_1 = \eta\left(\mathcal{B}_0 + \|f\|_{\infty}^2/\nu\right)$, $b_2 = \nu$ and $b_3 = 4\mathcal{B}_0\eta^2$,  then $U(t)\leq Z(t)$, and $F(t) = \eta\left(\mathcal{B}_0+2\langle w_{s, t}, f(t)\rangle - 2\nu\|w_{s, t}\|_1^2\right)\leq b_1-b_22Z(t)$, and $\|G(t)\|_{\mathbb{R}^d}^2 = 4\eta^2\sum_{i=1}^d\langle w_{s, t}, g_i\rangle^2\leq b_3Z(t)$. Therefore by Lemma \ref{lemmasemimartingale}, we have for every $\eta<\sqrt{\frac{\nu}{4\mathcal{B}_0}}$,
\begin{align*}
\mathbf{E} \exp \left(\eta\|w_{s, t}\|^{2}+\frac{\nu}{4} e^{-\frac{\nu}{4}(t-s)} \int_{s}^{t} \eta\|w_{s, r}\|_{1}^{2} d r\right)\leq C\exp\left(\eta\|w_0\|^2e^{-\frac{\nu(t-s)}{4}}\right),
\end{align*}
where $C = \frac{\nu\exp\left(\frac{4\eta\left(\nu\mathcal{B}_{0}+\|f\|_{\infty}^2\right)}{\nu^2}\right)}{\nu-4\mathcal{B}_0\eta^2}$. From the bound on the Jacobian \eqref{Jacobian}, we have that for $t\in[0, \mathcal{T}]$ and any $\beta>0$, there is a constant $C = C(\nu, \beta,\mathcal{T} )$ such that
\[\|\nabla w_{s, s+t}(w_0)\| \leq C\exp\left(\beta\int_s^{s+t}\|w_{s, \tau}\|_1^2d\tau\right).\]
Therefore by taking $\beta = \frac{\nu\eta}{4}e^{-\frac{\nu\mathcal{T}}{4}}$, one has
\[\mathbf{E} \exp \left(\eta\|w_{s, s+t}\|^{2}\right)\left(1+\|\nabla w_{s, s+t}(w_0)\|\right)\leq C \exp\left(\eta\|w_0\|^2e^{-\frac{\nu t}{4}}\right).\]
The inequality \eqref{eqLya01} then folllows by  noting $\alpha(t) = e^{-\frac{\nu}{4}t}.$

\vskip0.05in

By inequality \eqref{eqLya01}, one has for any $s\in \mathbb{R}$, and $n\in \mathbb{N}$,
\[\mathcal{P}_{s+(n-1)\mathcal{T}, s+n\mathcal{T}}V^{r}(w) = \mathbf{E}V^r(\Phi_{s+(n-1)\mathcal{T},s+n\mathcal{T}}(w))\leq CV^{r\alpha(\mathcal{T})}(w).\]
Then by Jensen's inequality, we have
\begin{align*}
\mathcal{P}_{s+(n-2)\mathcal{T}, s+n\mathcal{T}}V^{r}(w)& =(\mathcal{P}_{s+(n-2)\mathcal{T}, s+(n-1)\mathcal{T}} \mathcal{P}_{s+(n-1)\mathcal{T}, s+n\mathcal{T}}V^{r})(w)\\
&\leq C \mathcal{P}_{s+(n-2)\mathcal{T}, s+(n-1)\mathcal{T}} V^{r\alpha(\mathcal{T})}(w)\\
&\leq C\left(\mathcal{P}_{s+(n-2)\mathcal{T}, s+(n-1)\mathcal{T}} V^r(w)\right)^{\alpha(\mathcal{T})}\leq C^{1+\alpha(\mathcal{T})}V^{r\alpha(\mathcal{T})^2}(w).
\end{align*}
 It follows by repeating this process that
\[\mathcal{P}_{s, s+n\mathcal{T}}V^r \leq C^{\sum_{j=0}^{n-1}\alpha(\mathcal{T})^j}V^{r\alpha(\mathcal{T})^n}.\]
Now for any $t\geq \mathcal{T}$, one has
\begin{align*}
\mathcal{P}_{s, s+t}V^r &= \mathcal{P}_{s, s+\beta}\mathcal{P}_{s+\beta, s+\beta+k\mathcal{T}}V^r \leq \mathcal{P}_{s, s+\beta}C^{\sum_{j=0}^{k-1}\alpha(\mathcal{T})^j}V^{r\alpha(\mathcal{T})^{k}}\\
&\leq C^{\sum_{j=0}^{k-1}\alpha(\mathcal{T})^j}(\mathcal{P}_{s, s+\beta}V^{r})^{\alpha(\mathcal{T})^{k}}\leq C^{\frac{1}{1-\alpha(\mathcal{T})}}V^{r\alpha(\mathcal{T})^{k}\alpha(\beta)}.
\end{align*}
The proof is then complete.
\end{proof}
With the help of this proposition, one can show the following lemma on a contraction property under the metric weighted by the Lyapunov function, which will be used in the proof of the contraction of $\mathcal{P}_{s, t}^*$ on $\mathcal{P}(H)$. In time homogeneous case \cite{HM08}, the authors proved a time one step contraction under the metric weighted by the Lyapunov function to show the contraction of the Markov semigroup on $\mathcal{P}(H)$.  Here  we  will prove and utilize the time $\mathcal{T}$ step contraction in order to adapt to the periodic structure appeared in the weak irreducibility. Recall that $\Phi_{s, s+t}$  is the stochastic flow generated by Navier-Stokes system \eqref{NS} and $\rho_r$ is as defined in \eqref{familymetrics}.
\begin{lemma}\label{twopoint}
 For any $r\in[r_0, 1]$, there are constants $\sigma\in (0,1)$ and $C, K>0$ such that
\begin{align}\label{eqtwopoint1}
\mathbf{E}\rho_r(\Phi_{s, s+t}(u), \Phi_{s, s+t}(v))\leq C\rho_r(u, v),
\end{align}
\begin{align}\label{eqtwopoint2}
\mathbf{E}\rho_r(\Phi_{s, s+n\mathcal{T}}(u), \Phi_{s, s+n\mathcal{T}}(v))\leq\sigma^n\rho_r(u, v)+K,
\end{align}
for any $s\in\mathbb{R}$, $n\in\mathbb{N}$, $t\in[0, \mathcal{T}]$ and $u, v\in H$.
\end{lemma}
\begin{proof}
 For any $\varepsilon>0$, there is a path $\gamma$ connecting $u$ and $v$ such that
\[\rho_r(u, v)\leq \int_0^1V^r(\gamma(t))\|\dot{\gamma}(t)\|dt\leq\rho_r(u, v)+\varepsilon. \]
Let $\tilde{\gamma}(\tau) = \Phi_{s, s+t}(\gamma(\tau))$ for $t\in [0, \mathcal{T}]$. Then by \eqref{eqLya01},
\begin{align*}
\mathbf{E}\rho_r(\Phi_{s, s+t}(u), \Phi_{s, s+t}(v))& \leq \mathbf{E}\int_0^1V^r(\tilde{\gamma}(\tau))\|\dot{\tilde{\gamma}}(\tau)\|d\tau\\
&\leq \mathbf{E}\int_0^1V^r(\tilde{\gamma}(\tau))\left\|\nabla \Phi_{s, s+t}(\gamma(\tau))\frac{\dot{\gamma}(\tau)}{\|\dot{\gamma}(\tau)\|}\right\|\|\dot{{\gamma}}(\tau)\|d\tau\\
&\leq C\int_0^1 V^{r\alpha(t)}(\gamma(\tau))\|\dot{\gamma}(\tau)\|d\tau\leq C\rho_r(u, v)+C\varepsilon.
\end{align*}
The inequality \eqref{eqtwopoint1} then follows since $\varepsilon$ is arbitrary. Now let $R$ be large enough so that $C V^{r\alpha(\mathcal{T})}(u)\leq \sigma  V^{r}(u)$ for some $\sigma\in (0, 1)$ and for all $u$ with $\|u\|\geq R$. Then
\begin{align}\notag
\mathbf{E}\rho_r(\Phi_{s, s+\mathcal{T}}(u), \Phi_{s, s+\mathcal{T}}(v))&\leq C\int_0^1 V^{r\alpha(\mathcal{T})}(\gamma(\tau))\|\dot{\gamma}(\tau)\|d\tau\\ \notag
&\leq \sigma\rho_r(u, v)+C\int_0^1 \mathbb{I}_{B_{R}(0)}(\gamma(\tau))V^{r\alpha(\mathcal{T})}(\gamma(\tau))\|\dot{\gamma}(\tau)\|d\tau+\sigma \varepsilon\\\label{eqtwopint01}
&\leq  \sigma\rho_r(u, v) + CV(R)\int_0^1 \mathbb{I}_{B_{R}(0)}(\gamma(\tau))\|\dot{\gamma}(\tau)\|d\tau+\varepsilon.
\end{align}
Note that if we set $u_0, v_0$ to be the points where $\gamma$ first enters and last exits the ball $B_{R}(0)$, then we have
\begin{align}\label{eqtwopoint3}
\int_0^1 \mathbb{I}_{B_{R}(0)}(\gamma(\tau))\|\dot{\gamma}(\tau)\|d\tau\leq \int_0^1 \mathbb{I}_{B_{R}(0)}(\gamma(\tau))V^r(\gamma(\tau))\|\dot{\gamma}(\tau)\|d\tau<\rho_r(u_0, v_0)+ \varepsilon.
\end{align}
The second inequality of \eqref{eqtwopoint3} is true since otherwise, one has
\[\int_0^1 \mathbb{I}_{B_{R}(0)}(\gamma(\tau))V^r(\gamma(\tau))\|\dot{\gamma}(\tau)\|d\tau\geq \rho_r(u_0, v_0)+ \varepsilon,\]
so that $\int_0^1 V^r(\gamma(\tau))\|\dot{\gamma}(\tau)\|d\tau\geq\rho_r(u, v)+ \varepsilon$ by observing that
\[\int_0^1 \mathbb{I}_{B_{R}^c(0)}(\gamma(\tau))V^r(\gamma(\tau))\|\dot{\gamma}(\tau)\|d\tau\geq \rho_r(u, u_0)+ \rho_r(v, v_0),\]
which contradicts to the choice of the path $\gamma(\tau)$.
Note that by considering the straight line that connects $u_0$ and $v_0$, we have  $\rho_r(u_0, v_0)\leq 2RV^r(R)$. Hence
\begin{align}\label{eqtwopint02}
\int_0^1 \mathbb{I}_{B_{R}(0)}(\gamma(\tau))\|\dot{\gamma}(\tau)\|d\tau\leq 2RV^r(R) + \varepsilon.
\end{align}
By arbitrariness of $\varepsilon$, we conclude from \eqref{eqtwopint01} and \eqref{eqtwopint02} that for $K = 2CRV^{r+1}(R)$.
\[\mathbf{E}\rho_r(\Phi_{s, s+\mathcal{T}}(u), \Phi_{s, s+\mathcal{T}}(v)) \leq \sigma\rho_r(u, v)  + K.\]
  For $n>1$, we apply the Markov property. Observing that
\begin{align*}
&\mathbf{E}\left[\rho_r(\Phi_{s, s+n\mathcal{T}}(u), \Phi_{s, s+n\mathcal{T}}(v))|\mathcal{F}_{s+(n-1)\mathcal{T}}\right] \\
&= \mathbf{E}\left[\rho_r(\Phi_{s+(n-1)\mathcal{T}, s+n\mathcal{T}}(\Phi_{s, s+(n-1)\mathcal{T}}(u)), \Phi_{s+(n-1)\mathcal{T}, s+n\mathcal{T}}(\Phi_{s, s+(n-1)\mathcal{T}}(v)))\right]\\
&\leq \sigma \rho_r(\Phi_{s, s+(n-1)\mathcal{T}}(u), \Phi_{s, s+(n-1)\mathcal{T}}(v)) +K.
\end{align*}
And
\begin{align*}
&\mathbf{E}\left[\rho_r(\Phi_{s, s+n\mathcal{T}}(u), \Phi_{s, s+n\mathcal{T}}(v))|\mathcal{F}_{s+(n-2)\mathcal{T}}\right] \\
&\leq \sigma \mathbf{E}\left[\rho_r(\Phi_{s+(n-2)\mathcal{T}, s+(n-1)\mathcal{T}}(\Phi_{s, s+(n-2)\mathcal{T}}(u)), \Phi_{s+(n-2)\mathcal{T}, s+(n-1)\mathcal{T}}(\Phi_{s, s+(n-2)\mathcal{T}}(v)))\right]\\
&\leq \sigma^2 \rho_r(\Phi_{s, s+(n-2)\mathcal{T}}(u), \Phi_{s, s+(n-2)\mathcal{T}}(v)) + (\sigma+1)K.
\end{align*}
By iterating the procedure, we obtain that
\begin{align*}
\mathbf{E}\rho_r(\Phi_{s, s+n\mathcal{T}}(u), \Phi_{s, s+n\mathcal{T}}(v))  \leq \sigma^n \rho_r(u, v) + K\sum_{j =0}^{n-1}\sigma^j,
\end{align*}
which completes the proof.
\end{proof}
\begin{remark}
From the proof we know that the Lyapunov structure in Proposition \ref{verifyLyapunov} is independent of the periodic structure of system \eqref{NS}. One can replace $\mathcal{T}$ by any fixed positive real number to yield the same estimates.
\end{remark}

\subsection{The gradient estimate}\label{gradientinequality}

In this subsection, we will show that the transition operator $\mathcal{P}_{s,t}$ has the following gradient estimate.

\begin{proposition}\label{gradientinequalityprop}
Assume $A_{\infty}=H$. Then there exists constant $\eta_0>0$ so that for every $\eta\in (0, \eta_0]$ and  $a>0$
there exists constants $ C = C(\eta, a)>0$ and $p\in(0,1)$ such that
\begin{align}\label{gradientprop}
\left\|\nabla \mathcal{P}_{s, s+t} \phi(w)\right\| \leq C \exp(p\eta\|w\|^2)\left(\sqrt{\left(\mathcal{P}_{s, s+t}|\phi|^{2}\right)(w)}+ e^{-at} \sqrt{\left(\mathcal{P}_{s, s+t}\|\nabla \phi\|^{2}\right)(w)}\right)
\end{align}
for every  Frechet differentiable function $\phi$, every $w \in H, s\in\mathbb{R}$ and  $t\geq 0$. Here $ C(\eta, a)$ does not depend on initial condition $(s, w)$ and $\phi$.
\end{proposition}
\begin{proof}
The proof is a combination of inequality \eqref{gradient}, Proposition \ref{errorbounds} and Proposition \ref{control} below. Note that the proof and result do not depend on the periodicity.
\end{proof}

The general scheme for the proof of the above gradient inequality is quite standard in the time homogeneous case after the groundbreaking works \cite{HM06,HM08,HM11,MP06}. However, there is no known proof for the time inhomogeneous case in the literature, hence we supply a proof in this subsection, using the same arguments as that in the time homogeneous case. We first apply the integration by parts formula from the theory of Malliavin calculus, to transfer the variation on the initial condition in a solution to a variation $v$ on the Wiener path.  The problem of obtaining estimate \eqref{gradientprop} is then reduced to finding an appropriate $v$ with bounded cost to approximately compensate the variation on the initial condition, so that the error of the two variations in the solution goes to $0$ as time goes to infinity.

\vskip0.05in

The invertibility of the Malliavin matrix is crucial when constructing such a desired control $v$. However, it is not easy to verify the invertibility in the infinite dimensional case, hence the inverse of its Tikhonov regularization is taken into consideration. Moreover, since the noise here is extremely degenerate, the unstable directions of the system are not directly forced. Hence one needs an infinite dimensional version of Lie bracket condition $A_{\infty} = H$ as in Hörmander's theorem to ensure the propagation of the noise to those unstable directions and obtain a spectral estimate of the Malliavin matrix to control the dynamics on the determining modes. Since there is no existing proof of the spectral property of the the Malliavin matrix in the time inhomogeneous setting, we give a proof here.
\subsubsection{The Malliavin matrix}\label{TheMalliavinmatrix}
In this subsection, we recall several facts about the Malliavin matrix and give a specific description about the construction of the control $v$.  To introduce the Malliavin derivative of the solution process, we first consider its linearized equations.

\vskip0.05in

As in \cite{HM06,MP06}, the linearized flow $J_{\tau, r}\xi\in C\left([\tau, t] ; H\right) \cap \mathrm{L}^2\left((\tau, t] ;H_1\right)$ is the solution to the  equation
 \begin{equation}\label{eq: Jacobian}
 \partial_rJ_{\tau, r}\xi = \nu\mathrm{\Delta} J_{\tau, r}\xi+\widetilde{B}(w_{s, r} ,J_{\tau, r}\xi),\quad  r>\tau, \quad J_{\tau ,\tau}\xi = \xi.
 \end{equation}
 for any $r>\tau\geq s$ and $\xi\in H$, where $\widetilde{B}(u,w)=-B(\mathcal{K}u,w)-B(\mathcal{K}w,u)$, where $w_{s, r}: =w(r, \omega; s, w_0)$ is the solution of \eqref{NS}, starting at initial time $s$.

\vskip0.05in

It is also helpful to consider  the time-reversed, $H$-adjoint $U^{t, (\cdot)}(r)$ of $J_{r, t}(\cdot)$ to analyze the Malliavin matrix. Here we use the notation $U^{t, \varphi}(r)$ to emphasize that the time $t$ is the initial time and $\varphi\in H$ the is the initial data, and the process $U^{t, \varphi}(r)$ runs backward in time for $s\leq r\leq t$. It follows that $U^{t, \varphi}(r)\in C\left([s, t] ; H\right) \cap \mathrm{L}^2\left([s, t) ;H_1\right)$ is the unique solution to the backward random PDE
\begin{align}\label{reversed}
\left\{
\begin{array}
{l}\partial_r U^{t, \varphi}(r)= -\nu \mathrm{\Delta} U^{t, \varphi}(r) + B\left(\mathcal{K}w_{s, r}, U^{t, \varphi}(r)\right) +C\left(\mathcal{K}U^{t, \varphi}(r), w_{s, r}\right), \quad s \leq r < t,  \\
 U^{t, \varphi}(t)=\varphi.
\end{array}\right.
\end{align}
Here $C\left(\mathcal{K}\left(\cdot\right), w_{s, r}\right)$ is the adjoint of $B\left(\mathcal{K}\left(\cdot\right), w_{s, r}\right)$ determined by the relation $\langle B\left(\mathcal{K}u, w_{s, r}\right), v\rangle = \langle C\left(\mathcal{K}v, w_{s, r}\right), u\rangle$.

\vskip0.05in

The second derivative $K_{\tau,t}$ of  $w_{s, r}$ with respect to its initial condition is the solution of the following equation
\begin{align}\label{eq: Hessian}
\left\{
\begin{array}{l}
 \partial_{t} K_{\tau, t}\left(\xi, \xi^{\prime}\right) =\nu \mathrm{\Delta} K_{\tau, t}\left(\xi, \xi^{\prime}\right)+\widetilde{B}\left(w_{s,t}, K_{\tau, t}\left(\xi, \xi^{\prime}\right)\right)+\widetilde{B}\left(J_{\tau,t} \xi^{\prime}, J_{\tau, t} \xi\right),\\
 K_{\tau, \tau}\left(\xi, \xi^{\prime}\right) =0 .
 \end{array}\right.
 \end{align}
By the variation of constants formula $K_{\tau, t}\left(\xi, \xi^{\prime}\right)$ is given by
\[K_{\tau, t}\left(\xi, \xi^{\prime}\right)=\int_{\tau}^{t} J_{r, t} \widetilde{B}\left(J_{\tau, r} \xi^{\prime}, J_{\tau, r} \xi\right) d r.\]

\vskip0.05in

Note that the solution $w_{s, t}(\omega, w_0)$ is a functional of the two sided Wiener process via the Ito map $\mathrm{\Phi}_{s, t}^{w_0}: C([s, t], \mathbb{R}^d)\rightarrow H$ with  $w_{s, t}(\omega, w_0) = \mathrm{\Phi}_{s, t}^{w_0}\left(W_{[s, t]}\right)$, where $W_{[s, t]}$ is the restriction of the Wiener process on $[s, t]$. And the Cameron-Martin space associated with the Wiener space $(\mathrm{\Omega, \mathcal{F},\mathbf{P}})$ is
\[\mathcal{CM}=\left\{V\in \mathrm{L}^{2}\left(\mathbb{R}, \mathbb{R}^{d}\right): \partial_tV\in \mathrm{L}^{2}\left(\mathbb{R}, \mathbb{R}^{d}\right),\, V(0)=0\right\},\]
endowed with the norm $\|V\|_{\mathcal{CM}}^2:=\int_{\mathbb{R}}|\partial_tV(t)|_{\mathbb{R}^d}^2dt$, which is a Hilbert space isometric to $\mathcal{CM}' = \mathrm{L}^2(\mathbb{R},\mathbb{R}^d)$.
As section 4.1 in \cite{HM11}, for any $V\in \mathcal{CM}$,   denote the directional derivative of the $H$ valued random variable $w_{s, t}$ along the direction $V$ as
\[\mathrm{D}\mathrm{\Phi}_{s, t}^{w_0}V: = \lim_{\varepsilon\rightarrow 0}\frac{\mathrm{\Phi}_{s, t}^{w_0}(W+\varepsilon V)-\mathrm{\Phi}_{s, t}^{w_0}(W)}{\varepsilon}, \]
which exists and satisfies
\begin{align}\label{direction}
\mathrm{D}\mathrm{\Phi}_{s, t}^{w_0}V = \int_{s}^tJ_{r,t}GV'(r)dr.
\end{align}
Now for any $v\in \mathcal{CM}'$, define $V(t)=\int_0^tv(t)dt$.  Then $D^vw_{s, t}=D^v\mathrm{\Phi}_{s, t}^{w_0}:= \mathrm{D}\mathrm{\Phi}_{s, t}^{w_0}V$ is called the Malliavin derivative of the random variable $w_{s, t}$ in the direction $v$. Since $D^vw_{s, t}$ is a (random) bounded linear operator from $\mathcal{CM}'$ to $H$, by Riesz's representation theorem, there exists a random element, $Dw_{s, t}\in\mathcal{CM}'\otimes H$ such that for every $v\in \mathcal{CM}'$,
 \[D^vw_{s, t} = \langle Dw_{s, t}, v\rangle_{\mathcal{CM}'}=\int_{\mathbb{R}}\left(D_rw_{s, t}\right)v(r)dr.\]
The random element $Dw_{s, t}$ is the Malliavin derivative of $w_{s, t}$, which can be regarded as a stochastic process $\left(D_rw_{s, t}\right)_{r\in\mathbb{R}}$ with values in $\mathbb{R}^d\otimes H$.  From equation \eqref{direction}, we see that  $D_rw_{s, t} = J_{r,t}G$ for $r\in[s, t]$ and $D_rw_{s, t} = 0$ for $r\in\mathbb{R}\setminus[s, t]$. The operator $D: \mathrm{L}^2(\mathrm{\Omega},\mathbb{R})\otimes H\rightarrow \mathrm{L}_{\mathrm{ad}}^2(\mathrm{\Omega},\mathcal{F}_t,\mathcal{CM}')\otimes H$ is actually a closed unbounded  linear operator, which is called the Malliavin derivative. Here $\mathrm{L}_{\mathrm{ad}}^2$ is the space of $\mathrm{L}^2$ functions adapted to the filtration $\mathcal{F}_t$ \cite{HM11}.

\vskip0.05in

For any $t\geq \tau\geq s$, define the random operator $A_{\tau, t}: \mathcal{CM}'\rightarrow H$ by
 \[A_{\tau,t}v:= \langle Dw_{s, t}, v\mathbb{I}_{[\tau,t]}\rangle_{\mathcal{CM}'} = \int_{\tau}^tJ_{r,t}Gv(r)dr,\]
 and its adjoint $A_{\tau, t}^*$ by the relation $\langle A_{\tau, t}v, u\rangle_H =\langle v,A_{\tau,t}^*u\rangle_{\mathcal{CM}'}.$  Then the Malliavin matrix is defined as $M_{\tau, t}: =A_{\tau,t}A_{\tau,t}^* $. We have for $\xi\in H$,
 \begin{align}\label{backrep}
 \langle M_{\tau, t}\xi, \xi\rangle = \sum_{i=1}^d\int_{\tau}^t \langle J_{r, t}g_i,\xi\rangle^2dr=\sum_{i=1}^{d} \int_{\tau}^{t}\left\langle g_{i},  U^{t, \xi}(r) \right\rangle^{2} d r,
 \end{align}
where the second identity is due to the fact that $U^{t, (\cdot)}(\tau)$ is the adjoint of $J_{\tau,t}$ in $H$, which has been proved in \cite{HM11,MP06}.

\vskip0.05in

For any $v\in\mathcal{CM}'$, denote by $v_{\tau,t}=v\mathbb{I}_{[\tau,t]}$ the restriction of $v$ on the interval $[\tau,t]$. Set the error
\begin{align}\label{errorterm}
\mathfrak{R}_{t} = J_{s, t} \xi-A_{s, t} v_{s, t}
\end{align}
caused by the infinitesimal variation on the Wiener path $W$ by $v$ that is used to compensate the variation on the initial condition of the solution process.  Applying integration by parts formula \cite{Nua06} for the Malliavin derivative,  we have for any Fréchet differentiable $\varphi: H\rightarrow \mathbb{R}$, any initial condition $w_0\in H$ and $\alpha := t - s\geq0$
\begin{align}\left\langle\nabla \mathcal{P}_{s, t} \varphi(w_0), \xi\right\rangle \notag
&=\mathbf{E}\left\langle\nabla\Big(\varphi\big(w_{s, t}(w_0)\big)\Big), \xi\right\rangle=\mathbf{E}\Big(\big(\nabla \varphi\big)\big(w_{s, t}(w_0)\big) J_{s, t} \xi\Big) \\\notag
 &=\mathbf{E}\Big(\big(\nabla \varphi\big)\big(w_{s, t}(w_0)\big)  A_{s, t} v_{s, t}\Big)+\mathbf{E}\Big(\big(\nabla \varphi\big)\big(w_{s, t}(w_0)\big) \mathfrak{R}_t\Big)\\\notag
  &=\mathbf{E}\Big(D^{v_{s,t}} \varphi\big(w_{s, t}(w_0)\big) \Big) +\mathbf{E}\Big(\big(\nabla \varphi\big)\big(w_{s, t}(w_0)\big) \mathfrak{R}_t\Big)\label{ineqASF} \\
  &=\mathbf{E}\left(\varphi\big(w_{s, t}(w_0)\big) \int_{s}^{t} v(r) d W(r)\right)+\mathbf{E}\Big(\big(\nabla \varphi\big)\big(w_{s, t}(w_0)\big) \mathfrak{R}_t\Big) \\
  &\leq \left( \mathbf{E}\left|\int_{s}^{s+\alpha} v(r) d W(r)\right|^{2}\right)^{1/2} \sqrt{\mathcal{P}_{s, s+\alpha }|\varphi|^2(w_0)}
+\sqrt{\mathcal{P}_{s, s+\alpha }\|\nabla\varphi\|^2(w_0)}\left(\mathbf{E}\|\mathfrak{R}_{s+\alpha }\|^2\right)^{1/2}, \label{gradient}
\end{align}
where we used Hölder's  inequality at the last step. Fix $\|\xi\| =1$, where $\xi\in H$ represents the direction of  the variation of the solution on the initial condition. To show the gradient inequality,  we will choose an appropriate random process $v$  with sample paths  in $\mathcal{CM}'$ to make sure the existence of constant $C>0$ and $p\in (0, 1)$,  such that
\begin{equation}\label{malliavin-costanderror}
\begin{array}{l}
\mathbf{E}\left|\int_{s}^{s+\tau} v(r) d W(r)\right|<C\exp(p\eta\|w_0\|) ,\\
\mathbf{E}\left\|\mathfrak{R}_{s+\tau}\right\|\leq Ce^{-a\tau}\exp(p\eta\|w_0\|).
\end{array}
\end{equation}
for some $a>0$ and all $\tau\geq0, w_0\in H$.
Note once we fixed the initial time $s$, the values of $v$ before $s$ do not affect the gradient estimate \eqref{gradient}.  Hence we will set $v(r)=0$ for $r<s$ and mainly focus on the construction of $v$ after the initial time $s$.

\vskip0.05in

The proof of  Proposition \ref{gradientinequalityprop} is then reduced to find such an appropriate control $v$, which involves the inverse of the Malliavin  matrix. However, it is unclear that if the Malliavin matrix is invertible or not in the present infinite dimensional setting, therefore we consider its Tikhonov regularization  $\widetilde{M}_{\tau, t}:= M_{\tau, t}+\beta$ for small constant $\beta>0$, which is  invertible.  For integer values $n\geq 0$, define $J_{n}=J_{s+n, s+n+1}$,  $A_{n}=A_{s+n, s+n+1}$, $M_{n}=A_{n} A_{n}^{*}$, $\widetilde{M}_{n}=\beta+M_{n}$.  The process $v$ is then recursively defined as
\begin{align}\label{controlv}
v(r) = \left\{
\begin{array}{ll}
A_{2n}^*\widetilde{M}_{2n}^{-1}J_{2n}\mathfrak{R}_{s+2n}&\text{ for } r\in[s+2n, s+2n+1), n\geq 0,\\
0&\text{ for }r\in[s+2n+1, s+2n+2), n\geq 0,
\end{array}\right.
\end{align}
where $\mathfrak{R}_s=\xi$, and $\mathfrak{R}_{t} = J_{s, t} \xi-A_{s, t} v_{s,t}$.  The definition is not circular since the construction of $v(r)$ for $r\in[s+2n, s+2n+2]$ only requires the knowledge of $\mathfrak{R}_{s+2n}$, which depends only on $v(r)$ for $r\in[s, s+2n]$. For instance, known $\mathfrak{R}_s=\xi$, we obtain the definition of $v(r)$ for $r\in[s, s+2]$ from formula \eqref{controlv}, and $ \mathfrak{R}_{t} = J_{s, t} \xi-A_{s, t} v_{s, t}$ for $t\in[s, s+2]$. Then we use $\mathfrak{R}_{s+2}$ to construct $v(r)$ for $r\in[s+2,s+4]$ and iterate this procedure.
\subsubsection{Estimate of the error $\mathfrak{R}$ and the control $v$}\label{section gradient}
Now it remains to check \eqref{malliavin-costanderror}. This will be accomplished through Proposition \ref{errorbounds} and Proposition \ref{control}.
We first establish several basic lemmas.

\vskip0.05in

The following lemma is a version of the well known Foias–Prodi estimate. It shows that the linearized system of equation \eqref{NS} has only finite number of unstable directions along the low modes. The proof of the asymptotic regularizing inequality \eqref{gradientprop} relies on an estimate of the spectrum of the Malliavin matrix on such determining modes.
\begin{lemma}\label{highmodes}
For any constants $p \geq 1,T, \gamma, \eta>0,$ there exists an orthogonal projection $\pi_{\ell}: = \pi_{\ell}(p, T, \gamma, \eta)$ onto a finite dimensional subspace of $H$  such that
\begin{equation}
\begin{array}{l}
\mathbf{E}\left\|\left(1-\pi_{\ell}\right) J_{s, s+T}\right\|^{p} \leq \gamma \exp \left(\eta\left\|w_0\right\|^{2}\right),\\
\mathbf{E}\left\|J_{s, s+T}\left(1-\pi_{\ell}\right)\right\|^{p} \leq \gamma \exp \left(\eta\left\|w_0\right\|^{2}\right),
\end{array}
\end{equation}
for every $w_0\in H$.
\end{lemma}
\begin{proof}
Let $\{\lambda_n\}$ be the eigenvalues of $-\mathrm{\Delta}$ associated with $\eqref{NS}$, and $\mathrm{\Pi}_N$ the projection onto the subspace of $H$ spanned by the first $N$ eigenfunctions. Let  $\mathrm{\Pi}_N^{\perp}=1-\mathrm{\Pi}_N$.
Since $\left\|\mathrm{\Pi}_{N}^{\perp} J_{s, s+T}\xi\right\| \leq \frac{1}{N}\left\|J_{s, s+T}\xi\right\|_{1}$, from bound \eqref{xi_t1} for the linearization flow and \eqref{enstrophy}, we obtain
\begin{align*}
\mathbf{E}\left\|\mathrm{\Pi}_{N}^{\perp} J_{s,s+T}\right\|^p \leq \frac{1}{N^p}\mathbf{E}\left\|J_{s, s+T}\right\|_{1}^p\leq \gamma \exp \left(\eta\left\|w_0\right\|^{2}\right),
\end{align*}
for any $\gamma>0$, by choosing $N$ sufficiently large. It follows from Proposition 6.1 in \cite{CF88} that
\begin{align}\label{B(u,v)-1/4}
\|\widetilde{B}(w_{s, r}, J_{s, r}\xi)\|_{-1/4}\leq \|w_{s, r}\|\|J_{s, r}\xi\|_1+\|w_{s, r}\|_1\|J_{s, r}\xi\|.
\end{align}
Denote $\bar{\xi}_N=\mathrm{\Pi}_N^{\perp}\xi$. From inequality \eqref{B(u,v)-1/4}, equation \eqref{eq: Jacobian} and using the variation of constant formula, and the analyticity of $e^{t\mathrm{\Delta}}$, we have that
\begin{align}\notag
\|J_{s, s+t}\bar{\xi}_N\| &=\left\|e^{\nu t\mathrm{\Delta}}\bar{\xi}_N + \int_s^{s+t}e^{\nu(s+t-r)\Delta}\widetilde{B}(w_{s, r}, J_{s, r}\bar{\xi}_N)dr\right\| \\\notag
&\leq \|e^{\nu t\mathrm{\Delta}}\bar{\xi}_N\| +C \int_s^{s+t} (s+t-r)^{-1/4}\|\widetilde{B}(w_{s, r}, J_{s, r}\bar{\xi}_N)\|_{-1/4}dr\\\notag
&\leq e^{-\nu t\lambda_{N+1}}\|\xi\|+C\sup_{r\in[s,s+T]}C(r)\int_s^{s+t}\left(s+t-r\right)^{-1/4}(r-s)^{-1/2}dr\\\notag
&\leq e^{-\nu t\lambda_{N+1}}\|\xi\|+C\sup_{r\in[s,s+T]}C(r)\left(\int_s^{s+t}\left(s+t-r\right)^{-3/4}dr\right)^{1/3}\left(\int_s^{s+t}(r-s)^{-3/4}dr\right)^{2/3}\\\label{Jortho}
&\leq e^{-\nu t\lambda_{N+1}}\|\xi\|+t^{1/4}C\sup_{r\in[s, s+T]}C(r),
\end{align}
where $C(r) = \left(\|w_{s, r}\|\|J_{s, r}\xi\|_1+\|J_{s, r}\xi\|\|w_{s, r}\|_1\right)(r-s)^{1/2}$. It then follows from the bound \eqref{xi_t1}, the estimates \eqref{enstrophy} and \eqref{Jacobian} from Lemma \ref{bounds} and Lemma \ref{hb} that there exists a constant $C>0$ independent of $s$, such that
\begin{align}\label{C(r)}
\sup_{r\in[s, s+T]}C(r) \leq C\exp\left({\eta\|w_0\|^2}\right).
\end{align}
Hence for every $p\geq 1, \gamma>0, \eta>0$, by first choosing sufficiently small $t>0$ and then choosing sufficiently large $N$, we have by \eqref{Jortho} and \eqref{Jacobian}, \eqref{enstrophy}  that
\begin{align*}
\mathbf{E}\|J_{s, s+T}\mathrm{\Pi}_{N}^{\perp} \|^p\leq \left(\mathbf{E}\|J_{s+t, s+T}\|^{2p}\mathbf{E}\|J_{s, s+t}\mathrm{\Pi}_{N}^{\perp} \|^{2p}\right)^{1/2}\leq\gamma \exp \left(\eta\left\|w_0\right\|^{2}\right).
\end{align*}
The proof is complete by setting $\pi_{\ell}=\mathrm{\Pi}_{N}$ for such a large enough $N$.
\end{proof}
We need the following important  result about the spectrum property of the Malliavin matrix over the unstable modes to have  desired controls on the dynamics. The same result in the time homogeneous setting has been obtained in  \cite{BM07,HM06,HM11,MP06}. Since there is no known proof for a time inhomogeneous system as \eqref{NS}, we supply a proof in the appendix.
\begin{theorem}\label{thmmatt}
 Assume $A_{\infty} = H$. For any $p\geq 1$ , positive $\alpha, \eta, n$,  and any orthogonal projection $\mathrm{\Pi}: H\rightarrow H$ on a finite finite dimensional subspace of $H$, there exists $C=C(p,n, \eta, \nu, \mathrm{\Pi}, f, \mathcal{B})$  and $\varepsilon_0 = \varepsilon_0(n,\alpha, \mathrm{\Pi}, f, \mathcal{B})$ such that
\[
\mathbf{P}\left(\langle M_{s,s+n} \varphi, \varphi\rangle<\varepsilon\|\varphi\|^{2}\right) \leq C \varepsilon^{p} \exp \left(\eta\left\|w_0\right\|^{2}\right)
\]
holds for every $($random$)$ vector $\varphi \in H $ satisfying $\left\|\mathrm{\Pi} \varphi\right\| \geq \alpha\|\varphi\|$ almost surely,
for every $\varepsilon \in(0,\varepsilon_0), s\in \mathbb{R}$,  and for every $w_0\in H$.
\end{theorem}
\begin{proof}
The proof is based on the approach in section 6 of \cite{HM11} along with some bounds on the solution of the Navier-Stokes equation \eqref{NS} and its linearization.
Fixing $T>0$, we will consider the problem in the interval $[s, s+T]$. To avoid the singularities at the terminal times, we introduce $I_{\delta}:=[s+\frac{T}{2},s+T-\delta]$, where $\delta= \frac{T}{4}\varepsilon^r$ for $0<\varepsilon<1$ and some $r>0$ that will be determined later.  Also for $\alpha \in(0,1)$, and for a given orthogonal projection $\mathrm{\Pi}: H \rightarrow H,$ we define $S_{\alpha} \subset H$ by
\[S_{\alpha}=\{\varphi \in H\backslash\{0\}:\|\mathrm{\Pi} \varphi\| \geq \alpha\|\varphi\|\}.\]
The following estimates about the process $U^{t, \varphi}$ in the time homogeneous case has been proved in \cite{HM11}. Since the setting is a bit different here due to time inhomogeneity, we give the proof below for the reader's convenience.
\begin{lemma}\label{back1}
For any $\delta \in(0, T / 2]$, $p>0, \eta>0$,  one has the bound
\[\mathbf{E}\sup_{\|\varphi\|\leq 1}\left\|U^{s+T, \varphi}(s+T-\delta)-e^{\delta\nu\mathrm{\Delta}} \varphi\right\|^{2p} \leq C \delta^{p}\exp\left(p\eta\|w_0\|^2\right), \]
\[\mathbf{E}\sup_{\|\varphi\| \leq 1}\left\|U^{s+T, \varphi}(s+T-\delta)-\varphi\right\|_{-1}^{2p} \leq C \delta^{p}\exp\left(p\eta\|w_0\|^2\right)\]
\end{lemma}
\begin{proof}
We first reverse the time of the processes by setting $\bar{w}_{s, r} = w_{s, T+2s-r}(w_0)$, and $\bar{U}_r = U^{s+T,\varphi}(T+2s-r)$. Then $\bar{U}_r$ solves the parabolic equation
\[\left\{
\begin{array}
{l}\partial_r \bar{U}_r= \nu \mathrm{\Delta} \bar{U}_r+B\left(\mathcal{K}\bar{w}_{s, r}, \bar{U}_r\right) - C\left(\mathcal{K}\bar{U}_r, \bar{w}_{s, r} \right), \quad s < r \leq s+T,  \\
\bar{U}_s=\varphi.
\end{array}\right.\]
It then follows from the variation of constant formula that,
\[\bar{U}_{s+\delta} = e^{\delta\nu\mathrm{\Delta}} \varphi +\int_s^{s+\delta} e^{\nu\mathrm{\Delta}\left(s+\delta -r\right)}\left[B\left(\mathcal{K}\bar{w}_{s, r}, \bar{U}_r\right) - C\left(\mathcal{K}\bar{U}_r, \bar{w}_{s, r} \right)\right]dr.\]
Since both  $\|B\left(\mathcal{K}\bar{w}_{s, r}, \bar{U}_r\right)\|$ and $\|C\left(\mathcal{K}\bar{U}_{r}, \bar{w}_{s, r} \right)||$ are bounded by $C\|\bar{w}_{s, r}\|_1\|\bar{U}_r\|_1$,  one has
\begin{align*}
\left\|\bar{U}_{s+\delta} -e^{\delta\nu\mathrm{\Delta}} \varphi\right\|\leq C\sup_{T+s-\delta\leq r\leq s+T}\|{w}_{s, r}\|_1\int_{s}^{s+\delta}\| \bar{U}_r\|_1dr.
\end{align*}
To estimate $\|\bar{U}_r\|_1$, setting $\zeta_r = \|\bar{U}_r\|^2 + \nu(r-s)\|\bar{U}_r\|_1^2$.  As in the proof of inequality \eqref{xi_t1},  one obtains from equation for $\bar{U}_r$ that
\begin{align}\label{Appendixspectral01}
\|\bar{U}_r \|_1\leq \frac{C}{\sqrt{r-s}} \|\varphi\|\exp\left(\int_s^{s+T}\eta\|w_{s, r}\|_1^2 dr\right),
\end{align}
 where $C$ is a constant depending on $\nu,\eta, T$.
Now it follows from bounds for the solution \eqref{enstrophy} that
\begin{align*}
\mathbf{E}\sup_{\|\varphi\|\leq 1}\left\|\bar{U}_{s+\delta}-e^{\delta\nu\mathrm{\Delta}} \varphi\right\|^{2p} \leq C\delta^p \exp\left(p\eta\|w_0\|^2\right),
\end{align*}
which is the first inequality of Lemma \eqref{back1}. The second inequality follows from the first one and the following fact
\[\left\|\bar{U}_{s+\delta}-\varphi\right\|_{-1}\leq \left\|\bar{U}_{s+\delta}-e^{\delta\nu\mathrm{\Delta}}\varphi\right\|_{-1}+\left\|e^{\delta\nu\mathrm{\Delta}}\varphi-\varphi\right\|_{-1} \leq \left\|\bar{U}_{s+\delta}-e^{\delta\nu\mathrm{\Delta}}\varphi\right\|+C\delta.\]
\end{proof}
The next lemma allows one to transfer the properties of $\varphi$ back from the terminal time.
\begin{lemma}\label{backphi}
Fix any orthogonal projection $\mathrm{\Pi}$ of $H$ onto a finite dimensional subspace of $H$ spanned by elements of $H_1$. There exists a constant $c\in (0,1)$ such that for every $r>0$ and every $\alpha>0$, the event
\begin{align*}
\mathrm{\Omega}_{\delta,\mathrm{\Pi}}:= \left\{\omega\in\mathrm{\Omega}: \varphi\in S_{\alpha}\Longrightarrow U^{s+T,\varphi}(T+s-\delta)\in S_{c\alpha} \quad\text{and}\quad \|\mathrm{\Pi} U^{s+T,\varphi}(T+s-\delta)\|\geq \frac{\alpha}{2}\|\varphi\|\right\}
\end{align*}
satisfies $\mathbf{P}\left(\mathrm{\Omega}_{\delta,\mathrm{\Pi}}^c\right)\leq C\exp\left(\eta\|w_0\|^2\right)\varepsilon^p $ for every $p\geq 1$. Note that here $C$ depends on $1/r$.
\end{lemma}
\begin{proof}
Since we have proved Lemma \ref{back1},  this is a reformulation of Lemma 6.15 in \cite{HM11}.
\end{proof}
Since the randomness spreads over the state space through the nonlinear term,  we define recursively the following sets $\{A_k\}_{k=1}^{\infty}$ formed by the symmetrized nonlinear term $\widetilde{B}(u,w)=-B(\mathcal{K}u,w)-B(\mathcal{K}w,u)$. Set $A_1=\{g_k : 1\leq k\leq d\}$, and $A_{k+1} = A_{k}\cup \{\widetilde{B}(h, g_{l}): h\in A_{k}, g_l\in A_1\}$. Also define $A_{\infty} =\overline{ \mathrm{span}({\cup_{k\geq 1} A_{k}})}$.  Note that each $A_{k}$ here, consisting of constant vector fields in $H$,  is a subset of the $k$-th Hörmander bracket defined in section 6 of \cite{HM11}. To each $A_n$ we associate a quadratic form $\mathcal{Q}_n$ by $\left\langle\varphi, \mathcal{Q}_{n} \varphi\right\rangle=\sum_{h \in \mathrm{A}_{n}}\langle\varphi, h\rangle^{2}$.  Just as in \cite{BM07,HM11,MP06}, it is typical to apply arguments that use local time regularity to replace the analysis of non-adapted processes.  Hence for $\theta\in(0, 1]$ we define the following  (semi-)norm for functions $g: I_{\delta} \rightarrow H$ by
\[\|g\|_{\theta,s}:=\sup _{r, t \in I_{\delta}} \frac{\|g(t)-g(r)\|_{s}}{|t-r|^{\theta}}, \quad \text{and}\quad \|g\|_{\infty,s}:=\sup_{t\in I_{\delta}}\|g\|_s.\]
Also for $g:I_{\delta} \rightarrow \mathbb{R}$, we use the following notation for the corresponding norms
\[\|g\|_{{\infty}} = \sup_{t\in I_{\delta}}|g(t)|,\quad \|g\|_{\theta}:=\sup _{r, t \in I_{\delta}} \frac{|g(t)-g(r)|}{|t-r|^{\theta}}.\]

\vskip0.05in

The following lemma is the key to prove theorem \ref{thmmatt}. The proof  requires a technical result which roughly states that two distinct monomials in a Wiener polynomial cannot cancel each other out \cite{BM07,HM11,MP06}.
\begin{lemma}\label{implysmall}
For every $p\geq 1$ and  integer $N>0$, there exist $0<\varepsilon_N<1, r_{N}>0, p_{N}>0$ and $q_{N}=q_N(p)>0$ such that, provided that $r \leq r_{N},$  the event
\[
\mathrm{\Omega}_{\varepsilon, N}:=\left\{\omega\in\mathrm{\Omega}: \left\langle\varphi, {M}_{s, s+T} \varphi\right\rangle \leq \varepsilon\|\varphi\|^{2} \Longrightarrow \sup _{h \in \mathrm{A}_{N}} \sup _{t \in I_{\delta}}\left|\left\langle U^{s+T,\varphi}(t) , h\right\rangle\right| \leq \varepsilon^{p_{N}}\|\varphi\|\right\}
\]
satisfies
\[\mathbf{P}\left(\mathrm{\Omega}^c_{\varepsilon,N}\right)\leq C_{q_N}\exp\left(\eta\|w_0\|^2\right)\varepsilon^p,\]
for $\varepsilon\in(0, \varepsilon_N] $ and $\eta\in (0, \eta_0]$.
\end{lemma}
\begin{proof}
The proof proceeds by induction on $N$. It suffices to show the result for $\varphi$ with unit norm $\|\varphi\| =1$.  We first prove that the result is true for $A_1$.
Assume that $\left\langle\varphi, {M}_{s, s+T} \varphi\right\rangle \leq \varepsilon$, then by representation  \eqref{backrep}, one has
\[\sup_{1\leq k\leq d} \int_{I_{\delta}}\left\langle g_{k},  U^{s+T, \varphi}(\tau) \right\rangle^{2} d \tau\leq\varepsilon.\]
Setting $R(t) =  \int_{s+\frac{T}{2}}^{t}\left\langle g_{k},  U^{s+T, \varphi}(\tau) \right\rangle d \tau$, Lemma 6.14 in \cite{HM11} implies that
\begin{align*}
\sup_{t\in I_{\delta}}\left\langle g_{k},  U^{s+T, \varphi}(t) \right\rangle = \|\partial_t R\|_{\infty}&\leq 4\|R\|_{\infty}\max\left\{\frac{1}{|I_{\delta}|}, \|R\|_{\infty}^{-\frac12} \|{\partial_t R}\|_1^{\frac12}\right\}\\
&\leq C_T\max\left\{\varepsilon^{\frac12},\varepsilon^{\frac14}\|{\partial_t R}\|_1^{\frac12}\right\},
\end{align*}
where $C_T= 4\max\left\{1, \frac{2}{\sqrt{T}}, \left(\frac{T}{2}\right)^{1/4}\right\}$.
It follows from Lemma E.1 of \cite{MP06} that
\[\|U^{s+T,\varphi}\|_{1,0}\leq C\left(1+\|w\|_{\infty,1}^2 + \|U^{s+T,\varphi}\|_{\infty,1}^2\right).\]
From bounds \eqref{Appendixspectral01}, Lemma \ref{bounds} and Lemma \ref{hb}, we have that for any $p'\geq 1$,
\begin{align*}
\mathbf{E}\|{\partial_t R}\|_1^{p'}\leq \mathbf{E}\|g_k\|^{p'}\|U^{s+T,\phi}\|_{1,0}^{p'}\leq C \mathbf{E}\|U^{s+T,\phi}\|_{1,0}^{p'}\leq C\exp\left({\eta\|w_0\|^2}\right)\delta^{-p'}.
\end{align*}
Therefore by the Markov inequality, one has for $\tilde{p}>0, \alpha>0,$
\[\mathbf{P}\left(\|\partial_t R\|_1>\alpha\varepsilon^{-\tilde{p}}\right) = \mathbf{P}\left(\|\partial_t R\|_1^{p'}>\alpha^{p'}\varepsilon^{-\tilde{p}p'}\right)\leq \alpha^{-p'} \varepsilon^{\tilde{p}p'}\mathbf{E}\|{\partial_t R}\|_1^{p'}\leq C(\alpha,p')\exp\left({\eta\|w_0\|^2}\right)\varepsilon^{\tilde{p}p'-rp'}. \]
Then on a set $\widetilde{\mathrm{\Omega}}_{\varepsilon,1}\subset \mathrm{\Omega}$, such that $\mathbf{P}(\widetilde{\mathrm{\Omega}}_{\varepsilon,1}^c)\leq  C(\alpha,p')\exp\left({\eta\|w_0\|^2}\right)\varepsilon^{\tilde{p}p'-rp'}$, we have
\[\|\partial_t R\|_1\leq \alpha \varepsilon^{-\tilde{p}}.\]
Now choose $\alpha= C_T^{-2}$, $\tilde{p}=\frac14$, and $r_1 = \frac{1}{2}\tilde{p}$. Then  provided $r<r_1$, for any $p\geq 1$, by choosing $p' = \frac{p}{\tilde{p}-r}$,  one has on $\widetilde{\mathrm{\Omega}}_{\varepsilon,1}$,
\[\sup_{t\in I_{\delta}}\left\langle g_{k},  U^{s+T, \varphi}(t) \right\rangle \leq C_T\max\left(\varepsilon^{\frac12}, \alpha^{\frac12}\varepsilon^{\frac18}\right)\leq \varepsilon^{\frac18}\]
and $\mathbf{P}(\widetilde{\mathrm{\Omega}}_{\varepsilon,1}^c)\leq  C\exp\left({\eta\|w_0\|^2}\right)\varepsilon^{p}$ for $\varepsilon\in(0, \varepsilon_1]$, where $\varepsilon_1=C_T^{-8/3}$.
Observe that the event set $\widetilde{\mathrm{\Omega}}_{\varepsilon,1}$ does not depend on the choice of $g_k\in A_1$ and it is contained in $\mathrm{\Omega}_{\varepsilon, 1}$. So
\[\mathbf{P}({\mathrm{\Omega}}_{\varepsilon,1}^c)\leq \mathbf{P}(\widetilde{\mathrm{\Omega}}_{\varepsilon,1}^c)\leq  C\exp\left({\eta\|w_0\|^2}\right)\varepsilon^{p}.\]
 Hence the proof for the base case is complete with $p_1=r_1=1/8, q_1 = p/r_1$ and $\varepsilon_1=C_T^{-8/3}$.
 \end{proof}
The inductive step is accomplished through the following lemma.
\begin{lemma}
For $N\geq 2$, fix $h\in A_{N-1}$, and suppose that $q := p_{N-1}$  has been given. Then for $p\geq 1$, provided $r<r_N$, the event
\[\widetilde{\mathrm{\Omega}}_{N, \varepsilon}: = \left\{\sup_{t\in I_{\delta}}\left|\langle U^{s+T,\varphi}(t), h\rangle\right|\leq \varepsilon^{q}\Longrightarrow \sup_{1\leq k\leq d}\sup_{t\in I_{\delta}}\left|\langle U^{s+T,\varphi}(t),\widetilde{B}(h, g_k)\rangle\right|\leq \varepsilon^{p_N}\right\}\]
satisfies $\mathbf{P}(\widetilde{\mathrm{\Omega}}_{\varepsilon,N}^c) \leq C_{q_N}\exp\left({\eta\|w_0\|^2}\right)\varepsilon^{p}$ for $\varepsilon\in (0, \varepsilon_N)$. Here $p_N = q/24$, $r_N = q/12$, $q_N = 12p/q$ and $\varepsilon_N= C_T^{-8/(7q)}$ with $C_T= 4\max\left\{4/T,1\right\}$.
\end{lemma}
\begin{proof}
Let $R(t)=\partial_t \langle U^{s+T,\varphi}(t), h\rangle$. Then Lemma 6.14 in \cite{HM11} (with $\alpha=1/3$) implies that
\begin{align*}
\|R\|_{\infty}\leq 4\max\left\{\frac{1}{|I_{\delta}|}\varepsilon^{q}, \varepsilon^{\frac{q}{4}} \|R\|_{1/3}^{3/4}\right\}\leq C_T\max\left\{\varepsilon^{q},\varepsilon^{\frac{q}{4}}\|R\|_{1/3}^{3/4}\right\}.
\end{align*}
Next we show that $\|R\|_{1/3}$ has a bounded expectation. As in \cite{HM11}, we consider the process $v_{s, t}: = w_{s, t}- \sum_{k=1}^dg_kW_k(t)$ since it has more time regularity. Note \[R(t) = \langle -\mathrm{\Delta} h+\widetilde{B}(v_{s, t}, h)+ \sum_{k=1}^d \widetilde{B}(g_k , h)W_{k}(t), U^{s+T,\varphi}(t)\rangle.\]
And recall that we assumed elements of $\{g_k\}_{k=1}^d$ are smooth, so each $h$ has bounded $H_s$ norm for any $s$. Observe that
\begin{align}\notag
&\|\langle \widetilde{B}(v_{s, t},h), U^{s+T,\varphi}(t)\rangle\|_{1/3}\leq C\|U^{s+T,\varphi}\|_{1,0}\left(\|v_{s, t}\|_{\infty,1}+\|v_{s, t}\|_{1/3,1}\right)\\\notag
&\leq C\|U^{s+T,\varphi}\|_{1,0} \left(\|w_{s, t}\|_{\infty,1}+\sum_{k=1}^d\|g_{k}\|_1\|W_{k}\|_{\infty}+\|\partial_t v_{s, t}\|_{\infty,1}\right)\\\label{Appendixspectral02}
&\leq C\|U^{s+T,\varphi}\|_{1,0} \left(1+\|w_{s, t}\|_{\infty,3}^2+\sum_{k=1}^d\|W_{k}\|_{\infty}^2\right),
\end{align}
for constant $C$ depends on $\|h\|_2, \mathcal{B}_1$ and $\|f\|_{\infty,1}$, where we used the fact  that
\begin{align*}
\|\partial_t v_{s, t}\|_{\infty,1}&=\|\mathrm{\Delta} w_{s, t} - B(\mathcal{K}w_{s, t},w_{s, t})+f(t)\|_{\infty,1}\\
&\leq \|w_{s, t}\|_{\infty,3} + C\|w_{s, t}\|_{\infty,3}^2 + \|f\|_{\infty,1}\leq C\left(1+\|w_{s, t}\|_{\infty,3}^2\right).
\end{align*}
For other terms in the expression of $R(t)$, one has
\begin{align}\label{Appendixspectral03}
\|\langle -\mathrm{\Delta} h, U^{s+T,\varphi}\rangle \|_{1/3}\leq C\|U^{s+T,\varphi}\|_{1,0},
\end{align}
and
\begin{align}\label{Appendixspectral04}
\|\langle  \sum_{k=1}^d \widetilde{B}(g_k , h)W_{k}(t), U^{s+T,\varphi}(t)\rangle\|_{1/3}\leq C\sum_{k=1}^d\|W_k\|_{1/3}\|U^{s+T,\varphi}\|_{1,0}.
\end{align}
Therefore one obtains
\begin{align*}
\|R\|_{1/3}\leq C \|U^{s+T,\varphi}\|_{1,0}\left(1+\|w_{s, t}\|_{\infty,3}^2+\sum_{k=1}^d\|W_{k}\|_{\infty}^2+\sum_{k=1}^d\|W_{k}\|_{1/3}^2\right).
\end{align*}
From the proof of Lemma 7.12 in \cite{MP06}, we know that $\mathbf{E}\|W_k\|_{\alpha}^{\gamma}<C(T,\gamma)$ for all $\gamma\geq 1$ and $\alpha\in [0,\frac12)$. This fact, together with the bounds for the solution from Lemma \ref{bounds} and Lemma \ref{hb}, implies that for any $p'\geq 1$, $\alpha>0$ and $\tilde{p}>0$
\[\mathbf{P}\left(\|R\|_{1/3}>\alpha\varepsilon^{-\tilde{p}}\right)\leq C(p',\alpha)\exp\left(\eta\|w_0\|^2\right)\delta^{-p'}\varepsilon^{\tilde{p}p'}\leq C(p',\alpha)\exp\left(\eta\|w_0\|^2\right) \varepsilon^{\tilde{p}p'-p'r}.\]
Therefore on a set $\widetilde{\mathrm{\Omega}}_{\varepsilon,N,1}\subset \mathrm{\Omega}$, such that $\mathbf{P}(\widetilde{\mathrm{\Omega}}_{\varepsilon,N,1}^c)\leq  C\exp\left({\eta\|w_0\|^2}\right)\varepsilon^{\tilde{p}p'-p'r}$, we have
\[\| R\|_{1/3}\leq \alpha\varepsilon^{-\tilde{p}}.\]
Now choose $p_N, r_N$ and $\varepsilon_N$ as stated in the Lemma and let $\tilde{p} = q/6$, $\alpha=C_T^{-4/3}$. Then  on the set $\widetilde{\mathrm{\Omega}}_{\varepsilon,N,1}$, we have
\begin{align*}
\|R\|_{\infty}\leq  C_T\max\left\{\varepsilon^{q}, \alpha^{\frac34}\varepsilon^{\frac{q}{8}}\right\} \leq \varepsilon^{\frac{q}{8}}.
\end{align*}
for all $\varepsilon\in(0,\varepsilon_{N}]$. Note that the $\alpha$ and $\varepsilon_N$ are determined when taking the maximum.  Observing that for any $p\geq 1$, provided $r<r_N$, we can take $p' = \frac{p}{\tilde{p}-r}$ , so that
\[\mathbf{P}(\widetilde{\mathrm{\Omega}}_{\varepsilon,N,1}^c)\leq  C_{q_N}\exp\left({\eta\|w_0\|^2}\right)\varepsilon^{p}.\]
Denote $R_0=\langle-\mathrm{\Delta} h+\widetilde{B}(v_{s, t},h), U^{s+T,\varphi}(t)\rangle$ and $R_k =\langle \widetilde{B}(g_k, h), U^{s+T,\varphi}(t)\rangle $ for $k=1,2,\cdots,d$, which are the coefficients of the Wiener polynomial $R(t)$.
Then the technical Theorem 7.1 from \cite{HM11} implies that on a set $\widetilde{\mathrm{\Omega}}_{\varepsilon, N, 2}\subset\mathrm{\Omega}$, one has
\begin{align*}
\|R\|_{\infty}\leq \varepsilon^{q/8} \Rightarrow
\left\{\begin{array}{ll}
\text { either } \displaystyle\sup _{0\leq k\leq d}\left\|R_k\right\|_{\infty} \leq \varepsilon^{q/24}, \\
\text { or } \displaystyle\sup _{0\leq k\leq d} \|R_k\|_{1}\geq \varepsilon^{-q/72},
\end{array}\right.
\end{align*}
and $\mathbf{P}(\widetilde{\mathrm{\Omega}}_{\varepsilon,N,2}^c)\leq  C\varepsilon^{p},$  where $C$ depends only on $p$ and the events $\widetilde{\mathrm{\Omega}}_{\varepsilon, N, 2}$ depends on the processes $A_k$ only through the value of the highest degree of the Wiener polynomial, which is $1$ here.
The Markov inequality and the estimates \eqref{Appendixspectral02}-\eqref{Appendixspectral04} implies that there is an event $\widetilde{\mathrm{\Omega}}_{\varepsilon, N, 3}$, on which $ \|R_k\|_{1}< \varepsilon^{-q/72}$ for each $k$,  and
\begin{align*}
\mathbf{P}(\widetilde{\mathrm{\Omega}}_{\varepsilon,N,3}^c) \leq C\mathbf{P}\left(\|R_k\|_1^{72p/q}\geq \varepsilon^{-p}\right)  \leq  C\exp\left({\eta\|w_0\|^2}\right)\varepsilon^{p}.
\end{align*}
Now observe that $\bigcap_{i=1}^3\widetilde{\mathrm{\Omega}}_{\varepsilon,N,i}\subset\widetilde{\mathrm{\Omega}}_{\varepsilon,N}$, hence
\[\mathbf{P}(\widetilde{\mathrm{\Omega}}_{\varepsilon,N}^c)\leq \sum_{i=1}^3\mathbf{P}(\widetilde{\mathrm{\Omega}}_{\varepsilon,N,i}^c)\leq C\exp\left({\eta\|w_0\|^2}\right)\varepsilon^{p}.\]
This completes the proof of the induction step.
\end{proof}
Now we give a proof of Theorem \ref{thmmatt} by combining the above lemmas.
Since $A_{\infty} = H$, by Lemma 8.3 in \cite{HM11}, for any fixed finite dimensional projection $\mathrm{\Pi}$, there exists $N>0$ ( $N$ depends on the projection $\mathrm{\Pi}$, so that $p_N, r_N, \varepsilon_N$ depends on $\mathrm{\Pi}$) such that for each $\alpha>0$, there exists a constant $\mathrm{\Lambda}_{\alpha}>0$, such that for every $n\geq N$,
\[\inf_{\varphi\in S_{\alpha}}\frac{\left|\left\langle\varphi, \mathcal{Q}_{n} \varphi\right\rangle\right|}{\|\mathrm{\Pi}\varphi\|^2}\geq \mathrm{\Lambda}_{\alpha}.\]
On the other hand, it follows from Lemma \ref{backphi} and Lemma \ref{implysmall} that  there exist constants $p_N, r_N, \varepsilon_N, c >0$  such that  for every $\alpha>0$,  on the set $\mathrm{\Omega}_{\varepsilon, N}\bigcap\mathrm{\Omega}_{\delta,\mathrm{\Pi}}$, the condition
 \[\varphi\in S_{\alpha} \quad\text{and} \quad \left\langle\varphi, {M}_{s, s+T} \varphi\right\rangle \leq \varepsilon\|\varphi\|^{2} \]
 implies that
\[U^{s+T,\varphi}(T+s-\delta)\in S_{c\alpha}, \quad \|\mathrm{\Pi} U^{s+T,\varphi}(T+s-\delta)\|\geq \frac{\alpha}{2}\|\varphi\|,\]
\[\quad \text{and}\quad \sup _{h \in \mathrm{A}_{N}} \sup _{t \in I_{\delta}}\left|\left\langle U^{s+T,\varphi}(t) , h\right\rangle\right| \leq \varepsilon^{p_N}\|\varphi\|, \]
and $\mathbf{P}\left(\mathrm{\Omega}_{\varepsilon, N}^c\bigcup\mathrm{\Omega}_{\delta,\mathrm{\Pi}}^c\right)\leq C\exp\left(\eta\eta\|w_0\|^2\right)\varepsilon^p $, for any $\varepsilon\in(0,\varepsilon_N)$ and $p\geq 1$.
Then it follows that on the set $\mathrm{\Omega}_{\varepsilon, N}\bigcap\mathrm{\Omega}_{\delta,\mathrm{\Pi}}$, one  has
\[\frac{\alpha}{2}\|\varphi\|\leq \|\mathrm{\Pi} U^{s+T,\varphi}(T+s-\delta)\|\leq C\mathrm{\Lambda}_{c\alpha}^{-1/2}\sup_{h\in A_N}\langle U^{s+T,\varphi}(T+s-\delta), h\rangle\leq C\mathrm{\Lambda}_{c\alpha}^{-1/2}\varepsilon^{p_N}\|\varphi\|. \]
This in turn shows that on $\mathrm{\Omega}_{\varepsilon, N}\bigcap\mathrm{\Omega}_{\delta,\mathrm{\Pi}}$, $\left\langle\varphi, {M}_{s, s+T} \varphi\right\rangle \leq \varepsilon\|\varphi\|^{2}$ and $\varphi\in S_{\alpha}$ implies that
\[\frac{\alpha}{2}< C \varepsilon^{p_N},\]
which is not true for $\varepsilon\leq \varepsilon_0:=\min\left\{\varepsilon_N, \left(\frac{\alpha}{2C}\right)^{1/p_N}\right\}$. \\
Hence $\mathbf{P}\left(\langle M_{s, s+T} \varphi, \varphi\rangle<\varepsilon\|\varphi\|^{2}\right) \leq C \varepsilon^{p} \exp \left(\eta\left\|w_0\right\|^{2}\right)$ for $\varphi\in S_{\alpha}$.
\end{proof}

The following lemma gives a quantitative control of the error between the Malliavin matrix and its regularization.
\begin{lemma}\label{lowmodes}
Fix $\xi \in H$ and set

\[\zeta=\beta\left(\beta+M_{0}\right)^{-1} {J}_{0} \xi.\]
Then for any constants $p\geq 1$, $\gamma, \eta>0$ and every finite dimensional orthogonal projector $\pi_{\ell},$ there exists a small $\beta_0: = \beta_0(p, \gamma, \eta)>0$ such that for every $\beta\in(0, \beta_0]$,
\[\mathbf{E}\left\|\pi_{\ell} \zeta\right\|^{p} \leq \gamma \exp\left(\eta\left\|w_0\right\|^{2}\right)\|\xi\|^{p}.\]
\end{lemma}
\begin{proof}
For $\alpha>0$, define $A_{\alpha}: = \{\omega\in\Omega: \left\|\pi_{\ell} \zeta\right\|(\omega)>\alpha\|\zeta\|(\omega)\}$. Let   $\zeta_{\alpha}(\omega)=\zeta(\omega) \mathbb{I}_{A_{\alpha}}(\omega)$ and $\bar{\zeta}_{\alpha}(\omega)=\zeta(\omega)-\zeta_{\alpha}(\omega) = \zeta(\omega) \mathbb{I}_{A_{\alpha}^c}(\omega)$, where $\mathbb{I}_{A}$ is the characteristic function of the set $A$.
Since $\|\beta(\beta+M_0)^{-1}\|\leq 1$, it follows from estimates \eqref{enstrophy} and \eqref{Jacobian} in  Lemma \ref{bounds} that
\begin{align}\label{lemma{lowmodes}-1}
\mathbf{E}\left\|\pi_{\ell} \bar{\zeta}_{\alpha}\right\|^{p} \leq \alpha^{p} \mathbf{E}\|\zeta\|^{p} \leq \alpha^{p} \mathbf{E}\left\|{J}_{0} \xi\right\|^{p} \leq \frac{\gamma}{2}\exp\left(\eta\left\|w_0\right\|^{2}\right)\|\xi\|^{p}
\end{align}
by choosing $\alpha$ sufficiently small. Fix such an $\alpha$. We also have
\begin{align*}
\left\langle\zeta_{\alpha}, M_{0} \zeta_{\alpha}\right\rangle  \leq\left\langle\zeta, M_{0} \zeta\right\rangle \leq\left\langle\zeta,\left(M_{0}+\beta\right) \zeta\right\rangle  =\left\langle\beta\left(M_{0}+\beta\right)^{-1} {J}_{0} \xi,\beta{J}_{0} \xi\right\rangle \leq \beta\left\|{J}_{0} \xi\right\|^{2}.
 \end{align*}
 By Theorem \ref{thmmatt}, we know that for every $p\geq 1$ and $\alpha>0$,  there exists a constant $C$  and $\varepsilon_0$ such that
 \[\mathbf{P}\left(\left\langle M_{0} \zeta_{\alpha}, \zeta_{\alpha}\right\rangle<\varepsilon\left\|\zeta_{\alpha}\right\|^{2}\right) \leq C \varepsilon^{p} \exp \left(\eta\left\|w_0\right\|^{2}\right)\]
 holds for every $w_0\in H$ and every $\varepsilon \in (0,\varepsilon_0)$. Therefore
 \[\mathbf{P}\left(\frac{\left\|\zeta_{\alpha}\right\|^{2}}{\left\|{J}_{0} \xi\right\|^{2}}>\frac{\beta}{\varepsilon}\right) \leq \mathbf{P}\left(\left\langle M_{0} \zeta_{\alpha}, \zeta_{\alpha}\right\rangle<\varepsilon\left\|\zeta_{\alpha}\right\|^{2}\right) \leq C \varepsilon^{p} \exp \left(\eta\left\|w_0\right\|^{2}\right).\]
 Choosing $\beta=\varepsilon^2$, and noting $\frac{\left\|\zeta_{\alpha}\right\|}{\left\|{J}_{0} \xi\right\|}\leq 1$, we find that
 \begin{align}\label{lemma{lowmodes}-2}
 \mathbf{E}\left(\frac{\left\|\zeta_{\alpha}\right\|^{2p}}{\left\|{J}_{0} \xi\right\|^{2p}}\right) \leq \mathbf{P}\left(\frac{\left\|\zeta_{\alpha}\right\|^{2p}}{\left\|{J}_{0} \xi\right\|^{2p}}>\frac{\beta^p}{\varepsilon^p}\right)+\frac{\beta^p}{\varepsilon^p}\leq C \varepsilon^{p} \exp \left(\eta\left\|w_0\right\|^{2}\right).
 \end{align}
Noting that
\begin{align*}
\mathbf{E}\left\|\pi_{\ell} \zeta_{\alpha}\right\|^{p}\leq\mathbf{E}\left\|\zeta_{\alpha}\right\|^{p} \leq \sqrt{\mathbf{E}\left(\left\|\zeta_{\alpha}\right\|^{2 p}\left\|{J}_{0} \xi\right\|^{-2 p}\right) \mathbf{E}\left\|{J}_{0} \xi\right\|^{2 p}}.
\end{align*}
Then combining this \eqref{Jacobian} from Lemma \ref{bounds} with \eqref{lemma{lowmodes}-2}, it follows that
\begin{align}\label{lemma{lowmodes}-3}
\mathbf{E}\left\|\pi_{\ell} \zeta_{\alpha}\right\|^{p} \leq \frac{\gamma}{2} e^{\eta\left\|w_0\right\|^{2}}\|\xi\|^{p}
\end{align}
 by choosing $\varepsilon$ sufficiently small, which in turn gives the desired $\beta_0$. The lemma then follows from \eqref{lemma{lowmodes}-1} and \eqref{lemma{lowmodes}-3} by observing that  $\mathbf{E}\left\|\pi_{\ell} \zeta\right\|^{p}=\mathbf{E}\left\|\pi_{\ell} \zeta_{\alpha}\right\|^{p}+\mathbf{E}\left\|\pi_{\ell} \bar{\zeta}_{\alpha}\right\|^{p}$.
\end{proof}
\begin{remark}
By Markov property in its generalized form (see for example Theorem 9.18 in \cite{DZ14}), it follows from Lemma \ref{lowmodes}  that  for  each positive integer $n$ and  $\zeta=\beta\left(\beta+M_{n}\right)^{-1} {J}_{n} \xi$, one has
\[\mathbf{E}\left(\left\|\pi_{\ell} \zeta\right\|^{p}\,\middle\vert\, \mathcal{F}_{s+n}\right) \leq \gamma e^{\eta\left\|w_{s, s+n}\right\|^{2}}\|\xi\|^{p}.\]
\end{remark}

\begin{lemma}\label{error}
For any constants $\gamma, \eta>0$ and $p \geq 1,$ there exists a constant $\beta_{0}: =\beta_0(p, \gamma, \eta)>0$ such that whenever $0<\beta \leq \beta_{0}$, we have
\[\mathbf{E}\left(\left\|\mathfrak{R}_{s+2(n+1)}\right\|^{p} \,\middle\vert\, \mathcal{F}_{s+2n}\right) \leq \gamma e^{\eta\left\|w_{s, s+2n}\right\|^{2}}\left\|\mathfrak{R}_{s+2n}\right\|^{p},\,\, \mathbf{P}\text{-}\mathrm{a.s}..\]
\end{lemma}
\begin{proof}
The proof is mainly based on Lemma \ref{lowmodes} and Lemma \ref{highmodes}.  Let $\zeta = \beta \widetilde{M}_{2n}^{-1}{J}_{2n}\mathfrak{R}_{s+2n}$. Observe that
 \begin{align}\label{lemma{error}-0}
 \mathfrak{R}_{s+2(n+1)}={J}_{2n+1}\mathfrak{R}_{s+2n+1} = {J}_{2n+1}\zeta.
 \end{align}
Also note that $\|\beta \widetilde{M}_{2n}^{-1}\|\leq 1$ and $\mathfrak{R}_{s+2n}$ is $\mathcal{F}_{s+2n}$ measurable. Hence by the bound \eqref{Jacobian} in Lemma \ref{bounds}and the Markov property, one has
\begin{align}\label{lemma{error}-1}
\mathbf{E}\left(\left\|\zeta\right\|^{p} \,\middle\vert\, \mathcal{F}_{s+2n}\right)\leq \left\|\mathfrak{R}_{s+2n}\right\|^p\mathbf{E}\left(\left\|{J}_{2n}\right\|^p\,\middle\vert\, \mathcal{F}_{s+2n}\right) \leq C e^{\frac{\eta}{2}\left\|w_{s, s+2n}\right\|^{2}}\left\|\mathfrak{R}_{ s+2n}\right\|^{p}.
\end{align}
Applying Lemma \ref{highmodes}, the Hölder's  inequality and estimate \eqref{lemma{error}-1}, it follows that there exists a projection $\pi_l$ on a finite dimensional subspace of $H$ such that
\begin{align}\notag
\mathbf{E}\left(\left\|{J}_{2n+1}\left(1-\pi_{\ell}\right) \zeta\right\|^{p} \,\middle\vert\, \mathcal{F}_{s+2n}\right)
&\leq \sqrt{\mathbf{E}\left(\left\|{J}_{2n+1}\left(1-\pi_{\ell}\right) \right\|^{2p} \,\middle\vert\, \mathcal{F}_{s+2n}\right)\mathbf{E}\left(\left\|\zeta\right\|^{2p} \,\middle\vert\, \mathcal{F}_{s+2n}\right)}\\\label{lemma{error}-2}
&\leq \tilde{\gamma} e^{\frac{\eta}{2}\left\|w_{s+2n}\right\|^{2}} \left(C e^{\frac{\eta}{2}\left\|w_{s, s+2n}\right\|^{2}}\left\|\mathfrak{R}_{s+2n}\right\|^{p}\right)
 \leq \gamma e^{\eta\left\|w_{s, s+2n}\right\|^{2}}\left\|\mathfrak{R}_{s+2n}\right\|^{p}.
\end{align}
From Lemma \ref{lowmodes} and the Markov property, it follows that for an arbitrarily small $\tilde{\gamma}$, one can choose $\beta$ sufficiently small such that
\[\mathbf{E}\left(\left\|\pi_{\ell} \zeta\right\|^{p} \,\middle\vert\, \mathcal{F}_{s+2n}\right) \leq \tilde{\gamma} e^{\frac{\eta}{2}\left\|w_{s, s+2n}\right\|^{2}}\|\mathfrak{R}_{s+2n}\|^{p}.\]
Again applying Hölder's  inequality and the bound \eqref{Jacobian} on the Jacobian ${J}_{2n+1}$,  one can deduce that for any $\gamma>0$,
\begin{align}\label{lemma{error}-3}
\mathbf{E}\left(\left\|{J}_{2n+1} \pi_{\ell} \zeta\right\|^{p} \,\middle\vert\, \mathcal{F}_{s+2n}\right) \leq \gamma e^{\eta\left\|w_{s+2n}\right\|^{2}}\left\|\mathfrak{R}_{s+2n}\right\|^{p}.
\end{align}
by choosing $\beta$ sufficiently small. The proof is then complete by combining \eqref{lemma{error}-0}, \eqref{lemma{error}-2} and \eqref{lemma{error}-3}.
\end{proof}
The following result gives a desired estimate on the error  between the variations on the initial condition and that on the Wiener path.
\begin{proposition}\label{errorbounds}
There exists an $\eta_0>0$ such that for any  $\eta\in(0,\eta_0],$  and $a>0$ there is constant $C = C(\eta, a)$, and $p\in(0,1)$ such that
\[\left(\mathbf{E}\|\mathfrak{R}_{s+t}\|^2\right)^{1/2} \leq C\exp(p\eta\|w_0\|^2)e^{-at},\]
for  all $s\in\mathbb{R}$ and $t\geq 0$.
\end{proposition}
\begin{proof}
The proof is based on Lemma \ref{error}, Lemma \ref{bounds} and an iteration procedure.
Let  $C_{n} = \frac{\left\|\mathfrak{R}_{s+2n+2}\right\|^{10}}{\left\|\mathfrak{R}_{s+2n}\right\|^{10}}$, where we set $C_n = 0$ if $\mathfrak{R}_{s+2n} = 0$.  Note that $\left\|\mathfrak{R}_{s+2N}\right\|^{10}=\prod_{n=0}^{N-1} C_{n}$ since $\left\|\mathfrak{R}_s\right\| = \left\|\xi\right\| = 1$.  Also one observes that $\left\|\beta\widetilde{M}_{2n}^{-1}\right\|\leq 1$, and $\left\|\mathfrak{R}_{s+2n+2}\right\|\leq\left\|{J}_{2n+1} \beta \widetilde{M}_{2n}^{-1} {J}_{2n}\right\|\left\| \mathfrak{R}_{s+2n}\right\|$ in view of \eqref{lemma{error}-0}. So by the bound \eqref{Jacobian} on the Jacobian  from Lemma \ref{bounds}, it follows that for every $\eta>0$, there exists a  constant $C: =C(\eta, \nu)$ such that
\begin{align}\label{cn}
C_{n} \leq\left\|{J}_{2n+1} \beta \widetilde{M}_{2n}^{-1} {J}_{2n}\right\|^{10} \leq\left\|{J}_{2n+1}\right\|^{10}\left\|{J}_{2n}\right\|^{10} \leq \exp \left(\eta \int_{s+2n}^{s+2n+2}\left\|w_{r}\right\|_{1}^{2} d r+C\right),\,\, \mathbf{P}\text{-}\mathrm{a.s}..
\end{align}
Now define for $\eta, R>0$,
\[C_{n, R}=\left\{
\begin{array}{ll}
e^{-\eta R} & \text { if }\left\|w_{s, s+2n}\right\|^{2} \geq 2 R, \\ e^{\eta R} C_{n} & \text { otherwise.}
\end{array}\right.\]
Note that both  $C_n$ and $C_{n,R}$ are $\mathcal{F}_{s+2n+2}$ measurable. We denote
\[\mathrm{\Omega}_{R}: = \left\{\omega\in \mathrm{\Omega}:\left\|w_{s, s+2n}\right\|^{2} \geq 2 R \right\}\]
and $\overline{\mathrm{\Omega}}_{R}$ its complement. It follows from Lemma \ref{error} that for every $R>\eta^{-1}$, there is $\beta>0$ making $\gamma$ sufficiently small such that
\begin{align}\label{cnr}
\mathbf{E}\left(C_{n,R}^2\,\middle\vert\,\mathcal{F}_{s+2n}\right) &= \mathbf{E}\left(\mathbb{I}_{\mathrm{\Omega}_R}e^{-2\eta R} +\mathbb{I}_{\bar{\mathrm{\Omega}}_R}e^{2\eta R} C_{n}^2 \,\middle\vert\,\mathcal{F}_{s+2n}\right) =\mathbb{I}_{\mathrm{\Omega}_R}e^{-2\eta R} + \mathbb{I}_{\bar{\mathrm{\Omega}}_R}e^{2\eta R}\mathbf{E}\left(C_{n}^2 \,\middle\vert\,\mathcal{F}_{s+2n}\right)\nonumber\\
&\leq \mathbb{I}_{\mathrm{\Omega}_R}e^{-2\eta R} + \mathbb{I}_{\bar{\mathrm{\Omega}}_R}\gamma e^{4\eta R}\leq\frac{1}{2}, \,\, \mathbf{P}\text{-}\mathrm{a.s}..
\end{align}
It now follows from the definition of $C_n$ and inequality \eqref{cn} that
\begin{align*}
C_{n} \leq C_{n, R} \exp \left(\eta \int_{s+2n}^{s+2n+2}\left\|w_{s, r}\right\|_{1}^{2} d r+\eta\left\|w_{s, s+2n}\right\|^{2}+C-\eta R\right),\,\, \mathbf{P}\text{-}\mathrm{a.s}..
\end{align*}
Therefore by the Cauchy-Schwarz inequality,
\begin{align}\notag
 \prod_{n=0}^{N-1} C_{n} \leq & \prod_{n=0}^{N-1} C_{n, R}^{2}+\prod_{n=0}^{N-1} \exp \left(2 \eta \int_{s+2n}^{s+2n+2}\left\|w_{s, r}\right\|_{1}^{2} d r+2 \eta\left\|w_{s, s+2n}\right\|^{2}+2 C-2 \eta R\right) \\\notag
 \leq & \prod_{n=0}^{N-1} C_{n, R}^{2}+\exp \left(4 \eta \sum_{n=0}^{N-1}\left\|w_{s, s+2n}\right\|^{2}+2 N\left(C-\eta R\right)\right) \\\label{productofCn}
 &+\exp \left(4 \eta \int_{s}^{s+2N}\left\|w_{s, r}\right\|_{1}^{2} d r+2 N\left(C-\eta R\right)\right).
\end{align}
Now from inequality \eqref{cnr} one has
\[\mathbf{E}\left(\prod_{n=0}^{N-1} C_{n, R}^{2}\,\middle\vert\, \mathcal{F}_{s+2(N-1)}\right)\leq\frac12 \prod_{n=0}^{N-2} C_{n, R}^{2},\,\, \mathbf{P}\text{-}\mathrm{a.s}..\]
Taking conditional expectation repeatedly, one obtains
\begin{align}\label{prod}
\mathbf{E}\left(\prod_{n=0}^{N-1} C_{n, R}^{2}\right)\leq \frac{1}{2^N}.
\end{align}
Fix $\eta>0$ such that $\eta\leq\min\{\frac14 \eta_0\nu, \frac{1}{4}\eta_1\}$. Then the bounds from inequality \eqref{enstrophy} and \eqref{boundsum} implies that
\begin{align}\label{error-02}
\begin{array}{c}
\mathbf{E}\exp \left(4 \eta \int_{s}^{s+2N}\left\|w_{s, r}\right\|_{1}^{2} d r+2 N\left(C-\eta R\right)\right)\leq C\exp\left(4\eta\nu^{-1}\left\|w_0\right\|^2+2N(C-\eta R)\right),\\
\mathbf{E}\exp \left(4 \eta \sum_{n=0}^{N-1}\left\|w_{s, s+2n}\right\|^{2}+2 N\left(C-\eta R\right)\right)\leq \exp\left(4a\eta\left\|w_0\right\|^{2}+N\left(\gamma+2C-2\eta R\right)\right).
\end{array}
\end{align}
Choosing $R$ sufficiently large such that these two terms satisfy the desired bounds. Then choose $\beta$ sufficiently small so that the estimate \eqref{prod} holds and hence by \eqref{productofCn}-\eqref{error-02} we have
\begin{align}\label{error1}
\mathbf{E}\left\|\mathfrak{R}_{s+2N}\right\|^{10} \leq \frac{C \exp \left(\eta\left\|w_0\right\|^{2}\right)}{2^{N}}
\end{align}
 for every $N\in \mathbb{N}$.

\vskip0.05in

 Note that for $t\in [2n, 2n+1)$, one has
\begin{align*}
\mathfrak{R}_{s+t}& = J_{s+2n, s+t}J_{s, s+2n}\xi - A_{s, s+t}v_{s, s+t}\\
& = J_{s+2n, s+t}\mathfrak{R}_{s+2n} + J_{s+2n, s+t}A_{s, s+2n}v_{s, s+2n}- A_{s, s+t}v_{s,s+ t}\\
& =  J_{s+2n, s+t}\mathfrak{R}_{s+2n} + A_{s, s+ t}v_{s, s+t} - A_{s+2n, s+t}v_{s+2n, s+t}- A_{s, s+t}v_{s, s+t} \\
&=   J_{s+2n, s+t}\mathfrak{R}_{s+2n}- A_{s+2n, s+t}v_{s+2n, s+t}.
\end{align*}
Hence by the definition of $v$ \eqref{controlv}, and the fact that $\|A_{2n}^{*}\widetilde{M}_{2n}^{-1}\|\leq \beta^{-1/2}$, we have that
\begin{align}\label{Rs+t-01}
\|\mathfrak{R}_{s+t}\|\leq \|J_{s+2n, s+t}\mathfrak{R}_{s+2n}\| + \|A_{s+2n, s+t}v_{s+2n, s+t}\|\leq C \beta^{-1/2}\left(1+\sup_{\tau\in[s+2n, s+t]}\|J_{\tau, s+t}\|^2\right)\|\mathfrak{R}_{s+2n}\|
\end{align}
And for $t\in [2n+1, 2n+2)$,  we have
\[\mathfrak{R}_{s+t} = J_{s+2n, s+t}J_{s, s+2n}\xi - A_{s, s+t}v_{s,s+ t} = J_{s+2n+1, s+2n+2}J_{s, s+2n+1}\xi -  A_{s,s+ t}v_{s, s+t}  = J_{s+2n+1, s+t}\mathfrak{R}_{s+2n+1}. \]
Note that $\|\mathfrak{R}_{s+2n+1}\| = \|\beta\widetilde{M}_{2n}^{-1}J_{2n}\mathfrak{R}_{s+2n}\|\leq\|J_{2n}\|\|\mathfrak{R}_{s+2n}\|$. Hence
\begin{align}\label{Rs+t-02}
\|\mathfrak{R}_{s+t} \|\leq \sup_{\tau\in[s+2n+1, s+t]}\|J_{\tau, s+t}\|\|\mathfrak{R}_{s+2n+1}\|\leq \sup_{\tau\in[s+2n, s+t]}\|J_{\tau, s+t}\|^2\|\mathfrak{R}_{s+2n+1}\|.
\end{align}
Combining the above inequalities \eqref{Rs+t-01} and \eqref{Rs+t-02} with  bounds on the enstrophy \eqref{enstrophy} and Jacobian \eqref{Jacobian} and inequality \eqref{error1}, one has
\[\left(\mathbf{E}\left\|\mathfrak{R}_{s+t}\right\|^{2}\right)^{1 / 2} \leq C \exp \left(p\eta\|w_0\|^{2}\right) e^{-at}\]
for some $p\in (0, 1)$ and  all $t\geq 0, \eta\in(0, \eta_0]$.
\end{proof}

The following result shows that the cost of the  variation $v$ on the Wiener path can be bounded. Since the proof is the same as that in \cite{HM06} once we obtain the estimate \eqref{error1}, we omit it here.
\begin{proposition}\label{control}
There are constants $p\in(0,1), \eta_0>0$ such that for any $\eta\in (0, \eta_0]$, there exists a constant $C = C(f, \mathcal{B}_0, \eta, \nu)$ so that for all $t\geq 0$,
\begin{align*}
\quad \mathbf{E}\left|\int_{s}^{s+t} v(r) d W(r)\right|^{2} \leq \frac{C}{\beta^{2}} e^{p\eta\left\|w_0\right\|^{2}} \sum_{n=0}^{\infty}\left(\mathbf{E}\left\|\mathfrak{R}_{s+2n}\right\|^{10}\right)^{\frac{1}{5}}.
\end{align*}
\end{proposition}
As a byproduct, we have the following asymptotic strong Feller property. Note that the constant $C$ is independent of the initial time $s$ compared with the aymptotic strong Feller property proposed in \cite{DD08}.
\begin{corollary}\label{ASF}
Under the same condition as in Proposition \ref{gradientinequalityprop}, with $t_{n} =  2n$ and $\delta_{n} = 2^{-n}$, we have some $\eta_0>0$, such that for $\eta\in (0, \eta_0]$, there is a constant $C = C(\eta) >0$ such that
\[\|\nabla \mathcal{P}_{s,s+t_n}\varphi(w)\|\leq C\exp(\eta\|w\|)\left(\|\varphi\|_{\infty}+\delta_{n}\|\nabla \varphi\|_{\infty}\right)\]
for all $\varphi\in C_b^1(H)$, $s\in \mathbb{R}$, $n\in\mathbf{N}$ and $w\in H$.
\end{corollary}\label{corollaryASF}
\begin{proof}
The inequality follows by \eqref{ineqASF}, estimate \eqref {error1} and Proposition \ref{control}.
\end{proof}

%

\subsection{The proof of Lemma \ref{contractlemma1}-\ref{contractlemma3}}\label{lemmasforcontraction}
In what follows we give a  proof for Lemma \ref{contractlemma1}-\ref{contractlemma3}.
\begin{proof}[Proof of Lemma \ref{contractlemma1}.]
By Lemma \ref{twopoint}, we know that there are constants $\alpha\in(0,1)$ and $K>0$ such that for any $w_1, w_2$ with $\rho(w_1,w_2)\geq L$, we have
\begin{align*}
\mathbf{E}\rho(\Phi_{s, s+n\mathcal{T}}(w_1), \Phi_{s, s+n\mathcal{T}}(w_2))\leq \alpha^n\rho(w_1,w_2) + K\leq (\alpha+K/L)\rho(w_1,w_2).
\end{align*}
Choosing $L$ large such that $\alpha_0: = \alpha+K/L\in (0, 1).$ Then
\begin{align*}
d\left(\mathcal{P}_{s, s+ n\mathcal{T}}^{*} \delta_{{w_1}}, \mathcal{P}_{s, s+ n\mathcal{T}}^{*} \delta_{{w_2}}\right) &\leq \mathbf{E}d(\Phi_{s, s+n\mathcal{T}}(w_1), \Phi_{s, s+n\mathcal{T}}(w_2))\\
&\leq 1+ \beta \mathbf{E}\rho(\Phi_{s, s+n\mathcal{T}}(w_1), \Phi_{s, s+n\mathcal{T}}(w_2))\leq 1+\alpha_0\beta\rho(w_1,w_2).
\end{align*}
Since $\rho_r(w_1,w_2)>\delta$, by definition of the metric $d$, one has $d(w_1, w_2) = 1+ \beta\rho(w_1, w_2)\geq 1+\beta L$. Therefore
\begin{align*}
1-\alpha_0 \leq (1-\alpha_0)\frac{d(w_1, w_2)}{1+\beta L} = \frac{1+ \alpha_0\beta L}{1+\beta L}d(w_1, w_2)-\alpha_0d(w_1, w_2)
\end{align*}
As a result
\begin{align*}
d\left(\mathcal{P}_{s, s+ n\mathcal{T}}^{*} \delta_{{w_1}}, \mathcal{P}_{s, s+ n\mathcal{T}}^{*} \delta_{{w_2}}\right) \leq 1+\alpha_0\beta\rho(w_1, w_2) = 1 - \alpha_0 + \alpha_0d(w_1, w_2) \leq \frac{1+ \alpha_0\beta L}{1+\beta L}d(w_1, w_2),
\end{align*}
where $\alpha_1: = \frac{1+ \alpha_0\beta L}{1+\beta L}\in (0,1)$.
\end{proof}

\begin{proof}[Proof of Lemma \ref{contractlemma2}.]
By the Monge-Kantorovich daulity \cite{Ch04,Vil08}, one has
\[d(\mu_1, \mu_2) = \sup_{\mathrm{Lip}_d(\phi)\leq 1}\left|\int_{H}\phi(w)\mu_1(dw)-\int_{H}\phi(w)\mu_2(dw)\right|.\]
Without loss of generality, in the above formula we could assume the test function $\phi\in C_b^1(H)$ and $\phi(0) =0$. Then $\mathrm{Lip}_d(\phi)\leq 1$ implies that $\|\nabla\phi(w)\|\leq (\delta^{-1}+\beta)V(w)$. Also by Proposition \ref{verifyLyapunov} we have
\[|\phi(w)|\leq 1+\beta\|w\|V(w)\leq 1+\beta C V^{\kappa}(w)\leq 1+\beta CV^{\kappa}(w).\]
Now combining  Proposition \ref{verifyLyapunov} and Proposition \ref{gradientinequalityprop}, one has
\begin{align*}
\left\|\nabla \mathcal{P}_{s, s+t} \phi(w)\right\| &\leq C(\eta, a) V^p(w)\left(\sqrt{\left(\mathcal{P}_{s, s+t}|\phi|^{2}\right)(w)}+ e^{-at} \sqrt{\left(\mathcal{P}_{s, s+t}\|\nabla \phi\|^{2}\right)(w)}\right)\\
&\leq  C(\eta, a)  V^p(w)\left[\Big(1+\beta^2 C^2\mathbf{E}V^{2\kappa}\big(\Phi_{s,s+t}(w)\big)\Big)^{\frac12} +e^{-at}(\delta^{-1}+\beta)\Big(\mathbf{E}V^2\big(\Phi_{s,s+t}(w)\big) \Big)^{\frac12}\right]\\
&\leq C(\eta, a) V^{\kappa\alpha(t)+p}(w)(1+e^{-at}\delta^{-1}) = \delta^{-1}V^{\kappa\alpha(t)+p}(w)(\delta C(\eta, a) + C(\eta, a)e^{-at}).
\end{align*}
For any $\alpha_2\in(0,1)$,  choose large $T_0>0$ so that $C(\eta, a)e^{-at}<\frac{\alpha_2}{2}$ for all $t\geq T_0$. From the formula for $\alpha(t)$ in Proposition \ref{verifyLyapunov}, we see that there is a large time $T>T_0$ such that for all $t\geq T$, one has $\kappa\alpha(t)+p<1$.  Choosing $r = \max\{r_0, \kappa\alpha(t)+p\}<1$ and letting $\delta$ be small such that $\delta C(\eta, a)<\frac{\alpha_2}{2}$, then we have
\[\left\|\nabla \mathcal{P}_{s, s+t} \phi(w)\right\| \leq \delta^{-1}V^{r}(w)\alpha_2. \]
Note that for any $w_1, w_2\in H$ and any $\varepsilon>0$, there is a differentiable path $\gamma: [0,1]\rightarrow H$ with $\gamma(0) = w_1$ and $\gamma(1) = w_2$ such that
\[\rho_r(w_1, w_2)\leq \int_{0}^{1}V^r(\gamma(\tau))\|\dot{\gamma}(\tau)\|d\tau \leq \rho_r(w_1, w_2) +\varepsilon.\]
Then
\begin{align*}
&\left|\mathcal{P}_{s, s+t}\phi(w_1)-\mathcal{P}_{s, s+t}\phi(w_2) \right| = \left|\int_{0}^{1}\left\langle\nabla \mathcal{P}_{s, s+t}\phi(\gamma(\tau)) , \dot{\gamma}(\tau)\right\rangle d\tau\right|\\
&\leq \delta^{-1}\alpha_2 \int_{0}^{1}V^r(\gamma(\tau)) \|\dot{\gamma}(\tau)\| d\tau \leq \delta^{-1}\alpha_2\rho_r(w_1, w_2) + \delta^{-1}\alpha_2\varepsilon \leq \alpha_2d(w_1, w_2) + \delta^{-1}\alpha_2\varepsilon,
\end{align*}
where in the last step we use the fact that $\rho_r(w_1, w_2)<\delta$ implies $d(x, y) = \delta^{-1}\rho_r(w_1, w_2) + \beta\rho(w_1, w_2)$. Since $\varepsilon>0$ is arbitrary, we have for any $w_1, w_2\in H$
\[\sup_{\mathrm{Lip}_d(\phi)\leq 1}\left|\mathcal{P}_{s, s+t}\phi(w_1)-\mathcal{P}_{s, s+t}\phi(w_2)\right|\leq \alpha_2 d(w_1, w_2).\]
Hence by the Monge-Kantorovich daulity,
\[d\left(\mathcal{P}_{s, s+n\mathcal{T}}^{*} \delta_{{w_1}}, \mathcal{P}_{s, s+n\mathcal{T}}^{*} \delta_{{w_2}}\right) \leq \alpha_{2} d({w_1}, {w_2}).\]
The proof is complete.
\end{proof}

\begin{proof}[Proof of Lemma \ref{contractlemma3}.]
For $L, \delta>0$, and $r\in[r_0, 1)$, Lemma 3.11 of \cite{HM08} shows that the set $S = \{(w_1,w_2): \rho_r(w_1, w_2)\geq \delta, \rho(w_1, w_2)<L\}$ is a bounded set in $H\times H$. So there exists $R = R(L,
\delta, r)>0$ such that $S\subset \{(w_1,w_2): \|w_1\|, \|w_2\|\leq R\}$.  By weak irreducibility from Proposition \ref{contractionirreducible}, we know there is $n_0 = n_0(L, \delta, r)$ such that for every $n\geq n_0$, there is positive constant $a>0$ so that for any $(w_1, w_2)\in S$, there is a coupling $(X_{s, s+n\mathcal{T}}, Y_{s, s+n\mathcal{T}})$ of
the transition probabilities $\mathcal{P}_{s, s+n\mathcal{T}}(w_1, \cdot)$ and $\mathcal{P}_{s, s+n\mathcal{T}}(w_2, \cdot)$, such that $\mathbf{P}\left(\rho_r(X_{s, s+n\mathcal{T}}, Y_{s, s+n\mathcal{T}})< \frac{\delta}{2}\right)>a>0$.
Note that there is a constant $C>0$ such that for any $w\in H$
\[\rho(w, 0)\leq \int_{0}^{1}V(\tau w)\|w\|d\tau \leq \|w\|V(\|w\|)\leq C V^{\kappa}(w).\]
Therefore
\begin{align*}
\mathbf{E}\rho(X_{s, s+n\mathcal{T}}, Y_{s, s+n\mathcal{T}})
&\leq \mathbf{E}\rho(X_{s, s+n\mathcal{T}}, 0) + \mathbf{E}\rho(0, Y_{s, s+n\mathcal{T}})\\
&\leq C\left(\mathbf{E}V^{\kappa}\left(X_{s, s+n\mathcal{T}}\right) + \mathbf{E} V^{\kappa}\left(Y_{s, s+n\mathcal{T}}\right)\right)\\
&= C\left(\mathbf{E}V^{\kappa}\left(\Phi_{s, s+n\mathcal{T}}(w_1)\right) + \mathbf{E} V^{\kappa}\left(\Phi_{s, s+n\mathcal{T}}(w_2)\right)\right)\\
&\leq C(V^{\kappa\alpha(n\mathcal{T})}(w_1)+ V^{\kappa\alpha(n\mathcal{T})}(w_2))\leq R_n,
\end{align*}
where $R_n = CV^{\kappa\alpha(n\mathcal{T})}(R)$. For given random variable $X$ and a measurable set  $A$, denote $\mathbf{E}(X; A) = \mathbf{E}X\mathbb{I}_{A}$.  Then
\begin{align*}
&\mathbf{E}d(X_{s, s+n\mathcal{T}}, Y_{s, s+n\mathcal{T}}) = \mathbf{E}\Big(1\wedge \frac{\rho_r(X_{s, s+n\mathcal{T}}, Y_{s, s+n\mathcal{T}})}{\delta}\Big) + \beta \mathbf{E}\rho(X_{s, s+n\mathcal{T}}, Y_{s, s+n\mathcal{T}})\\
&= \mathbf{E}\Big(1\wedge \frac{\rho_r(X_{s, s+n\mathcal{T}}, Y_{s, s+n\mathcal{T}})}{\delta}; \rho_r(X_{s, s+n\mathcal{T}}, Y_{s, s+n\mathcal{T}})<\frac{\delta}{2}\Big)\\
&+ \mathbf{E}\Big(1\wedge \frac{\rho_r(X_{s, s+n\mathcal{T}}, Y_{s, s+n\mathcal{T}})}{\delta}; \rho_r(X_{s, s+n\mathcal{T}}, Y_{s, s+n\mathcal{T}})\geq \frac{\delta}{2}\Big) + \beta \mathbf{E}\rho(X_{s, s+n\mathcal{T}}, Y_{s, s+n\mathcal{T}}). \\
&\leq \frac12 + \frac12 \mathbf{P}\left(\rho_r(X_{s, s+n\mathcal{T}}, Y_{s, s+n\mathcal{T}})\geq \frac{\delta}{2}\right)+ \beta R_n\leq \frac12 + \frac12(1-a) + \beta R_n  = 1 - \frac{a}{2} +  \beta R_n.
\end{align*}
Letting $\beta$ small enough so that $\alpha_3: = 1 - \frac{a}{2} +  \beta R_n<1$, then since $\rho_r(w_1, w_2)\geq \delta$ implies $d(w_1, w_2)\geq 1$, we have
\[d\left(\mathcal{P}_{s, s+n\mathcal{T}}^{*} \delta_{{w_1}}, \mathcal{P}_{s, s+n\mathcal{T}}^{*} \delta_{{w_2}}\right) \leq \mathbf{E}d(X_{s, s+n\mathcal{T}}, Y_{s, s+n\mathcal{T}})\leq \alpha_3 d(w_1, w_2), \]
which completes the proof.
\end{proof}


\end{document}